\documentclass[11pt, reqno]{amsart}
\usepackage{mathrsfs}
\usepackage{amssymb,amsthm,amsmath} 
\usepackage[numbers,sort&compress]{natbib}
\usepackage{amssymb,amsmath}
\usepackage{amsfonts}
\usepackage{mathrsfs}
\usepackage{latexsym}
\usepackage{amssymb}
\usepackage{amsthm}
\usepackage{xcolor}
\newif\ifshowrevisions
\newcommand{\rev}[1]{\ifshowrevisions{\color{blue}#1}\else#1\fi}
\newenvironment{revision}
{\begingroup\ifshowrevisions\color{blue}\fi}
{\endgroup}
\usepackage{pdfsync}
\usepackage{indentfirst}
\usepackage{graphics}
\usepackage{subfigure}
\usepackage{epsfig}
\usepackage{float}
\usepackage{cases}
\date{}
\allowdisplaybreaks

\usepackage{hyperref}
\hypersetup{hidelinks}

\usepackage{amsmath}
\usepackage{amsthm}

\newtheorem{remark}{Remark}[section]

\newtheorem{theorem}{Theorem}[section]
\newtheorem{proposition}{Proposition}[section]
\newtheorem{lemma}{Lemma}[section]
\newtheorem{corollary}{Corollary}[section]

\newcommand{\beq}{\begin{equation}}
	\newcommand{\eeq}{\end{equation}}
\newcommand{\ben}{\begin{eqnarray}}
	\newcommand{\een}{\end{eqnarray}}
\newcommand{\beno}{\begin{eqnarray*}}
	\newcommand{\eeno}{\end{eqnarray*}}
\newcommand{\alabel}{\stepcounter{equation}\tag{\theequation}\label}

\numberwithin{equation}{section}

\def\n{\nabla}

\begin{document}
	\title[Suppression of blow-up for PKS--NS system]{Blow-up suppression for  the 2D parabolic-parabolic Patlak-Keller-Segel-Navier-Stokes system}

	\author{Xiaomeng~Chen}
	\address[Xiaomeng~Chen]{School of Mathematical Sciences, Dalian University of Technology, Dalian, 116024,  China}
	\email{cxm@mail.dlut.edu.cn}
	
	\author{Shikun~Cui}
	\address[Shikun~Cui]{State Key Laboratory of Mathematical Sciences,
		Academy of Mathematics $\&$ Systems Science, The Chinese Academy of Sciences,
		Beijing,
		100190, China}
	\email{cskmath@163.com}
	
	\author{Wendong~Wang}
	\address[Wendong~Wang]{School of Mathematical Sciences, Dalian University of Technology, Dalian, 116024,  China}
	\email{wendong@dlut.edu.cn}

	\author{Yi~Zhang}
	\address[Yi~Zhang]{School of Mathematical Sciences, Dalian University of Technology, Dalian, 116024,  China}
	\email{zymath@mail.dlut.edu.cn}
	
	\date{\today}
	
	\begin{abstract} Chemotactic aggregation can overcome diffusion and lead to finite-time blow-up in Patlak--Keller--Segel models. Mixing induced by suitable fluid flows can counteract this aggregation through enhanced dissipation. Earlier results of  Zeng–Zhang–Zi (J. Funct. Anal., 2021) for the coupled system impose additional smallness assumptions on the initial chemical gradient. A natural question is whether this restriction can be removed for the coupled system near a stationary Couette flow. We give a positive answer to this question
		for the two-dimensional parabolic–parabolic Patlak–Keller–Segel–Navier–Stokes system near the Couette flow $(Ay,0)$ in $\mathbb T\times\mathbb R$. It is proved that, when the Couette flow is strong enough ($A$ is large enough), the system admits a global solution that remains uniformly bounded in time for arbitrary fixed initial data. 
		The main difficulties are the transient growth of the chemical gradient and the nonlinear coupling among the density, chemical field and fluid.
		We develop a two-stage energy method combining short-time estimates in Kelvin coordinates with long-time estimates in physical coordinates. 
		The key is to split the nonzero chemical modes into two parts while accounting for the nonlinear transport. One part carries the initial data, and the other is a forced component with zero initial data.
		A weighted coupled energy controls the former through a time-integrable Gronwall coefficient and absorbs the feedback from the latter without increasing the initial energy. Enhanced dissipation then provides the bounds needed to enter the long-time regime, while one-dimensional zero-mode estimates and Moser iteration complete the argument without a mass restriction.
		
	\end{abstract}

	\maketitle
	


	
	\tableofcontents

	\parskip5pt
	\parindent=1.5em
	\section{Introduction}
	
	Consider the following parabolic--parabolic Patlak--Keller--Segel (PKS) system
	\begin{equation}\label{eq:ks}
		\left\{
		\begin{array}{lr}
			\partial_tn=\Delta n-\nabla\cdot(n\nabla c), \\
			\partial_tc=\Delta c-c+n,
		\end{array}
		\right.
	\end{equation}
	where $n$ and $c$ denote the cell density and the chemical concentration, respectively. The PKS model, originating in the works of Patlak~\cite{Patlak1953} and Keller--Segel~\cite{KellerSegel1970}, describes the competition between diffusion and chemotactic aggregation. When aggregation dominates diffusion, solutions may develop singularities. For the planar parabolic--parabolic problem, Calvez--Corrias~\cite{CalvezCorrias2008} established global existence below the $8\pi$ mass threshold under their assumptions, while Schweyer~\cite{Schweyer2014} constructed radially symmetric finite-time blow-up solutions. These results motivate the search for mechanisms that prevent chemotactic singularities without imposing a smallness condition on the initial cell mass.
	
	As said in \cite{KiselevXu2016} or \cite{ZengZhangZi2021}, chemotactic motion often takes place in a moving fluid. A natural question is whether fluid transport can suppress the concentration driven by the aggregation term $\nabla\cdot(n\nabla c)$. In this paper, we investigate the two-dimensional fully parabolic PKS system coupled with the incompressible Navier--Stokes (NS) equations in $\mathbb T\times\mathbb R$, where $\mathbb T=[0,2\pi]$:
	\begin{equation}\label{ini}
		\left\{
		\begin{array}{lr}
			\partial_tn+v\cdot\nabla n=\Delta n-\nabla\cdot(n\nabla c), \\
			\partial_tc+v\cdot\nabla c+c=\Delta c+n, \\
			\partial_tv+v\cdot\nabla v+\nabla P=\Delta v+n\nabla \phi, \\
			\nabla\cdot v=0,
		\end{array}
		\right.
	\end{equation}
	with the initial conditions
	\[
	(n,c,v)\big|_{t=0}=(n_{\rm in},c_{\rm in},v_{\rm in}),
	\]
	where $v$ is the fluid velocity, $P$ is the pressure, and $\phi$ is a prescribed potential. Throughout this paper, we take $\phi=y$ and denote the conserved cell mass by
	\[
	M:=\|n_{\rm in}\|_{L^1(\mathbb T\times\mathbb R)}
	=\|n(t)\|_{L^1(\mathbb T\times\mathbb R)}.
	\]
	Unlike PKS equations with a prescribed ambient flow, \eqref{ini} includes the feedback of the cell density on the fluid. The resulting velocity transports both $n$ and $c$, so the estimates for the density, chemical concentration and velocity must be closed together. Let us briefly recall some results related to the present problem. 
	
	{\bf The PKS equations of parabolic-elliptic with prescribed transport.} Kiselev--Xu~\cite{KiselevXu2016} proved that sufficiently strong mixing flows suppress blow-up in the parabolic--elliptic setting. Bedrossian--He~\cite{BedrossianHe2017} established blow-up suppression via shear flows, whose mixing effect reduces the effective dimension of the dynamics. In the three-dimensional parabolic--elliptic setting, Feng--Shi--Wang~\cite{FengShiWang2022} established blow-up
	suppression on $\mathbb T^3$ via planar helical flows, while Shi--Wang~\cite{ShiWang2024} obtained a similar result in $\mathbb T^2\times\mathbb R$ for the plane Couette--Poiseuille flow. Both results allow arbitrary initial mass provided that the flow is sufficiently strong. The underlying mechanism is enhanced dissipation: advection creates small spatial scales on which diffusion acts more efficiently. This mechanism also plays an important role in hydrodynamic stability; see, for example, Bedrossian--Masmoudi--Vicol~\cite{BedrossianMasmoudiVicol2016}, Bedrossian--Vicol--Wang~\cite{BedrossianVicolWang2018}, and Wei--Zhang~\cite{WeiZhang2021,WeiZhang2023}.
	
	{\bf An additional difficulty arises in the parabolic--parabolic case.} Since the chemical concentration satisfies its own evolution equation, a strong shear flow may amplify the chemical gradient before enhanced dissipation has sufficiently reduced the nonzero modes. This gradient enters the aggregation flux, so mixing can have a destabilizing effect as well as a stabilizing one. He~\cite{He2018} studied strictly monotone shear flows in $\mathbb T\times\mathbb R$ under a smallness assumption on the initial chemical gradient. In $\mathbb R^3$, Deng--Shi--Wang~\cite{DengShiWang2025} established blow-up suppression via Couette flow. Their Green's function approach covers both parabolic--elliptic and parabolic--parabolic cases under suitable spatial decay assumptions. The latter case additionally requires a smallness condition on the initial chemical gradient. More recently, He \cite{He2025Alternating,He2025Shear} obtained large-data results on $\mathbb T^3$ by using time-dependent flows. In particular, under a restriction on the total mass, the shear-flow result in~\cite{He2025Shear} combines short-time estimates with a smallness argument to obtain long-time bounds.
	
	{\bf The coupled PKS--NS system.} Zeng--Zhang--Zi~\cite{ZengZhangZi2021} proved blow-up suppression near Couette flow in $\mathbb T\times\mathbb R$ in both the parabolic--elliptic and parabolic--parabolic cases. In the latter case, they imposed the additional anisotropic smallness condition
	\begin{align*}
		A^{\frac12}\|\partial_x(c_{\rm in})_{\neq}\|_{L^2}
		+A^{\frac16}\|\partial_y(c_{\rm in})_{\neq}\|_{L^2}
		\leq1,\alabel{small1}
	\end{align*}
	where $(c_{\rm in})_{\neq}$ denotes the nonzero Fourier modes in $x$. This condition compensates for the shear-induced transient growth of the chemical gradient. However, for fixed initial data with a nonzero chemical gradient in these modes, the condition becomes more restrictive as $A$ increases.
	For the coupled system in the whole space, Wang--Wang--Wu \cite{WangWangWu2026} studied blow-up suppression for the two-dimensional PKS--NS system near the Couette flow using Green's function estimates.
	
	{\bf Related results have been obtained for other flows and coupling mechanisms.} He~\cite{He2023} studied enhanced dissipation and blow-up suppression in a chemotaxis--fluid system with buoyancy coupling. For the Poiseuille flow, Li--Xiang--Xu~\cite{LiXiangXu2025} considered the problem in $\mathbb T\times\mathbb R$, while Cui--Wang~\cite{CuiWang2024} studied a multispecies system in a finite channel. Suppression by strong buoyancy was investigated by Hu~\cite{Hu2024} and Hu--Kiselev--Yao~\cite{HuKiselevYao2025}. In three dimensions, Cui--Wang--Wang~\cite{CuiWangWang2026} established blow-up suppression for the parabolic--elliptic PKS--NS system near the Couette flow under a mass restriction and suitable assumptions on the perturbation. These results reflect the different roles of the chemical equation, the background flow and the fluid coupling.
	
	The purpose of this paper is to remove the chemical-gradient smallness condition \eqref{small1} for the fully parabolic system \eqref{ini} by Zeng--Zhang--Zi~\cite{ZengZhangZi2021}, with a stationary Couette background and the full Navier--Stokes coupling. More precisely, we address the following question:
	\begin{quote}
		\itshape Can a sufficiently strong stationary Couette flow suppress blow-up in the fully parabolic PKS--NS system for arbitrary fixed, sufficiently regular initial density, chemical concentration and velocity perturbation?
	\end{quote}
	{We give a positive answer by combining short-time estimates in Kelvin coordinates with long-time estimates in physical coordinates. The contribution of the large initial chemical profile is controlled through a time-integrable coefficient in a coupled energy inequality, rather than absorbed as a small perturbation. By the end of the preparation interval, enhanced dissipation has compensated for the shear-induced transient growth and provides the bounds needed for the long-time argument.}
	
	{To state our main result, we introduce the velocity perturbation around the Couette flow,}
	\[
	u(t,x,y)=v(t,x,y)-U(y),\qquad U(y)=(Ay,0).
	\]
	Then \eqref{ini} becomes
	\begin{equation*}
		\begin{cases}
			\partial_t n + Ay\partial_x n + {u}\cdot\nabla n - \Delta n
			= -\nabla\cdot(n\nabla c),\\
			\partial_t c + Ay\partial_x c + {u}\cdot\nabla c-\Delta c
			= n-c,\\
			\partial_t {u} + Ay\partial_x {u}
			+ \begin{pmatrix} Au_2 \\ 0 \end{pmatrix}
			+ {u}\cdot\nabla {u}
			- \Delta {u} + \nabla p
			= \begin{pmatrix} 0 \\ n \end{pmatrix},\\
			\nabla\cdot{u} =0.
		\end{cases}
	\end{equation*}
	Define the vorticity and stream function by
	\[
	\omega=\partial_yu_1-\partial_xu_2,\qquad
	u=\nabla^\perp\Phi=(\partial_y\Phi,-\partial_x\Phi),\qquad
	\Delta\Phi=\omega.
	\]
	The vorticity satisfies
	\[
	\partial_t\omega+Ay\partial_x\omega-\Delta\omega+u\cdot\nabla\omega
	=-\partial_xn.
	\]
	{Rescaling time by replacing the original time variable with $\frac{t}{A}$, we obtain}
	\begin{equation}\label{eq:1.7}
		\begin{cases}
			\partial_t n+y\partial_x n-\frac{1}{A}\Delta n
			=-\frac{1}{A}(\nabla\cdot(n\nabla c)+{u}\cdot\nabla n),\\
			\partial_t c+y\partial_x c+\frac{1}{A}{u}\cdot\nabla c
			-\frac{1}{A}\Delta c
			=\frac{1}{A}(n-c),\\
			\partial_t\omega+y\partial_x\omega-\frac{1}{A}\Delta\omega
			=-\frac{1}{A}(\partial_x n+{u}\cdot\nabla\omega),\\
			{u}=\nabla^\perp\Delta^{-1}\omega.
		\end{cases}
	\end{equation}
	{For any function $f=f(t,x,y)$, we define its zero and nonzero modes in the $x$-direction by}
	\[
	P_0f=f_0(t,y)
	=
	\frac{1}{2\pi}\int_{\mathbb T}f(t,x,y)\,dx,
	\qquad
	P_{\neq}f=f_{\neq}:=f-f_0.
	\]
	The main result is stated as follows. 
	\begin{theorem}\label{thm:main}
		{Assume that $0\leq n_{\rm in}\in L^1\cap H^2(\mathbb T\times\mathbb R)$, $0\leq c_{\rm in}\in H^2\cap W^{1,\infty}(\mathbb T\times\mathbb R)$, and $u_{\rm in}\in H^3(\mathbb T\times\mathbb R)$ with $\nabla\cdot u_{\rm in}=0$. Then there exists a positive constant $\bar A$, depending only on $M=\|n_{\rm in}\|_{L^1(\mathbb T\times\mathbb R)}$, $\|n_{\rm in}\|_{H^2}$, $\|c_{\rm in}\|_{H^2}$ and $\|u_{\rm in}\|_{H^3}$, such that if $A\geq\bar A$, \rev{the unique local strong solution of \eqref{eq:1.7} extends globally in time and satisfies}}
		\begin{align*}
			\|n\|_{L^{\infty}L^{\infty}}+\|c\|_{L^{\infty}W^{1,\infty}}+\|u\|_{L^{\infty}L^{\infty}}
			\leq C\left(A,M,\|n_{\rm in}\|_{H^2},\|c_{\rm in}\|_{H^2},\|c_{\rm in}\|_{W^{1,\infty}},\|u_{\rm in}\|_{H^3}\right).
		\end{align*}
	\end{theorem}
	\begin{remark}
		The local strong solution of \eqref{eq:1.7} has the regularity stated in Lemma~\ref{lem:local-continuation}.
		{The initial density, chemical concentration and velocity perturbation may be arbitrarily large but fixed, with $\bar A$ chosen in terms of the norms specified in Theorem~\ref{thm:main}. In particular, neither a mass restriction nor the chemical-gradient smallness condition \eqref{small1} in~\cite{ZengZhangZi2021} is required. 
			We use the higher Sobolev regularity stated in Theorem~\ref{thm:main} without seeking to optimize the regularity assumptions on the initial data.
		}
\end{remark}

\begin{remark}
	{The threshold $\bar A$ is independent of $\|c_{\rm in}\|_{W^{1,\infty}}$, which enters only the final pointwise bound. The weighted energy estimates used to determine $\bar A$ are uniform in $A$, whereas the final uniform-in-time bounds may depend on $A$. Also, the bounded velocity in the theorem is the perturbation $u=v-(Ay,0)$, not the total velocity $v$.}
\end{remark}

\begin{remark}
	In the present two-dimensional setting, the $x$-average depends only on $y$. One-dimensional heat-kernel and Nash estimates therefore allow us to control the density zero mode without a mass threshold. By contrast, the shear-flow problem on $\mathbb T^3$ in~\cite{He2025Shear} has a two-dimensional average and retains a mass restriction. Moreover, our background flow is stationary, and the velocity perturbation evolves through the coupled Navier--Stokes equations. The prescribed-flow results therefore do not directly yield the present theorem.
\end{remark}
\noindent\textbf{Notations}:
\begin{itemize}
	\item {For a function $f=f(t,x,y)$, we use the Fourier series} 
	\begin{equation*}
		\begin{aligned}
			f(t,x,y)=\sum_{k\in\mathbb Z}\widehat f_k(t,y)e^{ikx}, \quad
			\widehat f_k(t,y)
			=\frac{1}{|\mathbb T|}\int_{\mathbb T}f(t,x,y)e^{-ikx}dx,
		\end{aligned}
	\end{equation*}
	and 
	\begin{align*}
		f(t,x,y)
		=\frac{1}{|\mathbb{T}|}\sum_{k\in\mathbb Z}\int_{\mathbb R}\widehat f_k(t,\eta)e^{i(kx+\eta y)} d\eta,\quad 
		\widehat f_k(t,\eta)
		=\frac{1}{|\mathbb T|}\int_{\mathbb T\times\mathbb R}
		f(t,x,y)e^{-i(kx+\eta y)}dxdy.
	\end{align*}

	\item Unless otherwise stated, $C$ denotes a positive constant independent of $A$, time and the initial data, and may vary from line to line.
	The constant $C$ may depend on the fixed parameters $a$ and $\delta$.

	\item Sobolev norms in Kelvin coordinates are taken with respect to $(X,Y)$; those in physical coordinates are taken with respect to $(x,y)$. The time interval in a norm such as $L^qL^p$ is specified in the corresponding section. \rev{ Unless otherwise stated, spatial norms of zero modes are taken over $\mathbb R$.
	}
\end{itemize}
\noindent\textbf{Outline of the paper.}
{The remainder of the paper is organized as follows. Section~\ref{sec:key-ideas} presents the main ideas, the energy functionals and the bootstrap argument. Section~\ref{sec:preliminaries} collects the space-time and anisotropic Sobolev estimates. Sections~\ref{sec:short-time} and~\ref{sec:long-time} establish the short-time estimates in Kelvin coordinates and the long-time estimates in physical coordinates, respectively. The uniform-in-time pointwise bounds are proved in Section~\ref{sec:uniform-bounds}. \rev{Appendix~\ref{sec_A} contains the Moser--Alikakos iteration lemma. Appendix \ref{sec:local-continuation} establishes local well-posedness and the blow-up criterion used in the proof of Theorem \ref{thm:main}.}}
\section{Key ideas and proof strategy}\label{sec:key-ideas}
\subsection{Main difficulties and new ideas}
\label{subsec:main difficulties}
The main difficulty is to control the transient growth of the chemical gradient before Couette enhanced dissipation becomes effective. We first examine the zero and nonzero modes equations and use a simplified model to identify this difficulty and the relevant time scales. These observations guide the construction of our energy functionals.

We begin by decomposing \eqref{eq:1.7} into two coupled subsystems. Applying $P_0$ to \eqref{eq:1.7}, we obtain
\begin{equation}\label{eq:zero-mode-system}
	\begin{cases}
		\partial_t n_0-\frac1A\partial_y^2n_0
		=-\frac1A\partial_y\big[n_0\partial_yc_0
		+(n_{\neq}\partial_yc_{\neq})_0
		+(u_{2,\neq}n_{\neq})_0\big],\\
		\partial_t c_0-\frac1A\partial_y^2c_0
		=\frac1A(n_0-c_0)
		-\frac1A\partial_y(u_{2,\neq}c_{\neq})_0,\\
		\partial_t\omega_0-\frac1A\partial_y^2\omega_0
		=-\frac1A\partial_y(u_{2,\neq}\omega_{\neq})_0,\\ 
		u_0=(u_{1,0},0), \quad 
		\partial_yu_{1,0}=\omega_0.
	\end{cases}
\end{equation}
For the nonzero modes, let $\mathcal L_A=\partial_t+y\partial_x-\frac1A\Delta$. Applying $P_{\neq}$ to \eqref{eq:1.7} yields
\begin{equation}\label{eq:non1}
	\begin{cases}
		\mathcal L_A n_{\neq}
		=-\frac1A\nabla\cdot\big[n_0\nabla c_{\neq}
		+n_{\neq}\nabla c_0
		+(n_{\neq}\nabla c_{\neq})_{\neq}\big]\\
		\qquad\qquad-\frac1A \nabla\cdot
		\big[u_{\neq}n_{0} 
		+u_{0}n_{\neq} 
		+(u_{\neq} n_{\neq})_{\neq}\big],\\
		\mathcal L_Ac_{\neq}
		=\frac1A(n_{\neq}-c_{\neq})
		-\frac1A\big[u_{1,0}\partial_xc_{\neq}
		+u_{2,\neq}\partial_yc_0
		+(u_{\neq}\cdot\nabla c_{\neq})_{\neq}\big],\\
		\mathcal L_A\omega_{\neq}
		=-\frac1A\big[\partial_xn_{\neq}
		+u_{1,0}\partial_x\omega_{\neq}
		+u_{2,\neq}\partial_y\omega_0
		+(u_{\neq}\cdot\nabla\omega_{\neq})_{\neq}\big],\\
		u_{\neq}
		=\nabla^\perp\Delta^{-1}\omega_{\neq}.
	\end{cases}
\end{equation}
To isolate the effect of the initial chemical profile, we construct a simplified toy model, retaining the Couette transport, diffusion and chemical production in \eqref{eq:non1}:
\begin{equation*}
	\begin{cases}
		\mathcal L_A n_{\neq}=0,\\
		\left(\mathcal L_A+\frac1A\right)c_{\neq}
		=\frac1A n_{\neq},\\
		n_{\neq}|_{t=0}=(n_{\rm in})_{\neq},\quad
		c_{\neq}|_{t=0}=(c_{\rm in})_{\neq}.
	\end{cases}
\end{equation*}
We separate the contribution of the initial chemical profile from the response to the density by writing $c_{\neq}=c_{1,\neq}+c_{2,\neq}$, where
\begin{equation*}\label{nc1}
	\begin{cases}
		\left(\mathcal L_A+\frac1A\right)c_{1,\neq}=0,\\
		\left(\mathcal L_A+\frac1A\right)c_{2,\neq}
		=\frac1A n_{\neq},\\
		c_{1,\neq}|_{t=0}=(c_{\rm in})_{\neq},\quad
		c_{2,\neq}|_{t=0}=0,
	\end{cases}
\end{equation*}
where $c_{1,\neq}$ carries the initial chemical data, while $c_{2,\neq}$ has zero initial data and is driven by forcing $\frac{n_{\neq}}{A}$.

Since both equations have the same transport and diffusion operators, we have the following exact identity
\begin{equation}\label{eq:toy-identity}
	c_{2,\neq}(t)
	=\big(1-e^{-\frac{t}{A}}\big)n_{\neq}(t).
\end{equation}
The Fourier representation gives
\begin{align*}
	\|n_{\neq}(t)\|_{L^2}
	&\leq e^{-\frac{t}{A}-\frac{t^3}{12A}}
	\|(n_{\rm in})_{\neq}\|_{L^2},\\
	\|\partial_x c_{1,\neq}(t)\|_{L^2}
	&\leq e^{-\frac{2t}{A}-\frac{t^3}{12A}}
	\|(\partial_xc_{\rm in})_{\neq}\|_{L^2},\\
	\|\partial_y c_{1,\neq}(t)\|_{L^2}
	&\leq e^{-\frac{2t}{A}-\frac{t^3}{12A}}
	\left(
	\|(\partial_yc_{\rm in})_{\neq}\|_{L^2}
	+t\|(\partial_xc_{\rm in})_{\neq}\|_{L^2}
	\right).
\end{align*}
The factor $e^{-\frac{t}{A}}$ in the density estimate comes from ordinary diffusion in the nonzero $x$-modes, while $c_{1,\neq}$ has the additional chemical damping $e^{-\frac{t}{A}}$.
\begin{revision}
	Moreover, the same Fourier representation yields
\end{revision}
\[
\|\nabla n_{\neq}(t)\|_{L^2}
\leq C\left(\frac At\right)^{\frac12}
e^{-\frac{t^3}{24A}}\|(n_{\rm in})_{\neq}\|_{L^2},
\qquad t>0.
\]
Using \eqref{eq:toy-identity} and $1-e^{-\frac{t}{A}}\leq \frac{t}{A}$, we obtain
\[
\begin{aligned}
	\|\nabla c_{2,\neq}(t)\|_{L^2}
	=\bigl(1-e^{-\frac{t}{A}}\bigr)\|\nabla n_{\neq}(t)\|_{L^2}
	\leq C\left(\frac tA\right)^{\frac12}
	e^{-\frac{t^3}{24A}}\|(n_{\rm in})_{\neq}\|_{L^2}.
\end{aligned}
\]
Taking the supremum in time and using $t=A^{\frac13}s$ in the scalar factors, we conclude that, for $A\geq1$,
\begin{align*}
	\sup_{t\geq0}\|\nabla c_{1,\neq}(t)\|_{L^2}
	&\leq C(\|(\nabla c_{\rm in})_{\neq}\|_{L^2}+
	A^{\frac13}
	\|(\partial_x c_{\rm in})_{\neq}\|_{L^2}),\\
	\sup_{t\geq0}\|\nabla c_{2,\neq}(t)\|_{L^2}
	&\leq CA^{-\frac13}
	\|(n_{\rm in})_{\neq}\|_{L^2}.
\end{align*}
Thus, in this simplified model, the initial chemical profile may produce a gradient of order $A^{\frac13}$, whereas the density-induced response is bounded by $CA^{-\frac13}\|(n_{\rm in})_{\neq}\|_{L^2}$. This distinction motivates our quasi-linear decomposition of the full system. The identity \eqref{eq:toy-identity} is specific to the simplified model; the nonlinear chemical components must instead be estimated through their coupled equations.

\begin{itemize}
	\item \textbf{Characteristic times.} In the rescaled time variable, set
	\[
	t_0=0,\qquad t_1=A^{\frac13},\qquad t_2=A.
	\]
	Here $t_0$ is the initial time, $t_1$ is the enhanced-dissipation time scale, and $t_2$ is the ordinary-diffusion time scale at fixed spatial scales. The quantities $A^{-\frac13}$ and $A^{-1}$ are the corresponding rates, not the time scales.
	
	\item \textbf{Time intervals.} Choose
	\[
	T^*=A^{\frac13+\delta},\qquad 0<\delta<\frac23.
	\]
	For $A>1$, this gives the four successive intervals
	\begin{equation*}\label{interbal}
		[t_0,t_1],\qquad [t_1,T^*],\qquad
		[T^*,t_2],\qquad [t_2,\infty).
	\end{equation*}
	
	\item \textbf{Dynamics and estimates.} These intervals distinguish the following mechanisms:
	\begin{equation*}
		\begin{cases}
			[t_0,t_1]: & \text{possible transient amplification of the chemical gradient},\\
			[t_1,T^*]: & \text{enhanced dissipation offsets the transient growth},\\
			[T^*,t_2]: & \text{small nonzero modes coupled to bounded zero modes},\\
			[t_2,\infty): & \text{continued control of the slowly evolving zero modes}.
		\end{cases}
	\end{equation*}
	This is a description of the relevant mechanisms, not an assertion of monotone decay on each interval. In particular, the late-time solution need not be exactly independent of $x$, and no exponential decay of the zero modes is assumed.
\end{itemize}

Therefore, the main difficulty is to control the solution on the preparation interval $[0,T^*]$. The equation for $\partial_yc$ contains the source $-\partial_xc$ without an $A^{-1}$ factor, so the large initial chemical gradient cannot be treated directly as a small perturbation. Our choice of $T^*$ satisfies
\[
\frac{T^*}{A^{\frac13}}=A^\delta\longrightarrow\infty,
\qquad
\frac{T^*}{A}=A^{-\frac23+\delta}\longrightarrow0
\qquad\text{as }A\to\infty.
\]
The first relation allows enhanced dissipation to produce the small factor $e^{-aA^\delta}$ at $T^*$, once the short-time estimates have been closed. The second makes the accumulation of the slow terms small after their coefficients have been controlled. The main contribution of the large initial chemical data is handled instead by a time-integrable Gronwall coefficient.

Moreover, the proof is divided into short-time and long-time estimates. On $[0,T^*]$, we work in Kelvin coordinates and combine the quasi-linear decomposition with a coupled energy estimate. This energy includes the density zero mode, since the zero and nonzero modes interact and the density also feeds back into the fluid. At $T^*$, the controlled nonzero modes are exponentially small, while the zero modes remain bounded. On $[T^*,\infty)$, we return to physical coordinates and use these prepared data to close the long-time energy estimates. Finally, Moser iteration yields the uniform-in-time pointwise bounds.

\subsection{Construction of the energy functionals}
\label{subsec:energy-construction}

We use different energy functionals in the two time regimes. In Kelvin coordinates, the short-time energy controls the large initial data without requiring an $A$-uniform bound for the physical transverse gradient. After the preparation interval, the physical-coordinate energy provides the space-time bounds needed to close the long-time estimates.

\noindent$\bullet$\quad
\textbf{Short-time energy functionals in Kelvin coordinates.}

We introduce the Kelvin coordinates
\begin{equation*}
	X=x-ty,\quad Y=y,
\end{equation*}
for a function $f=f(t,x,y)$, its Kelvin representation is defined by
\begin{equation*}
	F(t,X,Y):=f(t,X+tY,Y).
\end{equation*}
The chain rule yields
\begin{equation}\label{relation}
	(\partial_t+y\partial_x)f=\partial_tF,
	\quad
	\partial_xf=\partial_XF,
	\quad
	\partial_yf=(\partial_Y-t\partial_X)F.
\end{equation}
In these coordinates, \eqref{eq:zero-mode-system} and \eqref{eq:non1} become
\begin{equation}\label{eq:zero-mode-system-kelvin}
	\begin{cases}
		\partial_tN_0-\frac1A\partial_Y^2N_0
		=-\frac1A\partial_Y\big[N_0\partial_YC_0
		+(N_{\neq}(\partial_Y-t\partial_X)C_{\neq})_0
		+(U_{2,\neq}N_{\neq})_0\big],\\
		\partial_tC_0-\frac1A\partial_Y^2C_0
		=\frac1A(N_0-C_0)-\frac1A\partial_Y(U_{2,\neq}C_{\neq})_0,\\
		\partial_t\Omega_0-\frac1A\partial_Y^2\Omega_0
		=-\frac1A\partial_Y\big(U_{2,\neq}\Omega_{\neq}\big)_0,\\
		U_0=(U_{1,0},0),\quad \partial_YU_{1,0}=\Omega_0,
	\end{cases}
\end{equation}
and
\begin{equation}\label{eq:nonzero-mode-system-kelvin}
	\begin{cases}
		\left(\partial_t-\frac1A\Delta_L\right)N_{\neq}
		=-\frac1A\nabla_L\cdot\big[N_0\nabla_LC_{\neq}+N_{\neq}\nabla_LC_0
		+(N_{\neq}\nabla_LC_{\neq})_{\neq}\big]\\
		\qquad\qquad\qquad
		-\frac1A\nabla_L\cdot\big[U_{\neq}N_0+U_0N_{\neq}+(U_{\neq}N_{\neq})_{\neq}\big],\\
		\left(\partial_t-\frac1A\Delta_L\right)C_{\neq}
		=\frac1A(N_{\neq}-C_{\neq})-\frac1A\big[U_{1,0}\partial_XC_{\neq}
		+U_{\neq}\cdot\nabla_L C_0
		+(U_{\neq}\cdot\nabla_LC_{\neq})_{\neq}\big],\\ \left(\partial_t-\frac1A\Delta_L\right)\Omega_{\neq}
		=-\frac1A\partial_XN_{\neq}-\frac1A \nabla_L\cdot\big[U_{0}\Omega_{\neq}
		+U_{\neq}\Omega_0
		+(U_{\neq}\Omega_{\neq})_{\neq}\big],\\
		U_{\neq}=\nabla_L^\perp\Delta_L^{-1}\Omega_{\neq},
	\end{cases}
\end{equation}
where 
\begin{equation*}
	\nabla_L=(\partial_X,\partial_Y-t\partial_X)^T, \quad 
	\Delta_L=\partial_X^2+(\partial_Y-t\partial_X)^2.
\end{equation*}

The Kelvin transformation removes the Couette transport term, while the growth of the physical transverse derivative is retained in $\partial_Y-t\partial_X$. Accordingly, the $H^1$ norms below are defined using the ordinary derivatives $\partial_X$ and $\partial_Y$, and the dissipative terms involve $\nabla_L$. We use a time-dependent multiplier to capture both enhanced dissipation and the space-time control of $\partial_X\nabla_L\Delta_L^{-1}$. This gives a Kelvin-coordinate formulation of the Couette space-time framework in~\cite{WeiZhang2021,ZengZhangZi2021}.

Inspired by previous works, we employ the Fourier multiplier method to establish short-time estimates for the nonzero modes \cite{BGM2017,LZ2024}. For $A\geq1$, $a\geq0$, $k\in\mathbb Z\setminus\{0\}$, and $\eta\in\mathbb R$, define
\begin{equation*}
	\begin{gathered}
		M_k^{(1)}(t,\eta)
		=\arctan\left(
		A^{-\frac13}\frac{\eta-kt}{k}
		\right), \quad
		M_k^{(2)}(t,\eta)
		:=\arctan\left(
		\frac{\eta-kt}{k}
		\right),\\
		M_k(t,\eta)
		=M_k^{(1)}(t,\eta)+M_k^{(2)}(t,\eta)+\pi+1, \quad 
		\mathcal M_k(t,\eta)
		=\exp\left(aA^{-\frac13}t+M_k(t,\eta)-1\right).
	\end{gathered}
\end{equation*}
Since
\[
-\frac{\pi}{2}<M_k^{(j)}(t,\eta)<\frac{\pi}{2},
\qquad j=1,2,
\]
we have
\begin{equation*}
	1\leq M_k(t,\eta)\leq 2\pi+1.
\end{equation*}
Consequently,
\begin{equation}\label{eq:weighted-M-bounded}
	e^{aA^{-\frac13}t}
	\leq\mathcal M_k(t,\eta)
	\leq e^{2\pi}e^{aA^{-\frac13}t}.
\end{equation}
For a function $F_{\neq}$ with zero $X$-average, we define
the time-dependent Fourier multiplier $\mathcal M(t,D)$ by
\[
\mathcal M(t,D)F_{\neq}(t,X,Y)
=\frac{1}{2\pi}\sum_{k\neq0}\int_{\mathbb R}\mathcal M_k(t,\eta)\widehat F_k(t,\eta)e^{i(kX+\eta Y)}\,d\eta .
\]	
For simplicity, we write $\mathcal M F_{\neq}$
for $\mathcal M(t,D)F_{\neq}$.
For every $s\geq0$, it follows that
\begin{equation}\label{mul1}
	\begin{aligned}
		e^{aA^{-\frac13}t}\|F_{\neq}(t)\|_{H^s}
		\leq\|\mathcal M F_{\neq}(t)\|_{H^s}
		\leq e^{2\pi}e^{aA^{-\frac13}t}
		\|F_{\neq}(t)\|_{H^s},
	\end{aligned}
\end{equation}
where the Sobolev norms are taken with respect to $(X,Y)$.
The parameter $a>0$ will be chosen sufficiently small
according to Lemma~\ref{prop:kelvin-timespace}.

On a short-time interval $[0,T]$ with $T\leq T^*$, we define, for a function $F$ with zero $X$-average and $s\geq0$,
\begin{align*}
	\|F\|_{Z_a^s}^2:=\|e^{aA^{-\frac13}t}F\|_{L^\infty H^{s}}^2
	+\frac{\|e^{aA^{-\frac13}t}F\|_{L^2H^{s}}^2}{A^{\frac13}}
	+\frac{\|e^{aA^{-\frac13}t}\nabla_LF\|_{L^2H^{s}}^2}A
	+\|e^{aA^{-\frac13}t}\partial_X\nabla_L\Delta_L^{-1}F\|_{L^2H^{s}}^2.
\end{align*}
The corresponding instantaneous energy and dissipation are
\begin{align*}
	\mathcal{C}_{F}(t)=\|\mathcal{M}F_{\neq }(t)\|_{H^1}^2,\quad
	\mathcal{D}_{F}(t)
	=\frac{\|\mathcal{M}F_{\neq}(t)\|_{H^1}^2}{A^{\frac13}}
	+\frac{\|\mathcal{M}\nabla_LF_{\neq}(t)\|_{H^1}^2}{A}
	+\|\mathcal{M}\partial_X\nabla_L\Delta_L^{-1}F_{\neq}(t)\|_{H^1}^2.
\end{align*}
\noindent\textbf{Key observation 1. Separating the large chemical initial data from the forced response.}
We decompose
\[
C_{\neq}=C_{1,\neq}+C_{2,\neq},\qquad
C_{1,\neq}(0)=(C_{\rm in})_{\neq},\qquad C_{2,\neq}(0)=0.
\]
The components solve \eqref{equ:c1} and \eqref{equ:c2}, respectively. 
The first carries the initial chemical profile, while the second receives the density source and the interaction with $C_0$. Thus, unlike the simplified model, this decomposition retains the nonlinear fluid coupling.

For sufficiently large $A$, the estimates in Section~\ref{sec:short-time} give, on every interval $[0,T]$ of existence with $T\le T^*$,
\[
\frac1A\int_0^{T}
\|e^{aA^{-\frac13}t}\partial_X\nabla_LC_{1,\neq}(t)\|_{H^1}^2\,dt
\leq C\|(C_{\rm in})_{\neq}\|_{H^2}^2.
\]
Since the upper bound is independent of $A$, we can control the contribution of $C_{1,\neq}$ by Gronwall's inequality without imposing a smallness condition on the initial chemical data.

The choice of derivatives is specially important. The anisotropic product estimates in \eqref{embedding} use an additional $X$ derivative and do not require $H^1(\mathbb T\times\mathbb R)$ to be an algebra. We therefore estimate $\partial_XN_{\neq}$ and $\partial_XC_{j,\neq}$, $j=1,2$, at the same $H^1$ level, together with $N_0$ and $\Omega_{\neq}$.
This leads to the following short-time energies and dissipation:
\begin{align*}
	\mathcal E_0(t)
	&=\|N_0(t)\|_{H^1}^2,\\
	\mathcal E_1(t)
	&=\mathcal C_{\partial_XN}(t)
	+\mathcal C_{\partial_XC_1}(t)
	+\mathcal C_{\partial_XC_2}(t)
	+\mathcal C_{\Omega}(t),\\
	\mathcal E_2(t)
	&=\mathcal D_{\partial_XN}(t)
	+\mathcal D_{\partial_XC_1}(t)
	+\mathcal D_{\partial_XC_2}(t)
	+\mathcal D_{\Omega}(t).
\end{align*}
Let $[0,T]$ be a finite interval of existence of the solution
with $0<T\leq T^*$. On this interval,
we impose the bootstrap assumption
\begin{equation}\label{assume}
	\mathcal E_0(t)+\mathcal E_1(t)\leq 2E_1,
	\qquad {t\in[0,T].}
\end{equation}

\noindent\textbf{Key observation 2. Closing the density feedback by a weighted coupled energy.}
The main chemical terms in the density estimates are
\[
\frac1A\|e^{aA^{-\frac13}t}N_0\partial_X\nabla_LC_{\neq}\|_{H^1}^2,
\qquad
\frac1A\|e^{aA^{-\frac13}t}\partial_X(N_{\neq}\nabla_LC_{\neq})\|_{H^1}^2.
\]
After the decomposition, the contributions involving $C_{1,\neq}$ are bounded by
\[
C\bigl(\mathcal E_0+\mathcal C_{\partial_XN}\bigr)
\frac{\|e^{aA^{-\frac13}t}\partial_X\nabla_LC_{1,\neq}\|_{H^1}^2}{A},
\]
whereas those involving $C_{2,\neq}$ contain $CE_1\mathcal D_{\partial_XC_2}$. The former is handled by Gronwall's inequality, while the latter cannot be absorbed merely by increasing $A$, so we assign an additional weight to the energy of $C_{2,\neq}$:
\begin{equation}\label{modified E1}
	\begin{aligned}
		&\widetilde{\mathcal{E}_{1}}(t)=\mathcal{C}_{\partial_X N}(t)
		+\mathcal{C}_{\partial_X C_1}(t)
		+\kappa \mathcal{C}_{\partial_X C_2}(t) +\mathcal{C}_{\Omega}(t),
	\end{aligned}
\end{equation}
where 
\begin{equation*}
	\kappa= 1+6CE_1,
\end{equation*}
and $C$ is a fixed sufficiently large constant chosen in the absorption step.
Since $C_{2,\neq}(0)=0$, we have
\[
\widetilde{\mathcal E_1}(0)=\mathcal E_1(0).
\]
Thus, the large weight $\kappa$ strengthens the dissipation without enlarging the initial energy. We first fix $E_1$ in terms of the initial data and then choose $A$ sufficiently large to absorb the remaining errors, including those multiplied by $\kappa$. 

We also use the decomposition $C_0=C_{1,0}+C_{2,0}$
defined by \eqref{equ:c10} and \eqref{equ:c20}.
Estimating $N_0$ together with the nonzero modes and absorbing
the remaining errors, we obtain
\[
\frac{d}{dt}\bigl(\mathcal E_0+\widetilde{\mathcal E_1}\bigr)
+\frac12\mathcal E_2
\leq C G(t)\bigl(\mathcal E_0+\widetilde{\mathcal E_1}\bigr),
\]
{where}
\[
G(t)=\frac{
	\|e^{aA^{-\frac13}t}\partial_X\nabla_LC_{1,\neq}(t)\|_{H^1}^2
	+\|\partial_YC_{2,0}(t)\|_{H^1}^2
	+\|C_{1,0}(t)\|_{H^2}^2}{A}.
\]
Using the estimates for the chemical components and the fact
that $\frac{T^*}{A}\to0$ as $A\to\infty$, we obtain
\[
\int_0^T G(t)\,dt\leq C(1+\|C_{\rm in}\|_{H^2}^2),
\]
for sufficiently large $A$.
Applying Gronwall's inequality gives an energy bound independent
of $A$. This improves the short-time bootstrap assumption
without imposing any smallness condition on
$(\nabla c_{\mathrm{in}})_{\neq}$.

\noindent$\bullet$\quad
\textbf{Long-time energy functionals in physical coordinates.}
At $T^*$, the factor $e^{-aA^{-\frac13}T^*}=e^{-aA^\delta}$ compensates for every fixed polynomial loss in $A$ and $T^*$. We now estimate the full component $c_{\neq}$ in physical coordinates, without further decomposition. The long-time estimates concern the same solution, with an exponential weight based at $T^*$.

For a finite interval $[0,T]$ of existence with $T>T^*$,
we define the space-time norm on $[T^*,T]$ as follows
\begin{align*}
	\|f\|_{X_a}^2:=&\|e^{aA^{-\frac13}(t-T^*)}f\|_{L^\infty L^2}^2
	+\frac{\|e^{aA^{-\frac13}(t-T^*)}f\|_{L^2L^2}^2}{A^{\frac13}}
	+\frac{\|e^{aA^{-\frac13}(t-T^*)}\nabla f\|_{L^2L^2}^2}{A}\\
	&+\|e^{aA^{-\frac13}(t-T^*)}\partial_x\nabla^{\perp}\Delta^{-1}f\|_{L^2L^2}^2,
\end{align*}
and set
\begin{align*}
	\mathcal{E}_3(t):=\|\partial_x n_{\neq}\|_{X_a}
	+\|\partial_x\nabla c_{\neq}\|_{X_a}
	+\|\partial_x\omega_{\neq}\|_{X_a}.
\end{align*}
{We write the initial-data bound as}
\begin{align*}
	\mathcal{E}_{3,\rm in}=\left(\|n_{\rm in}\|_{H^2} +\|(u_{\rm in})_{\neq}\|_{H^3}
	+\|(c_{\rm in})_{\neq}\|_{H^2}\right) 
	e^{C\left(\|c_{\rm in}\|_{H^2}^2+1\right)}. 
\end{align*}
On $[T^*,T]$, the long-time bootstrap assumption is
\begin{equation}\label{assume1}
	\mathcal{E}_3(t)\leq 2E_2,\quad T^*\leq t\leq T.
\end{equation}
The time norms in $\mathcal E_3(t)$ are taken over $[T^*,t]$, with \(\mathcal E_3(T^*)\) defined by the corresponding \(L^2\) norms at \(T^*\).

To estimate the data at $T^*$, we use the following elementary bound: for fixed $a>0$, $\delta>0$ and every $\beta>0$,
\begin{equation}\label{Abeta}
	\begin{aligned}
		\sup_{A\geq 1}\sup_{t\geq A^{\frac13+\delta}} (1+t^2)A^{\beta}
		e^{-2aA^{-\frac13}t}\leq 2\Big(\frac{2+\frac{\beta+2/3}{\delta}}{2a e}\Big)^{2+\frac{\beta+2/3}{\delta}}< \infty.
	\end{aligned}
\end{equation} 
\rev{If the solution exists on $[0,T^*]$, Proposition~\ref{prop:short-time}
	and \eqref{Abeta} imply
}
\begin{align*}
	\|\partial_xn_{\neq}(T^*)\|_{L^2}+\|\partial_x\omega_{\neq}(T^*)\|_{L^2}+\|\partial_x\partial_yc_{\neq}(T^*)\|_{L^2}+A^{\frac13}\|\partial_x^2c_{\neq}(T^*)\|_{L^2}\leq C\mathcal{E}_{3,\rm in}.\alabel{nwc1}
\end{align*}
The weight $A^{\frac13}$ in front of $\|\partial_x^2c_{\neq}(T^*)\|_{L^2}$ accounts for the source $-\partial_x^2c_{\neq}$ in the equation for $\partial_x\partial_yc_{\neq}$. Its space-time estimate introduces precisely this factor. 

\noindent\textbf{Key observation 3. Closing the energy estimates before the pointwise estimates.}
We estimate the long-time nonlinear terms using the anisotropic product inequalities and the dissipation in $X_a$, without using a prior bound for $\|n\|_{L^\infty L^\infty}$. 
Combining these estimates with the one-dimensional zero-mode bounds, we obtain \eqref{E3} and close the long-time energy estimate for $\mathcal E_3$.

The pointwise bounds are derived after the energy estimates have been closed. In the Moser iteration for $n$, we use
\[
\|n^p\nabla c\|_{L^2}^2
\leq C\|n^p\|_{L^2}\|\partial_yn^p\|_{L^2}
\bigl(\|\partial_yc_0\|_{L^2}
+\|\partial_x\nabla c_{\neq}\|_{L^2}\bigr)^2.
\]
Thus, an $A$-uniform bound for $\|\nabla c\|_{L^\infty L^\infty}$ is not needed to close the nonlinear energy estimates. 
Once these estimates are closed, we apply Moser--Alikakos
iteration~\cite{Alikakos1979} and use the estimates for $c$ and $u$
to obtain uniform-in-time pointwise bounds for each fixed
admissible $A$. These bounds may depend on $A$ but are not
used to choose the shear threshold.

\begin{revision}
	For the estimates in Section~\ref{sec:long-time}, we denote the data at $T^*$ by
	\begin{equation*}
		\begin{aligned}
			(n_{\rm in})_{\neq}&=n_{\neq}|_{t=T^*}, \quad
			(c_{\rm in})_{\neq}= c_{\neq}|_{t=T^*},\quad 
			(\omega_{\rm in})_{\neq}=\omega_{\neq}|_{t=T^*},\\
			(n_{\rm in})_{0}&=n_{0}|_{t=T^*}, \quad\;
			(c_{\rm in})_{0}= c_{0}|_{t=T^*},\quad\;
			(\omega_{\rm in})_{0}=\omega_{0}|_{t=T^*}.
		\end{aligned}
	\end{equation*}
	This notation is used only in these estimates: the data in
	Theorem~\ref{thm:main}, the propositions below and the amplitude
	thresholds always refer to $t=0$.
\end{revision}

\subsection{Main steps of the proof}
\label{subsec:main-steps}
We organize the proof into three steps: short-time energy control, long-time energy closure, and uniform-in-time bounds. All estimates are first established on compact intervals of smooth existence, with constants independent of the terminal time. The two energy arguments are combined with the pointwise estimates to continue the same solution globally.

\noindent\textbf{Step 1. Short-time energy estimates.} 
We first control the short-time energy in Kelvin coordinates and derive the corresponding bounds in physical coordinates.

\begin{proposition}\label{prop:short-time}\label{prop:small time-main}
	Under the hypotheses of Theorem~\ref{thm:main}.
	Fix $0<\delta<\frac23$ and $a\in(0,a_0]$, where $a_0$
	is given by Lemma~\ref{prop:kelvin-timespace}, and set
	$T^*=A^{\frac13+\delta}$. Let $[0,T]$ be an interval of existence of the solution with $T\leq T^*$. There exists a positive constant $C$ such that, if $A\geq A^*$, 
	then
	\[
	\mathcal E_0(t)+\mathcal E_1(t)\leq E_1,
	\qquad \rev{t\in[0,T]},
	\]
	where
	\begin{align*}
		A^*:=C\left(
		1+M^4+\|u_{\rm in}\|_{H^3}^4
		+\|c_{\rm in}\|_{H^2}^4+\|n_{\rm in}\|_{H^2}^4
		\right)^{\frac{3}{2-3\delta}}
		e^{C\frac{\|c_{\rm in}\|_{H^2}^2+1}{2-3\delta}},
	\end{align*}
	and
	\begin{align*}
		E_1:=C\left(\|N_{\rm in}\|_{H^2}^2
		+\|(\Omega_{\rm in})_{\neq}\|_{H^2}^2
		+\|(C_{\rm in})_{\neq}\|_{H^2}^2
		\right)e^{C(\|C_{\rm in}\|_{H^2}^2+1)}.
	\end{align*}
	Consequently, the following estimates hold in physical coordinates
	\rev{for every $t\in[0,T]$}:
	\begin{align*}
		&\|n_0(t)\|_{L^2}\leq CE_1^{\frac12},\\
		&\|\partial_xn_{\neq}(t)\|_{L^2}
		+\|\partial_xc_{\neq}(t)\|_{L^2}
		+\|\omega_{\neq}(t)\|_{L^2}
		+\|\partial_x\omega_{\neq}(t)\|_{L^2}
		\leq CE_1^{\frac12}e^{-aA^{-\frac13}t},\alabel{prop-smalltime}\\
		&\|\partial_x\nabla c_{\neq}(t)\|_{L^2}
		\leq CE_1^{\frac12}(1+t^2)^{\frac12}e^{-aA^{-\frac13}t}.
	\end{align*}
	For the fixed parameters $a$ and $\delta$, the constants are
	independent of $A$, $T$ and the initial data.
\end{proposition}

\noindent\textbf{Step 2. Long-time energy estimates.} 
The short-time estimates control the data at $T^*$, including the weighted chemical derivative required by the shear source. We then close $\mathcal E_3$ in terms of the original initial data.

\begin{proposition}\label{prop:long-time}
	Assume the hypotheses of Theorem~\ref{thm:main}, and suppose that the solution exists on $[0,T]$ for some $T>T^*$.
	For the fixed parameters $a$ and $\delta$, there exists
	a positive constant $\tilde A$, depending only on
	$M$, $\|n_{\rm in}\|_{H^2}$, $\|c_{\rm in}\|_{H^2}$
	and $\|u_{\rm in}\|_{H^3}$, such that, if
	$A\geq\max\{A^*,\tilde A\}$, then
	\[
	\mathcal E_3(t)\leq E_2,\qquad t\in[T^*,T],
	\]
	where
	\begin{align*}
		E_2:=C\left(\|n_{\rm in}\|_{H^2}+\|(u_{\rm in})_{\neq}\|_{H^3}+\|(c_{\rm in})_{\neq}\|_{H^2}\right)e^{C(\|c_{\rm in}\|_{H^2}^2+1)}.
	\end{align*}
	Here the time norms in $\mathcal E_3(t)$ are taken over $[T^*,t]$,
	and $C>0$ is independent of $A$, $T$ and the initial data.
	The initial data in this statement are those of the original problem
	at $t=0$.
\end{proposition}

\noindent\textbf{Step 3. Uniform-in-time pointwise estimates.}
We establish uniform-in-time bounds for $n$, $u$, $c$, and $\nabla c$ needed for global continuation.
\begin{proposition}\label{prop:infty-estimate}
	Assume the hypotheses of Theorem~\ref{thm:main}, and let $[0,T]$ be a finite interval of existence of the solution. If $A\geq\max\{A^*,\tilde A\}$, then
	\begin{align*}
		\|n\|_{L^\infty L^\infty}
		+\|u\|_{L^\infty L^\infty}
		+\|c\|_{L^\infty W^{1,\infty}}
		\leq C\left(A,M,\|n_{\rm in}\|_{H^2},\|c_{\rm in}\|_{H^2},\|c_{\rm in}\|_{W^{1,\infty}},\|u_{\rm in}\|_{H^3}\right),
	\end{align*}
	where the constant is independent of $T$.
\end{proposition}


\begin{revision}
	\begin{proof}[Proof of Theorem~\ref{thm:main}]
		Fix $a\in(0,a_0]$ and $\delta\in(0,\frac23)$, choose $\bar A=\max\{A^*,\widetilde A\}$, and let $A\ge\bar A$.
		By Lemma~\ref{lem:local-continuation}, there exists a unique local strong solution with maximal existence time $T_{\max}\in(0,\infty]$.
		
		We first show that $T_{\max}>T^*$. Suppose that $T_{\max}\leq T^*$. For every $0<S<T_{\max}$, Proposition~\ref{prop:short-time} and the short-time estimates in Proposition~\ref{prop:infty-estimate} imply
		\begin{align*}
			\sup_{0\leq t\leq S}
			\left(
			\|n(t)\|_{L^\infty}
			+\|u(t)\|_{L^\infty}
			+\|c(t)\|_{W^{1,\infty}}
			\right)
			\leq C,
		\end{align*}
		where $C$ depends only on $A$ and the initial data.	
		As $S$ tends to $T_{\max}$ from below, we obtain a uniform bound on $[0,T_{\max})$, which contradicts the blow-up criterion in Lemma~\ref{lem:local-continuation}. 
		Thus $T_{\max}>T^*$, and the data at $T^*$ satisfy \eqref{nwc1}.
		

		For any $T^*<S<T_{\max}$, Proposition~\ref{prop:long-time} gives
		\begin{align*}
			\mathcal E_3(t)\leq E_2,\qquad t\in[T^*,S].
		\end{align*}
		Applying Proposition~\ref{prop:infty-estimate} on $[0,S]$, we obtain
		\begin{align*}
			\|n\|_{L^\infty(0,S;L^\infty)}
			+\|u\|_{L^\infty(0,S;L^\infty)}
			+\|c\|_{L^\infty(0,S;W^{1,\infty})}
			\leq C\left(A,M,\|n_{\rm in}\|_{H^2},\|c_{\rm in}\|_{H^2},\|c_{\rm in}\|_{W^{1,\infty}},\|u_{\rm in}\|_{H^3}\right).
		\end{align*}
		Since the bound is independent of $S$, the blow-up criterion in Lemma~\ref{lem:local-continuation} implies that $T_{\max}=\infty$. Letting $S\to\infty$ proves the asserted uniform-in-time bounds.
	\end{proof}

\end{revision}

\section{Preliminaries}\label{sec:preliminaries}

\subsection{Estimates in Kelvin coordinates}
\label{subsec:kelvin-preliminaries}
We establish space-time estimates and anisotropic Sobolev estimates in the Kelvin-coordinate analysis.

\medskip
\noindent$\bullet$\quad
\textbf{Space-time estimates.}


\begin{lemma}
	\label{prop:kelvin-timespace}
	Assume that $F=F(t,X,Y)$ solves
	\begin{equation}\label{eq:kelvin-linear}
		\partial_tF-\frac1A\Delta_LF
		=\partial_XF^{(1)}
		+F^{(2)}
		+\nabla_L\cdot F^{(3)},
		\quad
		F|_{t=0}=F_{\rm in},
	\end{equation}
	on $[0,T]\times\mathbb T\times\mathbb R$ and 
	$P_0F=P_0F^{(1)}=P_0F^{(2)}=P_0F^{(3)}=0.$
	Then there exist universal constants $0<a_0\leq\frac18$ and $C>0$,
	independent of $A\geq1$ and $T>0$, such that, for every
	$a\in[0,a_0]$,
	\begin{align}
		&\quad \|e^{aA^{-\frac13}t}F\|_{L^\infty L^2}^2
		+\frac1{A^{\frac13}}\|e^{aA^{-\frac13}t}F\|_{L^2L^2}^2
		+\frac1A\|e^{aA^{-\frac13}t}\nabla_LF\|_{L^2L^2}^2
		+\|e^{aA^{-\frac13}t}\partial_X\nabla_L\Delta_L^{-1}F\|_{L^2L^2}^2
		\nonumber\\
		&\leq
		C\Big(
		\|F_{\rm in}\|_{L^2}^2
		+\|e^{aA^{-\frac13}t}\nabla_LF^{(1)}\|_{L^2L^2}^2
		+A^{\frac13}\|e^{aA^{-\frac13}t}F^{(2)}\|_{L^2L^2}^2
		+A\|e^{aA^{-\frac13}t}F^{(3)}\|_{L^2L^2}^2
		\Big).
		\label{eq:kelvin-timespace-estimate}
	\end{align}
\end{lemma}

\begin{proof}
	Since $P_0F=0$, it suffices to consider the Fourier modes
	$k\in\mathbb Z\setminus\{0\}$. Taking the Fourier transform in
	$(X,Y)$, we denote
	$\widehat F_k(t,\eta)=\mathcal F_{X,Y}F(t,k,\eta),$
	and set
	\[
	\xi=\eta-kt,
	\qquad
	q_k(t,\eta)=k^2+\xi^2.
	\]
	Then
	\[
	\widehat{\nabla_LF}_k
	=(ik,i\xi)\widehat F_k,
	\qquad
	\widehat{-\Delta_LF}_k
	=q_k\widehat F_k.
	\]
	Hence \eqref{eq:kelvin-linear} becomes
	\begin{equation}\label{eq:kelvin-fourier-proof}
		\partial_t\widehat F_k
		+\frac1A q_k\widehat F_k
		=ik\widehat F^{(1)}_k
		+\widehat F^{(2)}_k
		+i(k,\xi)\cdot\widehat F^{(3)}_k .
	\end{equation}
	Recalling the definitions of the multipliers in
	Section~\ref{subsec:energy-construction}, we compute
	\begin{equation*}
		\partial_tM_k^{(1)}
		=-\frac{A^{-\frac13}k^2}{k^2+A^{-\frac23}(\eta-kt)^2},
		\qquad
		\partial_tM_k^{(2)}
		=-\frac{k^2}{k^2+(\eta-kt)^2}.
	\end{equation*}
	By the definition of $\mathcal M_k$, we have
	\[
	\frac{\partial_t\mathcal M_k}{\mathcal M_k}
	=aA^{-\frac13}+\partial_tM_k,
	\qquad
	-\partial_tM_k
	=
	\frac{A^{-\frac13}k^2}
	{k^2+A^{-\frac23}\xi^2}
	+\frac{k^2}{q_k}.
	\]
	Notice that
	\begin{align*}
		2\Re\left(\mathcal{M}_k^2\overline{\widehat F_k}\partial_t\widehat F_k\right)=\frac{d}{dt}|\mathcal{M}_k\widehat F_k|^2-2\frac{\partial_t\mathcal{M}_k}{\mathcal{M}_k}|\mathcal{M}_k\widehat{F}_k|^2.
	\end{align*}
	Multiplying \eqref{eq:kelvin-fourier-proof} by
	$2\mathcal M_k^2\overline{\widehat F_k}$
	and taking the real part, we obtain
	\begin{align*}
		&\frac{d}{dt}|\mathcal M_k\widehat F_k|^2
		+\left(\frac{2q_k}{A}-2\partial_tM_k-2aA^{-\frac13}\right)|\mathcal M_k\widehat F_k|^2\\
		\leq&2|\mathcal M_k k\widehat F_k^{(1)}||\mathcal M_k\widehat F_k|
		+2|\mathcal M_k\widehat F_k^{(2)}||\mathcal M_k\widehat F_k|
		+2|\mathcal M_k(k,\xi)\cdot\widehat F_k^{(3)}||\mathcal M_k\widehat F_k|.
	\end{align*}
	Since $k^2\geq1$, setting
	$r=A^{-\frac23}\xi^2/k^2\geq0$ gives
	\begin{align*}
		\frac{q_k}{2A}
		+\frac{2A^{-\frac13}k^2}
		{k^2+A^{-\frac23}\xi^2}
		&\geq A^{-\frac13}
		\left(\frac r2+\frac{2}{1+r}\right)
		=A^{-\frac13}
		\left(\frac{r+1}{2}+\frac{2}{1+r}-\frac{1}{2}\right)
		\geq\frac32A^{-\frac13}.
	\end{align*}
	Consequently, for $0\leq a\leq a_0\leq1/8$, we have
	\begin{equation}\label{eq:qk-estimate}
		\frac{2q_k}{A}-2\partial_tM_k-2aA^{-\frac13}
		\geq\frac{3q_k}{2A}
		+\frac54A^{-\frac13}
		+\frac{2k^2}{q_k}.
	\end{equation}
	For the forcing terms, Young's inequality yields
	\begin{align*}
		&2|\mathcal M_k k\widehat F_k^{(1)}|
		|\mathcal M_k\widehat F_k|
		\leq\frac{k^2}{q_k}|\mathcal M_k\widehat F_k|^2
		+q_k|\mathcal M_k\widehat F_k^{(1)}|^2,\\
		&2|\mathcal M_k\widehat F_k^{(2)}|
		|\mathcal M_k\widehat F_k|
		\leq\frac14A^{-\frac13}|\mathcal M_k\widehat F_k|^2
		+4A^{\frac13}|\mathcal M_k\widehat F_k^{(2)}|^2,\\
		&2|\mathcal M_k(k,\xi)\cdot\widehat F_k^{(3)}|
		|\mathcal M_k\widehat F_k|
		\leq\frac{q_k}{2A}|\mathcal M_k\widehat F_k|^2
		+2A|\mathcal M_k\widehat F_k^{(3)}|^2.
	\end{align*}
	Combining these estimates with \eqref{eq:qk-estimate}, we obtain
	\begin{align*}
		&\frac{d}{dt}|\mathcal M_k\widehat F_k|^2
		+A^{-\frac13}|\mathcal M_k\widehat F_k|^2
		+\frac{q_k}{A}|\mathcal M_k\widehat F_k|^2
		+\frac{k^2}{q_k}|\mathcal M_k\widehat F_k|^2\\
		\leq&
		q_k|\mathcal M_k\widehat F_k^{(1)}|^2
		+4A^{\frac13}|\mathcal M_k\widehat F_k^{(2)}|^2
		+2A|\mathcal M_k\widehat F_k^{(3)}|^2.
	\end{align*}
	Summing over $k\neq0$, and integrating with respect to $\eta$,
	we use Plancherel's theorem and
	\eqref{eq:weighted-M-bounded}. Observe that
	\[
	q_k|\widehat F_k|^2
	=|\widehat{\nabla_LF}_k|^2, \quad 
	\big|\widehat{\partial_X\nabla_L\Delta_L^{-1}F}_k\big|^2
	=\frac{k^2}{k^2+(\eta-kt)^2}|\widehat F_k|^2.
	\]
	Therefore,
	\begin{equation}\label{mui2}
		\begin{aligned}
			&\frac{d}{dt}\|\mathcal MF\|_{L^2}^2
			+\frac{1}{A^{\frac13}}\|\mathcal MF\|_{L^2}^2
			+\frac{1}{A}\|\mathcal M\nabla_LF\|_{L^2}^2
			+\|\mathcal M\partial_X\nabla_L
			\Delta_L^{-1}F\|_{L^2}^2\\
			&\quad\leq C\left(
			\|\mathcal M\nabla_LF^{(1)}\|_{L^2}^2
			+A^{\frac13}\|\mathcal MF^{(2)}\|_{L^2}^2
			+A\|\mathcal MF^{(3)}\|_{L^2}^2
			\right).
		\end{aligned}
	\end{equation}
	Integrating \eqref{mui2} over $[0,T]$ and using \eqref{mul1}, we complete the proof of \eqref{eq:kelvin-timespace-estimate}.

\end{proof}
\begin{remark}
	The estimate \eqref{eq:kelvin-timespace-estimate} remains valid if an additional damping term $\frac{1}{A}F$ is added to the left-hand side of \eqref{eq:kelvin-linear}, namely,
	\begin{align*}
		\partial_tF-\frac{1}{A}\Delta_LF+\frac{1}{A}F
		=\partial_XF^{(1)}+F^{(2)}+\nabla_L\cdot F^{(3)}.
	\end{align*}
	Indeed, the additional term $\frac{1}{A}F$ is dissipative and only improves the corresponding energy estimates. Hence, the proof of \eqref{eq:kelvin-timespace-estimate} carries over without any essential modification.
\end{remark}

\medskip
\noindent$\bullet$\quad
\textbf{Anisotropic Sobolev  estimates.}

\begin{lemma}\label{Poincare}
	Let $F_{\neq}\in H^1(\mathbb{T}\times\mathbb{R})$, we have
	\begin{equation}\label{F-L2}
		\begin{aligned}
			\|F_{\neq}\|_{L^2(\mathbb{T}\times\mathbb{R})}\leq\|\partial_XF_{\neq}\|_{L^2(\mathbb{T}\times\mathbb{R})}\leq\|\nabla_LF_{\neq}\|_{L^2(\mathbb{T}\times\mathbb{R})}.
		\end{aligned}
	\end{equation}
\end{lemma}
\begin{proof}
	Since 
	$\frac{1}{|T|}\int_{\mathbb{T}}F_{\neq}(t,X,Y)dX
	=\frac{1}{|T|}\int_{\mathbb{T}}(F-F_{0})dX=0,$
	the Poincar\'e inequality  gives
	\[
	\|F_{\neq}(\cdot,Y)\|_{L_X^2(\mathbb T)}
	\leq
	\|\partial_XF_{\neq}(\cdot,Y)\|_{L_X^2(\mathbb T)}.
	\]
	Integrating with respect to \(Y\), we obtain
	\[
	\|F_{\neq}\|_{L^2(\mathbb T\times\mathbb R)}
	\leq
	\|\partial_XF_{\neq}\|_{L^2(\mathbb T\times\mathbb R)}
	\leq
	\|\nabla_LF_{\neq}\|_{L^2(\mathbb T\times\mathbb R)},
	\]
	where we used 
	$
	|\nabla_LF_{\neq}|^2
	=
	|\partial_XF_{\neq}|^2
	+
	|(\partial_Y-t\partial_X)F_{\neq}|^2
	\geq
	|\partial_XF_{\neq}|^2.$
	This completes the proof.
\end{proof}

\begin{lemma}\label{lem:U}
	Let $\Omega_{\neq}\in H^1(\mathbb{T}\times\mathbb{R})$ and $U_{\neq}=\nabla_L^\perp\Delta_L^{-1}\Omega_{\neq}$.
	Then
	\begin{align*}
		\|\nabla_LU_{\neq}\|_{L^2}
		=\|\Omega_{\neq}\|_{L^2}.
	\end{align*}
\end{lemma}
\begin{proof}		
	Recall that 
	$
	\nabla_L \cdot U_{\neq}
	=\partial_X U_{1,\neq}+(\partial_Y-t\partial_X)U_{2,\neq}
	=0.
	$
	Moreover, integration by parts yields
	\[
	\begin{aligned}
		\int_{\mathbb{T}\times\mathbb{R}}
		(\partial_Y-t\partial_X)U_{1,\neq}\,\partial_X U_{2,\neq}\,dXdY
		&=
		\int_{\mathbb{T}\times\mathbb{R}}
		\partial_X U_{1,\neq}\,(\partial_Y-t\partial_X)U_{2,\neq}\,dXdY \\
		&=-\|\partial_X U_{1,\neq}\|_{L^2}^2
		=-\|(\partial_Y-t\partial_X)U_{2,\neq}\|_{L^2}^2.
	\end{aligned}
	\]
	Therefore,
	\begin{align*}
		\|\Omega_{\neq}\|_{L^2}^2
		&=\|(\partial_Y-t\partial_X)U_{1,\neq}\|_{L^2}^2
		+\|\partial_X U_{2,\neq}\|_{L^2}^2
		-2\int_{\mathbb{T}\times\mathbb{R}}
		(\partial_Y-t\partial_X)U_{1,\neq}\,\partial_X U_{2,\neq}\,dXdY \\
		&=
		\|(\partial_Y-t\partial_X)U_{1,\neq}\|_{L^2}^2
		+\|\partial_X U_{2,\neq}\|_{L^2}^2
		+\|\partial_X U_{1,\neq}\|_{L^2}^2
		+\|(\partial_Y-t\partial_X)U_{2,\neq}\|_{L^2}^2 
		=\|\nabla_L U_{\neq}\|_{L^2}^2.\alabel{u=w}
	\end{align*}
\end{proof}

\begin{lemma}\label{lem:embedding}
	Let $F_0\in H_Y^1(\mathbb{R})$, and let $F_{\neq},G_{\neq}\in H^1(\mathbb{T}\times\mathbb{R})$ with
	$\partial_XF_{\neq},\partial_XG_{\neq}\in H^1(\mathbb{T}\times\mathbb{R})$.
	Then
	\begin{equation}\label{embedding}
		\begin{aligned}
			\|F_0G_{\neq}\|_{H^1}
			\leq& ~C\|F_0\|_{H_Y^1}\|G_{\neq}\|_{H^1},\\
			\|\partial_X(F_{\neq}G_{\neq})\|_{H^1}
			\leq& ~C\|\partial_XF_{\neq}\|_{H^1}\|\partial_XG_{\neq}\|_{H^1},\\
			\|F_{\neq}G_{\neq}\|_{H^1}
			\leq& ~C\|\partial_XF_{\neq}\|_{H^1}\|G_{\neq}\|_{H^1},\\
			\|(F_{\neq}G_{\neq})_0\|_{H_Y^1}
			\leq& ~C\|F_{\neq}\|_{H^1}\|G_{\neq}\|_{H^1}
			\leq C\|\partial_XF_{\neq}\|_{H^1}\|\partial_XG_{\neq}\|_{H^1}.
		\end{aligned}
	\end{equation}
\end{lemma}
\begin{proof}	
	For a nonzero mode \(F_{\neq}\), we write its Fourier expansion in \(X\) as
	\begin{align*}
		F_{\neq}(t,X,Y)=\sum_{k\in\mathbb{Z} k\neq 0}\widehat{F_{k}}(t,Y)e^{ikX},
	\end{align*} 
	where $\widehat{F_k}(t,Y)=\frac{1}{2\pi}\int_{\mathbb{T}}F_{\neq}(t,X,Y)e^{-ikX}dX.$\\
	\textbf{Estimate of $\eqref{embedding}_1$.} Thanks to $F_0G_{\neq}=\sum_{k\neq0}F_0\widehat{G}_k(t,Y)e^{ikX}$ and Plancherel theorem, we have
	\begin{align*}
		\|F_0G_{\neq}\|_{H^1}^2=2\pi\sum_{k\neq0}\left((1+k^2)\|F_0\widehat{G}_k\|_{L_Y^2}^2+\|\partial_Y(F_0\widehat{G}_k)\|_{L_Y^2}^2\right). \alabel{f0g}
	\end{align*}
	By the Sobolev embedding and the Holder inequality, we obtain
	\begin{align*}
		\|F_0\widehat{G}_k\|_{L_Y^2}\leq&~\|F_0\|_{L_Y^{\infty}}\|\widehat{G}_k\|_{L_Y^2}
		\leq C\|F_0\|_{H_Y^1}\|\widehat{G}_k\|_{L_Y^2},\\
		\|\partial_Y(F_0\widehat{G}_k)\|_{L_Y^2}\leq&~\|\partial_YF_0\|_{L_Y^2}\|\widehat{G}_k\|_{L_Y^{\infty}}+\|\partial_Y\widehat{G}_k\|_{L_Y^2}\|F_0\|_{L_Y^{\infty}}
		\leq C\|F_0\|_{H_Y^1}\|\widehat{G}_k\|_{H_Y^1}.
	\end{align*}
	Substituting the above estimates into \eqref{f0g}, there holds
	\begin{align*}
		\|F_0G_{\neq}\|_{H^1}^2
		\leq&~C\sum_{k\neq0}\left((1+k^2)\|F_0\|_{H_Y^1}^2\|\widehat{G}_k\|_{L_Y^2}^2+\|F_0\|_{H_Y^1}^2\|\widehat{G}_k\|_{H_Y^1}^2\right)\\
		\leq&~C\|F_0\|_{H_Y^1}^2\sum_{k\neq0}\left((1+k^2)\|\widehat{G}_k\|_{L_Y^2}^2+\|\partial_Y\widehat{G}_k\|_{L_Y^2}^2\right)
		\leq C\|F_0\|_{H_Y^1}^2\|G_{\neq}\|_{H^1}^2.
	\end{align*}
	\textbf{Estimate of $\eqref{embedding}_2$.} 
	By the product rule, we have
	\begin{align*}
		\|\partial_X(F_{\neq}G_{\neq})\|_{H^1}
		\leq\|F_{\neq}\partial_XG_{\neq}\|_{H^1}+\|G_{\neq}\partial_XF_{\neq}\|_{H^1}.
	\end{align*}
	By Hölder's inequality, the Sobolev embedding $H^1(\mathbb{R})\hookrightarrow L^{\infty}(\mathbb{R})$ and Lemma \ref{Poincare}, we obtain
	\begin{align*}
		\|F_{\neq}\partial_XG_{\neq}\|_{L^2}
		\leq&~\|F_{\neq}\|_{L^{\infty}}\|\partial_XG_{\neq}\|_{L^2}
		\leq~ C\|F_{\neq}\|_{L_X^2L_Y^{\infty}}^{\frac12}\|\partial_XF_{\neq}\|_{L_X^2L_Y^{\infty}}^{\frac12}\|\partial_XG_{\neq}\|_{L^2}\\
		\leq&~C\|F_{\neq}\|_{L_X^2L_Y^2}^{\frac14}\|\partial_YF_{\neq}\|_{L_X^2L_Y^2}^{\frac14}\|\partial_XF_{\neq}\|_{L_X^2L_Y^2}^{\frac14}\|\partial_X\partial_YF_{\neq}\|_{L_X^2L_Y^2}^{\frac14}\|\partial_XG_{\neq}\|_{L^2}\\
		\leq&~C\|\partial_XF_{\neq}\|_{L^2}^{\frac12}\|\partial_X\partial_YF_{\neq}\|_{L^2}^{\frac12}\|\partial_XG_{\neq}\|_{L^2}
		\leq C\|\partial_XF_{\neq}\|_{H^1}\|\partial_XG_{\neq}\|_{L^2},
	\end{align*}
	and
	\begin{align*}
		\|\partial_X(F_{\neq}\partial_XG_{\neq})\|_{L^2}
		\leq&~\|\partial_XF_{\neq}\|_{L_X^2L_Y^{\infty}}\|\partial_XG_{\neq}\|_{L_X^{\infty}L_Y^2}+\|F_{\neq}\|_{L^{\infty}}\|\partial_X^2G_{\neq}\|_{L^2}\\
		\leq&~C\big(\|\partial_XF_{\neq}\|_{L_X^2L_Y^2}^{\frac12}\|\partial_X\partial_YF_{\neq}\|_{L_X^2L_Y^2}^{\frac12}\|\partial_XG_{\neq}\|_{L_X^2L_Y^2}^{\frac12}\|\partial_X^2G_{\neq}\|_{L_X^2L_Y^2}^{\frac12}\\
		&+\|\partial_XF_{\neq}\|_{H^1}\|\partial_XG_{\neq}\|_{H^1}\big)\\
		\leq&~C\|\partial_XF_{\neq}\|_{H^1}\|\partial_XG_{\neq}\|_{H^1}.
	\end{align*}
	Similarly, we have
	\begin{align*}
		\|\partial_Y(F_{\neq}\partial_XG_{\neq})\|_{L^2}
		\leq&~\|\partial_YF_{\neq}\|_{L_X^{\infty}L_Y^2}\|\partial_XG_{\neq}\|_{L_X^2L_Y^{\infty}}+\|F_{\neq}\|_{L^{\infty}}\|\partial_X\partial_YG_{\neq}\|_{L^2}\\
		\leq&~C\big(\|\partial_YF_{\neq}\|_{L_X^2L_Y^2}^{\frac12}\|\partial_X\partial_YF_{\neq}\|_{L_X^2L_Y^2}^{\frac12}\|\partial_XG_{\neq}\|_{L_X^2L_Y^2}^{\frac12}\|\partial_Y\partial_XG_{\neq}\|_{L_X^2L_Y^2}^{\frac12}\\
		&+\|\partial_XF_{\neq}\|_{H^1}\|\partial_XG_{\neq}\|_{H^1}\big)\\
		\leq&~C\|\partial_XF_{\neq}\|_{H^1}\|\partial_XG_{\neq}\|_{H^1}.
	\end{align*}
	Combining the above estimates yields
	\begin{align*}
		\|F_{\neq}\partial_XG_{\neq}\|_{H^1}
		\leq\|F_{\neq}\partial_XG_{\neq}\|_{L^2}+\|\partial_X(F_{\neq}\partial_XG_{\neq})\|_{L^2}+\|\partial_Y(F_{\neq}\partial_XG_{\neq})\|_{L^2}
		\leq C\|\partial_XF_{\neq}\|_{H^1}\|\partial_XG_{\neq}\|_{H^1}.
	\end{align*}
	By symmetry, we also have
	\begin{align*}
		\|G_{\neq}\partial_XF_{\neq}\|_{H^1}
		\leq C\|\partial_XF_{\neq}\|_{H^1}\|\partial_XG_{\neq}\|_{H^1},
	\end{align*}
	which completes the proof of $\eqref{embedding}_2$.\\
	\textbf{Estimate of $\eqref{embedding}_3$.}
	By the Sobolev embedding and Lemma \ref{Poincare}, we have
	\begin{align*}
		\|F_{\neq}\|_{L^\infty}
		+\|\partial_XF_{\neq}\|_{L_X^\infty L_Y^2}
		+\|\partial_YF_{\neq}\|_{L_X^\infty L_Y^2}
		\leq C\|\partial_XF_{\neq}\|_{H^1},\quad 
		\|G_{\neq}\|_{L_X^2L_Y^\infty}
		\leq C\|G_{\neq}\|_{H^1}.
	\end{align*}
	Hence,
	$\|F_{\neq}G_{\neq}\|_{L^2}\leq
	\|F_{\neq}\|_{L^\infty}\|G_{\neq}\|_{L^2}
	\leq
	C\|\partial_XF_{\neq}\|_{H^1}\|G_{\neq}\|_{H^1}.$
	Moreover, we have
	\begin{align*}
		\|\partial_Y(F_{\neq}G_{\neq})\|_{L^2}
		\leq&~
		\|\partial_YF_{\neq}\|_{L_X^\infty L_Y^2}
		\|G_{\neq}\|_{L_X^2L_Y^\infty}
		+\|F_{\neq}\|_{L^\infty}\|\partial_YG_{\neq}\|_{L^2}
		\leq
		C\|\partial_XF_{\neq}\|_{H^1}\|G_{\neq}\|_{H^1},\\
		\|\partial_X(F_{\neq}G_{\neq})\|_{L^2}
		\leq&~
		\|\partial_XF_{\neq}\|_{L_X^\infty L_Y^2}
		\|G_{\neq}\|_{L_X^2L_Y^\infty}
		+\|F_{\neq}\|_{L^\infty}\|\partial_XG_{\neq}\|_{L^2}
		\leq
		C\|\partial_XF_{\neq}\|_{H^1}\|G_{\neq}\|_{H^1}.
	\end{align*}
	Combining the above estimates gives
	\begin{align*}
		\|F_{\neq}G_{\neq}\|_{H^1}
		\leq
		C\|\partial_XF_{\neq}\|_{H^1}\|G_{\neq}\|_{H^1},
	\end{align*}
	which proves $\eqref{embedding}_3$.\\
	\textbf{Estimate of $\eqref{embedding}_4$.}
	Since $(F_{\neq}G_{\neq})_0=\sum_{k\neq0}\widehat{F}_k(t,Y)\widehat{G}_{-k}(t,Y)$, the Sobolev embedding $H^1(\mathbb{R})\hookrightarrow L^\infty(\mathbb{R})$, the Plancherel's theorem and the Cauchy--Schwarz inequality yield
	\begin{align*}
		\|(F_{\neq}\textsc{G}_{\neq})_0\|_{H_Y^1}
		&\leq
		C\sum_{k\neq0}
		\|\widehat{F}_k\|_{H_Y^1}
		\|\widehat{G}_{-k}\|_{H_Y^1}\\
		&\leq
		C\Big(\sum_{k\neq0}
		\|\widehat{F}_k\|_{H_Y^1}^2\Big)^{\frac12}
		\Big(\sum_{k\neq0}
		\|\widehat{G}_{-k}\|_{H_Y^1}^2\Big)^{\frac12}
		\leq
		C\|F_{\neq}\|_{H^1}\|G_{\neq}\|_{H^1}.
	\end{align*}
	By Lemma \ref{Poincare}, we further obtain
	\begin{align*}
		\|(F_{\neq}G_{\neq})_0\|_{H_Y^1}
		\leq C\|F_{\neq}\|_{H^1}\|G_{\neq}\|_{H^1}
		\leq C\|\partial_XF_{\neq}\|_{H^1}\|\partial_XG_{\neq}\|_{H^1}.
	\end{align*}
\end{proof}	

\subsection{Estimates in physical coordinates}
\label{subsec:physical-preliminaries}

We collect the corresponding estimates in physical coordinates.

\medskip
\noindent$\bullet$\quad
\textbf{Space-time estimates.}


The following consequence of Lemma~\ref{prop:kelvin-timespace}
provides the corresponding space-time estimate in physical
coordinates.

\begin{corollary}\label{cor:physical-timespace}
	Let $A\geq1$ and $T>T^*$. Assume that $f=f(t,x,y)$ solves
	\begin{equation}\label{eq:physical-linear}
		\partial_t f+y\partial_x f-\frac{1}{A}\Delta f
		=\partial_x f^{(1)}+f^{(2)}+\nabla\cdot f^{(3)}
	\end{equation}
	on $[T^*,T]\times\mathbb T\times\mathbb R$, and $P_0f=P_0f^{(1)}=P_0f^{(2)}=P_0f^{(3)}=0.$
	Then, for every $a\in[0,a_0]$, where $a_0$ is given by
	Lemma~\ref{prop:kelvin-timespace}, we have
	\begin{equation}\label{eq:physical-timespace-estimate}
		\begin{aligned}
			\|f\|_{X_a}^2
			\leq&~C\Big(\|f(T^*)\|_{L^2}^2
			+\|e^{aA^{-\frac13}(t-T^*)}\nabla f^{(1)}\|_{L^2L^2}^2\\
			&\qquad+A^{\frac13}\|e^{aA^{-\frac13}(t-T^*)}f^{(2)}\|_{L^2L^2}^2
			+A\|e^{aA^{-\frac13}(t-T^*)}f^{(3)}\|_{L^2L^2}^2
			\Big),
		\end{aligned}
	\end{equation}
	where all time norms are taken over $[T^*,T]$, and
	$C>0$ is independent of $A$, $T^*$ and $T$.
	The same estimate holds if an additional damping term
	$A^{-1}f$ is added to the left-hand side of
	\eqref{eq:physical-linear}.
\end{corollary}

\begin{proof}
	By the Kelvin transformation and \eqref{relation}, we have
	\begin{align*}
		&\|F(t)\|_{L^2}=\|f(t)\|_{L^2},\quad
		\|\nabla_LF(t)\|_{L^2}=\|\nabla f(t)\|_{L^2},\\
		&\|\partial_X\nabla_L\Delta_L^{-1}F(t)\|_{L^2}=\|\partial_x\nabla^\perp\Delta^{-1}f(t)\|_{L^2}.
	\end{align*}
	Integrating \eqref{mui2} over $[T^*,t]$ for $T^*\leq t\leq T$,
	multiplying by $e^{-2aA^{-\frac13}T^*}$, and using
	\eqref{mul1} together with the same norm identities
	for the forcing terms, we obtain
	\eqref{eq:physical-timespace-estimate}.
	The additional damping term contributes a nonnegative
	term to the energy estimate.
\end{proof}

\begin{lemma}\label{Poincare2}
	Let $f$ be a function such that $f_{\neq}\in H^1(\mathbb{T}\times\mathbb{R})$, we have
	\begin{align*}
		\|f_{\neq}\|_{L^2(\mathbb{T}\times\mathbb{R})}
		\leq C\|\partial_xf_{\neq}\|_{L^2(\mathbb{T}\times\mathbb{R})}
		\leq C\|\nabla f_{\neq}\|_{L^2(\mathbb{T}\times\mathbb{R})}.
	\end{align*}
\end{lemma}
\begin{proof}
	Similar to the proof of Lemma \ref{Poincare} and we omit it here.
\end{proof}
\noindent$\bullet$\quad
\textbf{Anisotropic Sobolev  estimates.}
\begin{lemma}
	Let $f_0\in H^1(\mathbb{R})$, $f_{\neq},g_{\neq}\in H^1(\mathbb{T}\times\mathbb{R})$, and $\partial_xf_{\neq}\in H^1(\mathbb{T}\times\mathbb{R})$. Then
	\begin{align*}
		\|f_0g_{\neq}\|_{L^2}
		\leq&~C\|f_0\|_{L^2}\|g_{\neq}\|_{L^2}^{\frac12}\|\nabla g_{\neq}\|_{L^2}^{\frac12},\\
		\|f_0g_{\neq}\|_{L^2}
		\leq&~C\|f_0\|_{H^1}\|g_{\neq}\|_{L^2},\\
		\|f_{\neq}g_{\neq}\|_{L^2}
		\leq&~C\|\partial_xf_{\neq}\|_{L^2}^{\frac12}\|\partial_x\nabla f_{\neq}\|_{L^2}^{\frac12}\|g_{\neq}\|_{L^2}.\alabel{fg-L2}
	\end{align*}
\end{lemma}
\begin{proof}
	Since $f_0$ is independent of $x$, we have
	\begin{align*}
		\|f_0g_{\neq}\|_{L^2}^2
		=\int_{\mathbb{T}}\int_{\mathbb{R}}|f_0(t,y)|^2|g_{\neq}(t,x,y)|^2dydx
		\leq\|f_0\|_{L_y^2}^2\|g_{\neq}\|_{L_x^2L_y^{\infty}}^2,
	\end{align*}
	By the one-dimensional Gagliardo--Nirenberg inequality and the Cauchy--Schwarz inequality, we obtain
	\begin{align*}
		\|f_0g_{\neq}\|_{L^2}^2
		\leq C\|f_0\|_{L^2}^2\|g_{\neq}\|_{L^2}\|\partial_yg_{\neq}\|_{L^2}
		\leq C\|f_0\|_{L^2}^2\|g_{\neq}\|_{L^2}\|\nabla g_{\neq}\|_{L^2}.
	\end{align*}
	According to the Sobolev embedding $H^1(\mathbb{R})\hookrightarrow L^\infty(\mathbb{R})$, there holds
	\begin{align*}
		\|f_0g_{\neq}\|_{L^2}^2
		\leq C\|f_0\|_{L_y^{\infty}}^2\|g_{\neq}\|_{L_x^2L_y^2}^2
		\leq C\|f_0\|_{H^1}^2\|g_{\neq}\|_{L^2}^2.   
	\end{align*}
	For the product of two nonzero modes, we use one-dimensional interpolation in each variable. The Cauchy--Schwarz inequality and Lemma~\ref{Poincare2} give
	\begin{align*}
		\|f_{\neq}g_{\neq}\|_{L^2}^2
		\leq&~C\|f_{\neq}\|_{L^{\infty}}^2\|g_{\neq}\|_{L^2}^2
		\leq C\|f_{\neq}\|_{L_x^2L_y^{\infty}}\|\partial_xf_{\neq}\|_{L_x^2L_y^{\infty}}\|g_{\neq}\|_{L^2}^2\\
		\leq&~C\|f_{\neq}\|_{L_x^2L_y^2}^{\frac12}\|\partial_yf_{\neq}\|_{L_x^2L_y^2}^{\frac12}\|\partial_xf_{\neq}\|_{L_x^2L_y^2}^{\frac12}\|\partial_x\partial_yf_{\neq}\|_{L_x^2L_y^2}^{\frac12}\|g_{\neq}\|_{L^2}^2\\
		\leq&~C\|\partial_xf_{\neq}\|_{L^2}\|\partial_x\nabla f_{\neq}\|_{L^2}\|g_{\neq}\|_{L^2}^2.
	\end{align*}
\end{proof}

\section{Proof of Proposition \ref{prop:short-time}}\label{sec:short-time}
\begin{revision}
	
	Throughout this section, let $[0,T]$ be an interval of existence
	with $0<T\leq T^*$.
	We first estimate the zero and nonzero modes and then use the
	modified weighted energy functional to close the short-time bootstrap.
	
\end{revision}


\subsection{Zero-mode estimates}
\label{subsec:short-zero-modes}
To close the coupled energy estimate, we need an $H^1$ estimate for the density zero mode.

\begin{lemma}\label{lem:zero mode-kelxin}
	\rev{Under the assumption \eqref{assume}, there exists a positive constant $A_0$ depending on $E_1$ such that if $A\geq A_0$, the following estimates hold on $[0,T]$:
	}
	\begin{equation}\label{zero-mode eatimate-small}
		\begin{gathered}
			\|U_{1,0}\|_{L^{\infty}H^1}
			+\frac{\|\partial_Y U_{1,0}\|_{L^2H^1}}{A^{\frac12}} 
			\leq C\Big(\|(U_{1,\rm in})_0\|_{H^1}+1\Big)=:J_1, \\
			\|\Omega_0\|_{L^{\infty}H^1}
			+\frac{\|\partial_Y \Omega_0\|_{L^2H^1}}{A^{\frac12}} 
			\leq C\Big(\|(\Omega_{\rm in})_0\|_{H^1}+1\Big)=:J_2, \\
			\|C_{1,0}\|_{L^{\infty}H^1}
			\leq CM=:J_3,\\
			\|C_{1,0}\|_{L^{\infty}H^2}
			\leq C\big(M+\|N_0\|_{L^{\infty}H^1}\big)=:J_4,\\
			\|C_{1,0}\|_{L^2H^2}
			\leq CA^{\frac16+\frac{\delta}{2}}E_1^{\frac12}=:J_5,\\
			\|C_{2,0}\|_{L^{\infty}H^1}+\frac{\| C_{2,0}\|_{L^{2}H^2}}{A^{\frac12}} 
			\leq C\Big(\|(C_{\rm in})_0\|_{H^1}+1\Big)=:J_6.
		\end{gathered}
	\end{equation}
	
\end{lemma}
\begin{proof}
	\underline{\textbf{I. Zero mode estimate of $U_{1,0}$.}}
	Since 
	\begin{equation*}
		\partial_t U_{1,0}-\frac{1}{A}\partial_Y^2 U_{1,0}
		=-\frac{\partial_Y(U_{2,\neq}U_{1,\neq} )_0}{A},
	\end{equation*}
	a standard energy estimate yields
	\begin{equation}\label{equ:u1}
		\|U_{1,0}\|_{L^{\infty}H^1}^2
		+\frac{\|\partial_Y U_{1,0}\|^2_{L^2H^1}}{A}
		\leq C\Big(\|(U_{1,\rm in})_0\|_{H^1}^2+\frac{\|(U_{2,\neq}U_{1,\neq} )_0\|_{L^2H^1}^2}{A}\Big).
	\end{equation} 
	By \eqref{F-L2}, \eqref{u=w} and $\eqref{embedding}_4$, we have
	\begin{align*}
		\|(U_{2,\neq}U_{1,\neq} )_0\|_{H^1}\leq C\|\partial_XU_{2,\neq}\|_{H^1}\|\partial_XU_{1,\neq}\|_{H^1}
		\leq C\|\nabla_LU_{\neq}\|_{H^1}^2=C\|\Omega_{\neq}\|_{H^1}^2.
	\end{align*}
	Moreover, \eqref{mul1} and \eqref{assume} imply that
	\begin{align*}
		\|\Omega_{\neq}(t)\|_{H^1}^2
		\leq e^{-2aA^{-\frac13}t}\|\mathcal{M}\Omega_{\neq}(t)\|_{H^1}^2
		\leq Ce^{-2aA^{-\frac13}t}E_1. \alabel{OH1}
	\end{align*}
	Consequently,
	\begin{equation*}
		\begin{aligned}
			\|(U_{2,\neq}U_{1,\neq} )_0\|_{L^2H^1}^2
			&=\int_{0}^T\|(U_{2,\neq}U_{1,\neq} )_0(t)\|_{H^1}^2dt
			\leq 
			C\int_{0}^T\|\Omega_{\neq}(t)\|_{H^1}^4dt \\
			&\leq CE_1^2\int_{0}^T  e^{-4aA^{-\frac13}t} dt \leq CA^{\frac13}E_1^2,
		\end{aligned}
	\end{equation*}
	Substituting this estimate into \eqref{equ:u1}, we obtain
	\begin{equation*}
		\|U_{1,0}\|_{L^{\infty}H^1}^2
		+\frac{\|\partial_Y U_{1,0}\|^2_{L^2H^1}}{A}
		\leq C\Big(\|(U_{1,\rm in})_0\|_{H^1}^2+\frac{E_1^2}{A^{\frac23}}\Big)
		\leq C\Big(\|(U_{1,\rm in})_0\|_{H^1}^2+1\Big),
	\end{equation*}
	provided that
	\begin{align*}
		A\geq \max\{1,E_1^3\}=:A_0. \alabel{A-E1}
	\end{align*}
	\underline{\textbf{II. Zero mode estimate of $\Omega_{0}$.}}
	The equation of $\Omega_0$ can be rewritten as follows
	\begin{equation*}
		\partial_t \Omega_{0}-\frac{1}{A}\partial_Y^2 \Omega_{0}
		=-\frac{\partial_Y(U_{2,\neq}\Omega_{\neq} )_0}{A}.
	\end{equation*}
	The energy estimate shows that 
	\begin{equation}\label{equ:u2}
		\|\Omega_0\|_{L^{\infty}H^1}^2
		+\frac{\|\partial_Y \Omega_0\|^2_{L^2H^1}}{A}
		\leq C\Big(\|(\Omega_{\rm in})_0\|_{H^1}^2+\frac{\|(U_{2,\neq} \Omega_{\neq} )_0\|_{L^2H^1}^2}{A}\Big)
		\leq C\Big(\|(\Omega_{\rm in})_0\|_{H^1}^2+\frac{E_1^2}{A^{\frac23}}\Big).
	\end{equation} 
	Indeed, we have
	\begin{align*}
		\|(U_{2,\neq}\Omega_{\neq})_0\|_{H^1}
		\leq C\|U_{2,\neq}\|_{H^1}\|\Omega_{\neq}\|_{H^1}
		\leq C\|\nabla_LU_{\neq}\|_{H^1}\|\Omega_{\neq}\|_{H^1}
		\leq C\|\Omega_{\neq}\|_{H^1}^2.
	\end{align*}
	Combining this estimate with \eqref{OH1}, we obtain 
	\begin{equation*}
		\begin{aligned}
			\|(U_{2,\neq}\Omega_{\neq})_0\|_{L^2H^1}^2
			\leq 
			C\int_{0}^T\|\Omega_{\neq}(t)\|_{H^1}^4dt 
			\leq CE_1^2\int_{0}^T  e^{-4aA^{-\frac13}t} dt \leq CA^{\frac13}E_1^2,
		\end{aligned}
	\end{equation*}
	which gives the last inequality in \eqref{equ:u2}.
	If \eqref{A-E1} holds, we obtain $\eqref{zero-mode eatimate-small}_2$.\\
	\underline{\textbf{III. Zero mode estimates of $C_{1,0},C_{2,0}$.}} For $C_0$, we introduce the decomposition $C_0=C_{1,0}+C_{2,0}$, where $C_{1,0}$ and $C_{2,0}$ satisfy
	\begin{equation}\label{equ:c10}
		\begin{cases}
			\partial_tC_{1,0}-\frac1A\partial_Y^2C_{1,0}+\frac1A C_{1,0}
			=\frac1A N_0,\\
			C_{1, 0}|_{t=0}=0,
		\end{cases}
	\end{equation}
	and
	\begin{equation}\label{equ:c20}
		\begin{cases}
			\partial_tC_{2,0}-\frac1A\partial_Y^2C_{2,0}+\frac1A C_{2,0}
			=-\frac1A\partial_Y(U_{2,\neq}C_{\neq})_0,\\
			C_{2, 0}|_{t=0}=(C_{\rm in})_0.
		\end{cases}
	\end{equation}
	The heat kernel representation for \eqref{equ:c10} is given by
	\[
	C_{1,0}(t,Y)
	=\frac{1}{A}\int_0^te^{-\frac{t-s}{A}}\int_{\mathbb{R}}\frac{1}{\sqrt{4\pi(t-s)A^{-1}}}e^{-\frac{(Y-Z)^2}{4(t-s)A^{-1}}}N_0(s,Z)dZds.
	\]
	Applying Minkowski's inequality and Young's convolution inequality, we obtain
	\[
	\begin{aligned}
		\|C_{1,0}(t)\|_{L^2}
		&\leq\frac{1}{A}\int_0^t e^{-\frac{t-s}{A}}
		\Big\|\int_{\mathbb{R}}\frac{1}{\sqrt{4\pi(t-s)A^{-1}}}e^{-\frac{(Y-Z)^2}{4(t-s)A^{-1}}}N_0(s,Z)dZ\Big\|_{L^2}ds\\
		&\leq\frac{1}{A}\int_0^t e^{-\frac{t-s}{A}}\Big\|\frac{1}{\sqrt{4\pi(t-s)A^{-1}}}e^{-\frac{Y^2}{4(t-s)A^{-1}}}\Big\|_{L^2}\|N_0(s,\cdot)\|_{L^1}ds\\
		&\leq\frac{CM}{A}\int_0^te^{-\frac{t-s}{A}}\Big(\frac{A}{t-s}\Big)^{\frac14}ds
		=CM\int_0^{t/A}e^{-\tau}\tau^{-\frac14}d\tau
		\leq CM.
	\end{aligned}
	\]
	Similarly, differentiating the heat kernel with respect to $Y$ and applying the same argument, we obtain
	\begin{align*}
		\|\partial_YC_{1,0}(t)\|_{L^2}
		&\leq\frac{1}{A}\int_0^t e^{-\frac{t-s}{A}}\Big\|\int_{\mathbb R}\partial_Y\Big(\frac{1}{\sqrt{4\pi(t-s)A^{-1}}}e^{-\frac{(Y-Z)^2}{4(t-s)A^{-1}}}\Big)N_0(s,Z)dZ\Big\|_{L^2}ds\\
		&\leq \frac{1}{A}\int_0^t e^{-\frac{t-s}{A}}\Big\|\frac{1}{\sqrt{4\pi(t-s)A^{-1}}}\frac{Y}{2(t-s)A^{-1}}e^{-\frac{Y^2}{4(t-s)A^{-1}}}\Big\|_{L^2}\|N_0(s,\cdot)\|_{L^1}ds\\
		&\leq
		\frac{CM}{A}\int_0^t e^{-\frac{t-s}{A}}
		\Big(\frac{A}{t-s}\Big)^{\frac34}ds
		=CM\int_0^{t/A}e^{-\tau}\tau^{-\frac34}d\tau\leq CM. \alabel{C1_YL2} 
	\end{align*}
	Therefore,
	\begin{align*}
		\|C_{1,0}\|_{L^{\infty}H^1}^2\leq CM^2.\alabel{C10-H1}
	\end{align*}
	Taking the derivative of \eqref{equ:c10} with respect to $Y$, we obtain
	\begin{align*}
		\partial_t\partial_YC_{1,0}-\frac{1}{A}\partial_Y^3C_{1,0}+\frac{1}{A}\partial_YC_{1,0}=\frac{1}{A}\partial_YN_0,\quad\quad\partial_YC_{1,0}|_{t=0}=0.
	\end{align*}
	Applying the heat kernel representation and Young's convolution inequality, we obtain
	\begin{align*}
		\|\partial_Y^2C_{1,0}(t)\|_{L^2}&\leq
		\frac{1}{A}\int_0^t e^{-\frac{t-s}{A}}
		\left\|\frac{1}{\sqrt{4\pi(t-s)A^{-1}}}\frac{Y}{2(t-s)A^{-1}}e^{-\frac{Y^2}{4(t-s)A^{-1}}}\right\|_{L^1}\|\partial_YN_0(s)\|_{L^2}ds\\
		&\leq\frac{C}{A}\int_0^te^{-\frac{t-s}{A}}\left(\frac{A}{t-s}\right)^{\frac12}\|\partial_YN_0(s)\|_{L^2}ds\\
		&\leq C\|N_0\|_{L^\infty H^1}\int_0^{t/A}e^{-\tau}\tau^{-\frac12}d\tau
		\leq C\|N_0\|_{L^\infty H^1}.
	\end{align*}
	Combining this estimate with \eqref{C10-H1}, we obtain
	\[
	\|C_{1,0}\|_{L^\infty H^2}^2
	\leq C\left(M^2+\|N_0\|_{L^\infty H^1}^2\right).
	\]

	Let $\Lambda_Y=(1-\partial_Y^2)^{\frac12}$. Taking the $L^2$
	inner product of \eqref{equ:c10} with $\Lambda_Y^4C_{1,0}$, we obtain
	\begin{equation}\label{estimate:C10-H2}
		\frac{d}{dt}\|C_{1,0}\|_{H^2}^2
		+\frac{c}{A}\left(
		\|C_{1,0}\|_{H^3}^2+
		\|C_{1,0}\|_{H^2}^2\right)
		\leq \frac CA\|N_0\|_{H^1}^2,
	\end{equation}
	where we used 
	\[
	|
	\langle N_0,\Lambda_Y^4C_{1,0}\rangle|
	=|
	\langle\Lambda_YN_0,\Lambda_Y^3C_{1,0}\rangle|
	\leq\frac1{2}\|C_{1,0}\|_{H^3}^2
	+ C \|N_0\|_{H^1}^2.
	\]
	Multiplying both sides of \eqref{estimate:C10-H2} by $e^{\frac{ct}{A}}$ gives
	\[
	\frac{d}{dt}\left(e^{\frac{ct}{A}}\|C_{1,0}(t)\|_{H^2}^2\right)
	\leq \frac{C}{A}e^{\frac{ct}{A}}\|N_0(t)\|_{H^1}^2.
	\]
	Integrating over $[0,t]$ and using \eqref{assume}, we obtain
	\begin{align*}
		\|C_{1,0}(t)\|_{H^2}^2
		\leq\frac{C}{A}\int_0^t e^{-\frac{c(t-s)}{A}}\|N_0(s)\|_{H^1}^2 ds
		\leq\frac{CE_1}{A}\int_0^t e^{-\frac{c(t-s)}{A}} ds
		\leq CE_1.
	\end{align*}
	\begin{revision}
		Combining it with $t\leq T^*=A^{\frac{1}{3}+\delta}$, there holds
		\begin{equation*}
			\|C_{1,0}\|_{L^2H^2}^2
			=\int_0^{t}\|C_{1,0}(s)\|_{H^2}^2 ds
			\leq\int_0^{T^*}\|C_{1,0}(t)\|_{H^2}^2 dt
			\leq CE_1 T^*
			=CA^{\frac13+\delta}E_1.
		\end{equation*}
	\end{revision}
	
	For $\eqref{equ:c20}$, by direct energy estimate, we obtain
	\begin{align*}
		\|C_{2,0}\|_{L^{\infty}H^1}^2+\frac{\|C_{2,0}\|_{L^2H^2}^2}{A}
		\leq C\Big(\|(C_{\rm in})_0\|_{H^1}^2+\frac{\|(U_{2,\neq}C_{\neq})_0\|_{L^2H^1}^2}{A}\Big). \alabel{C20-H1}
	\end{align*}
	Using \eqref{F-L2}, \eqref{u=w}, $\eqref{embedding}_4$ and \eqref{OH1}, we have
	\begin{align*}
		\|(U_{2,\neq}C_{\neq})_0\|_{H^1}
		\leq&~C\|\partial_XU_{\neq}\|_{H^1}\|\partial_XC_{\neq}\|_{H^1}\leq C\|\Omega_{\neq}\|_{H^1}\left(\|\partial_XC_{1,\neq}\|_{H^1}+\|\partial_XC_{2,\neq}\|_{H^1}\right)\\
		\leq&~Ce^{-2aA^{-\frac13}t}\|\mathcal{M}\Omega_{\neq}(t)\|_{H^1}\left(\|\mathcal{M}\partial_XC_{1,\neq}(t)\|_{H^1}+\|\mathcal{M}\partial_XC_{2,\neq}(t)\|_{H^1}\right)\\
		\leq&~Ce^{-2aA^{-\frac13}t}\mathcal{C}_{\Omega}^{\frac12}(t)\Big(\mathcal{C}_{\partial_XC_1}^{\frac12}(t)+\mathcal{C}_{\partial_XC_2}^{\frac12}(t)\Big)
		\leq Ce^{-2aA^{-\frac13}t}E_1,
	\end{align*}
	which implies
	\begin{align*}
		\|(U_{2,\neq}C_{\neq})_0\|_{L^2H^1}^2
		\leq C\int_0^T e^{-4aA^{-\frac13}t}E_1^2 dt\leq CA^{\frac13}E_1^2.\alabel{U2C-H1}
	\end{align*}
	Under the condition \eqref{A-E1}, combining \eqref{U2C-H1} with \eqref{C20-H1}, we obtain
	\begin{align*}
		\|C_{2,0}\|_{L^{\infty}H^1}^2+\frac{\|C_{2,0}\|_{L^2H^2}^2}{A}
		\leq C\Big(\|(C_{\rm in})_0\|_{H^1}^2+\frac{E_1^2}{A^{\frac23}}\Big)
		\leq C\Big(\|(C_{\rm in})_0\|_{H^1}^2+1\Big).\alabel{C20-L2H2}
	\end{align*}
\end{proof}

\subsection{Non-zero modes estimates}
\label{subsec:short-energy}

Using the zero-mode estimates obtained in the previous subsection, we now establish energy estimates for the non-zero modes.

\begin{lemma}\label{energye1}
	Under the assumptions of Lemma \ref{lem:zero mode-kelxin}, we obtain
	\begin{equation*}
		\begin{aligned}
			\frac{d}{dt}\mathcal{E}_{1}(t)+\mathcal{E}_{2}(t)
			\leq&~C\bigg[\big(\mathcal{E}_{0}(t)+\mathcal{E}_{1}(t)\big)
			\frac{\|e^{aA^{-\frac13}t} \partial_X\nabla_LC_{1,\neq}\|_{H^1}^2
				+\|\partial_Y C_{2,0}\|_{H^1}^2}{A}+4E_1 \mathcal{D}_{\partial_X C_2}(t)\\
			&+\frac{1+\| C_{1,0}\|_{L^{\infty}H^2}^2+\| C_{2,0}\|_{L^{\infty}H^1}^2
				+\|N_0\|^2_{L^{\infty}H^1}+\|U_0\|^2_{L^{\infty}H^1}+\|\Omega_0\|^2_{L^{\infty}H^1}+E_1}{A^{\frac23}}\mathcal{E}_2(t)\bigg].
		\end{aligned}
	\end{equation*}	
\end{lemma}
\begin{proof}
	Recall that 
	\begin{equation*}
		\left(\partial_t-\frac1A\Delta_L\right)C_{\neq}
		=\frac1A(N_{\neq}-C_{\neq})
		-\frac1A\big[U_{1,0}\partial_XC_{\neq}
		+U_{\neq}\cdot\nabla_L C_0
		+(U_{\neq}\cdot\nabla_LC_{\neq})_{\neq}
		\big].
	\end{equation*}
	We introduce the decomposition $C_{\neq}=C_{1, \neq}+C_{2,\neq}$, where $C_{1,\neq}$ and $C_{2,\neq}$ satisfy
	\begin{equation}\label{equ:c1}
		\begin{cases}
			\left(\partial_t-\frac1A\Delta_L+\frac1A \right)C_{1,\neq}
			=-\frac1A\big[U_{1,0}\partial_XC_{1,\neq}
			+(U_{\neq}\cdot\nabla_LC_{1,\neq})_{\neq}\big],\\
			C_{1,\neq}|_{t=0}=(C_{\rm in})_{\neq},
		\end{cases}
	\end{equation}
	and 
	\begin{equation}\label{equ:c2}
		\begin{cases}
			\left(\partial_t-\frac1A\Delta_L+\frac1A \right)C_{2,\neq}
			=\frac1A N_{\neq}-\frac1A\big[U_{1,0}\partial_XC_{2,\neq}
			+U_{\neq}\cdot\nabla_L C_0
			+(U_{\neq}\cdot\nabla_LC_{2,\neq})_{\neq}\big],\\
			C_{2,\neq}|_{t=0}=0.
		\end{cases}
	\end{equation}	
	
	Applying Lemma \ref{prop:kelvin-timespace} to \eqref{equ:c1}, and using Lemma \ref{lem:embedding}, \eqref{F-L2}, and $\eqref{zero-mode eatimate-small}_1$, we obtain
	\begin{equation*}
		\begin{aligned}
			\|\partial_X C_{1, \neq}\|_{Z_a^1}^2
			&\leq C\bigg(\|\partial_X C_{\rm in}\|_{H^1}^2
			+\frac{ \|e^{aA^{-\frac13}t} U_{1,0}\partial_XC_{1,\neq}\|^2_{L^2H^1}+
				\|e^{aA^{-\frac13}t} \partial_X(U_{\neq}C_{1,\neq})_{\neq}\|^2_{L^2H^1} }{A}\bigg)\\
			&\leq C\bigg(\|\partial_X C_{\rm in}\|_{H^1}^2
			+\frac{\left(\|U_{1,0}\|_{L^{\infty}H^1}^2+
				\|\Omega_{\neq}\|_{L^{\infty}H^1}^2\right)\|e^{aA^{-\frac13}t} \partial_XC_{1,\neq}\|^2_{L^2H^1} }{A}
			\bigg)\\
			&\leq C\bigg(\|\partial_X C_{\rm in}\|_{H^1}^2+\frac{(\|U_{1,\rm in}\|_{H^1}^2+E_1^2+1)}{A^{\frac23}}\|\partial_XC_{1,\neq}\|_{Z_a^1}^2\bigg).
		\end{aligned}
	\end{equation*}
	Thus, for
	\begin{align*}
		A\geq 2C\max\left\{E_1^3,\|U_{\rm in}\|_{H^2}^3,1\right\}=:A_1,
	\end{align*}
	the last term can be absorbed into the left-hand side, and hence
	\begin{equation}\label{c1r1}
		\|\partial_X C_{1, \neq}\|_{Z_a^1}^2
		\leq C\|\partial_X C_{\rm in}\|_{H^1}^2.
	\end{equation}
	For $C_{1,\neq}$ and $C_{2,\neq}$, using \eqref{mui2} and \eqref{mul1}, we obtain
	\begin{equation}\label{c11}
		\begin{aligned}
			\frac{d}{dt}\mathcal{C}_{\partial_X C_1}(t)+\mathcal{D}_{\partial_X C_1}(t)			 
			\leq C\bigg(
			\frac{
				\|e^{aA^{-\frac13}t}U_{1,0}\partial_X C_{1,\neq}\|_{H^1}^2+
				\|e^{aA^{-\frac13}t}\partial_X(U_{\neq} C_{1,\neq})_{\neq}\|_{H^1}^2		
			}{A}\bigg)
		\end{aligned}
	\end{equation}
	and
	\begin{equation}\label{c21}
		\begin{aligned}
			&\quad \frac{d}{dt}\mathcal{C}_{\partial_X C_2}(t)+\mathcal{D}_{\partial_X C_2}(t)	\\		 
			&\leq C\bigg(
			\frac{\|e^{aA^{-\frac13}t}\partial_XN_{\neq}\|_{H^1}^2
			}{A^{\frac53}}
			+\frac{
				\|e^{aA^{-\frac13}t}U_{1,0}\partial_X C_{2,\neq}\|_{H^1}^2+
				\|e^{aA^{-\frac13}t}\partial_X U_{\neq} C_0\|_{H^1}^2+
				\|e^{aA^{-\frac13}t}\partial_X(U_{\neq} C_{2,\neq})_{\neq}\|_{H^1}^2}{A}\bigg).
		\end{aligned}
	\end{equation}
	For $j\in\{1,2\}$, by \eqref{mul1}, Lemma \ref{Poincare} and \eqref{F-L2}, we have 
	\begin{equation*}
		\begin{aligned}
			&{\|e^{aA^{-\frac13}t}\partial_X U_{\neq} C_0\|_{H^1}^2}
			\leq C {\|C_0\|_{H^1}^2\|e^{aA^{-\frac13}t}\partial_X U_{\neq}\|_{H^1}^2}
			\leq CA^{\frac13}\big(\|C_{1,0}\|_{H^1}^2+\|C_{2,0}\|_{H^1}^2\big) \mathcal{D}_{\Omega}(t),\\
			&{\|e^{aA^{-\frac13}t}U_{1,0}\partial_X C_{j,\neq}\|_{H^1}^2}
			\leq C\|U_{1,0}\|_{H^1}^2\|e^{aA^{-\frac13}t}\partial_X C_{j,\neq}\|_{H^1}^2
			\leq CA^{\frac13}\|U_{1,0}\|_{H^1}^2\mathcal{D}_{\partial_X C_j}(t),\\
			&{\|e^{aA^{-\frac13}t}\partial_X(U_{\neq} C_{j,\neq})_{\neq}\|_{H^1}^2}
			\leq C{\|e^{aA^{-\frac13}t}\Omega_{\neq } \|_{H^1}^2
				\|e^{aA^{-\frac13}t}\partial_X C_{j,\neq}\|_{H^1}^2}
			\leq CA^{\frac13}{2E_1\mathcal{D}_{\partial_X C_j}(t)}.
		\end{aligned}
	\end{equation*}
	Combining these estimates with \eqref{c11} and \eqref{c21}, we obtain
	\begin{align*}
		&\frac{d}{dt}\mathcal{C}_{\partial_X C_1}(t)+\mathcal{D}_{\partial_X C_1}(t)
		\leq C\frac{2E_1+\|U_{1,0}\|_{L^{\infty}H^1}^2}{A^{\frac23}}\mathcal{E}_2(t),\alabel{C1_X}\\
		&\frac{d}{dt}\mathcal{C}_{\partial_X C_2}(t)+\mathcal{D}_{\partial_X C_2}(t)
		\leq C\frac{1+2E_1
			+\|C_{1,0}\|^2_{L^{\infty}H^1}
			+\|C_{2,0}\|^2_{L^{\infty}H^1}+\|U_{1,0}\|_{L^{\infty}H^1}^2}{A^{\frac23}}\mathcal{E}_2(t).\alabel{C2_X}
	\end{align*}
	Applying \eqref{mui2} to \eqref{eq:nonzero-mode-system-kelvin}, and using \eqref{mul1} and Lemma \ref{lem:embedding}, we obtain
	\begin{equation}\label{result:U}
		\begin{aligned}
			&\frac{d}{dt}\mathcal{C}_{\Omega}(t)+\mathcal{D}_{\Omega}(t)\\
			\leq&~ C\bigg(\frac{\|e^{aA^{-\frac13}t}\partial_XN_{\neq}\|_{H^1}^2}{A^{\frac53}}+
			\frac{\|e^{aA^{-\frac13}t}U_{1,0}\Omega_{\neq}\|_{H^1}^2
				+\|e^{aA^{-\frac13}t}U_{\neq}\Omega_0\|_{H^1}^2
				+\|e^{aA^{-\frac13}t}(U_{\neq}\Omega_{\neq})_{\neq}\|_{H^1}^2}{A}\bigg)\\
			\leq&~ C\bigg(\frac{\mathcal{D}_{\partial_XN}(t)}{A^{\frac43}}
			+\frac{\|U_{1,0}\|_{H^1}^2\|e^{aA^{-\frac13}t}\Omega_{\neq}\|_{H^1}^2
				+\|\Omega_0\|_{H^1}^2\|e^{aA^{-\frac13}t}U_{\neq}\|_{H^1}^2
				+\|e^{aA^{-\frac13}t}\partial_XU_{\neq}\|_{H^1}^2\|\Omega_{\neq}\|_{H^1}^2}{A}\bigg)\\
			\leq&~ C\bigg(\frac{1+\|U_{0}\|^2_{L^{\infty}H^1}
				+\|\Omega_0\|^2_{L^{\infty}H^1}+2E_1}{A^{\frac23}}\bigg)\mathcal{E}_2(t).
		\end{aligned}
	\end{equation}
	We next estimate $N_{\neq}$. Applying \eqref{mui2}, we obtain
	\begin{equation}\label{equ:n0}
		\begin{aligned}
			&\frac{d}{dt}\mathcal{C}_{\partial_X N}(t)+\mathcal{D}_{\partial_X N}(t)\\
			\leq&~ C\bigg(\frac{\|e^{aA^{-\frac13}t} N_0\partial_X\nabla_LC_{\neq}\|_{H^1}^2
				+\|e^{aA^{-\frac13}t} \partial_X(N_{\neq}\nabla_LC_{\neq})\|_{H^1}^2
				+\|e^{aA^{-\frac13}t} \partial_XN_{\neq}\nabla_LC_0\|_{H^1}^2}{A}\\
			&+\frac{\|e^{aA^{-\frac13}t}\partial_X U_{\neq}N_0\|_{H^1}^2
				+\|e^{aA^{-\frac13}t}U_0\partial_X N_{\neq}\|_{H^1}^2
				+\|e^{aA^{-\frac13}t}\partial_X(U_{\neq}N_{\neq})\|_{H^1}^2}{A}\bigg) .
		\end{aligned}
	\end{equation}
	Since $C_{\neq}=C_{1,\neq}+C_{2,\neq}$, Lemma \ref{lem:embedding} gives
	\begin{equation*}
		\begin{aligned}
			\|e^{aA^{-\frac13}t} N_0\partial_X\nabla_LC_{\neq}\|_{H^1}^2
			&\leq C \|N_0\|_{H^1}^2(\|e^{aA^{-\frac13}t} \partial_X\nabla_LC_{1,\neq}\|_{H^1}^2 
			+\|e^{aA^{-\frac13}t} \partial_X\nabla_LC_{2,\neq}\|_{H^1}^2 )\\
			&\leq C\left(\|N_0(t)\|_{H^1}^2\|e^{aA^{-\frac13}t}\partial_X\nabla_LC_{1,\neq}\|_{H^1}^2
			+A\|N_0\|_{L^{\infty}H^1}^2 \mathcal{D}_{\partial_X C_2}(t)\right)
		\end{aligned}
	\end{equation*}
	and 
	\begin{equation*}
		\begin{aligned}
			\|e^{aA^{-\frac13}t} \partial_X(N_{\neq}\nabla_LC_{\neq})\|_{H^1}^2
			&\leq C\|e^{aA^{-\frac13}t}  \partial_X N_{\neq}\|_{H^1}^2(\|\partial_X\nabla_LC_{1,\neq}\|_{H^1}^2 
			+\|\partial_X\nabla_LC_{2,\neq}\|_{H^1}^2 )\\
			&\leq C\left(\mathcal{C}_{\partial_X N}(t)\|e^{aA^{-\frac13}t} \partial_X\nabla_LC_{1,\neq}\|_{H^1}^2
			+2AE_1 \mathcal{D}_{\partial_X C_2}(t)\right).
		\end{aligned}
	\end{equation*}
	Using $C_0=C_{1,0}+C_{2,0}$ together with $\eqref{embedding}_1$, we obtain
	\begin{equation*}
		\begin{aligned}
			\|e^{aA^{-\frac13}t} \partial_XN_{\neq}\nabla_LC_0\|_{H^1}^2
			&\leq C \|e^{aA^{-\frac13}t} \partial_XN_{\neq}\|_{H^1}^2
			(\|\partial_Y C_{1,0}\|_{H^1}^2+\|\partial_Y C_{2,0}\|_{H^1}^2)\\
			&\leq C\left(\mathcal{C}_{\partial_X N}(t)
			\|\partial_Y C_{2,0}\|_{H^1}^2+A^{\frac13}\mathcal{D}_{\partial_X N}(t)\|\partial_Y C_{1,0}\|_{H^1}^2\right).
		\end{aligned}
	\end{equation*}
	Moreover, Lemma \ref{Poincare} and Lemma \ref{lem:embedding} imply that
	\begin{equation*}
		\begin{aligned}
			&\frac{\|e^{aA^{-\frac13}t}\partial_X U_{\neq}N_0\|_{H^1}^2}{A}
			\leq C \frac{\|N_0\|^2_{H^1}\|e^{aA^{-\frac13}t}\partial_X U_{\neq}\|_{H^1}^2}{A}
			\leq C \frac{\|N_0\|^2_{L^{\infty}H^1}\mathcal{D}_{\Omega}(t)}{A^\frac{2}{3}},\\
			&\frac{\|e^{aA^{-\frac13}t}U_0\partial_X N_{\neq}\|_{H^1}^2}{A}
			\leq C \frac{\|U_0\|^2_{H^1}\|e^{aA^{-\frac13}t}\partial_X N_{\neq}\|_{H^1}^2}{A}
			\leq C \frac{\|U_0\|^2_{L^{\infty}H^1}\mathcal{D}_{\partial_X N}(t)}{A^\frac{2}{3}},\\
			&\frac{\|e^{aA^{-\frac13}t}\partial_X(U_{\neq}N_{\neq})\|_{H^1}^2}{A}
			\leq C \frac{\|\partial_X N_{\neq}\|^2_{H^1}\|e^{aA^{-\frac13}t}\partial_X U_{\neq}\|_{H^1}^2}{A}
			\leq C \frac{E_1\mathcal{D}_{\Omega}(t)}{A^\frac{2}{3}}.
		\end{aligned}
	\end{equation*}	
	Substituting the above estimates into \eqref{equ:n0}, we obtain
	\begin{align}
		&\frac{d}{dt}\mathcal{C}_{\partial_X N}(t)+\mathcal{D}_{\partial_X N}(t)\nonumber\\
		\leq&~ C\bigg[\big(\|N_0(t)\|_{H^1}^2+\mathcal{C}_{\partial_X N}(t)\big)
		\frac{\|e^{aA^{-\frac13}t} \partial_X\nabla_LC_{1,\neq}(t)\|_{H^1}^2}{A}
		+\mathcal{C}_{\partial_X N}(t)\frac{\|\partial_Y C_{2,0}(t)\|_{H^1}^2}{A}
		\nonumber\\
		& +\big(\|N_0\|^2_{L^{\infty}H^1}+2E_1\big) \mathcal{D}_{\partial_X C_2}(t)
		+\frac{\|\partial_Y C_{1,0}\|_{H^1}^2+\|N_0\|^2_{L^{\infty}H^1}+\|U_0\|^2_{L^{\infty}H^1}+E_1}{A^\frac{2}{3}}\mathcal{E}_2(t)\bigg]\label{result:n}\\
		\leq&~ C\bigg[ \big(\|N_0(t)\|_{H^1}^2+\mathcal{C}_{\partial_X N}(t)\big)
		\frac{\|e^{aA^{-\frac13}t} \partial_X\nabla_LC_{1,\neq}\|_{H^1}^2
			+\|\partial_Y C_{2,0}\|_{H^1}^2}{A}+4E_1 \mathcal{D}_{\partial_X C_2}(t)\nonumber\\
		&+\frac{\|\partial_Y C_{1,0}\|_{H^1}^2+\|N_0\|^2_{L^{\infty}H^1}
			+\|U_0\|^2_{L^{\infty}H^1}+E_1}{A^\frac{2}{3}}\mathcal{E}_2(t)\bigg].
		\nonumber
	\end{align}
	Combining the estimates \eqref{C1_X}--\eqref{result:U} and \eqref{result:n}, we 
	finish the proof.
\end{proof}

\begin{lemma}\label{energye0}
	Under the assumptions of Lemma \ref{lem:zero mode-kelxin}, the following estimate holds
	\begin{equation*}
		\begin{aligned}
			\frac{d}{dt}\mathcal E_0(t)
			\leq C
			\bigg(\frac{\|C_{1,0}\|_{H^2}^2+\|\partial_YC_{2,0}\|_{H^1}^2}{A}\mathcal E_0(t)
			+2E_1\mathcal{D}_{\partial_X C_2}(t)
			+\frac{\| \partial_X\nabla_LC_{1,\neq}\|_{H^1}^2}{A}
			\mathcal E_1(t)
			+\frac{E_1}{A^\frac{2}{3}}\mathcal{E}_2(t)
			\bigg).
		\end{aligned}
	\end{equation*}
\end{lemma}
\begin{proof}
	Applying the $H^1$ energy estimate to $\eqref{eq:zero-mode-system-kelvin}_1$, we obtain
	\begin{equation}\label{energy:zero-N}
		\frac{d}{dt}\mathcal E_0(t)
		+\frac1A\|\partial_YN_0(t)\|_{H^1}^2
		\leq C\frac{\|N_0\partial_YC_0\|_{H^1}^2
			+\|(N_{\neq}\nabla_L C_{\neq})_0\|_{H^1}^2
			+\|(U_{2,\neq}N_{\neq})_0\|_{H^1}^2}{A}
	\end{equation}
	Since $H^1(\mathbb R)\hookrightarrow L^\infty(\mathbb R)$, we have
	\begin{align*}
		\|fg\|_{H^1}
		\leq C\|f\|_{H^1}\|g\|_{H^1}.
	\end{align*}
	Hence, using $C_0=C_{1,0}+C_{2,0}$, we obtain
	\begin{equation}\label{n0:es1}
		\|N_0\partial_YC_0\|_{H^1}^2
		\leq C\|N_0\|_{H^1}^2\|\partial_YC_0\|_{H^1}^2
		\leq C\mathcal E_0(t)\left(
		\|C_{1,0}\|_{H^2}^2
		+\|\partial_YC_{2,0}\|_{H^1}^2\right).
	\end{equation}
	Moreover, Lemma \ref{Poincare} and Lemma \ref{lem:embedding} yield
	\begin{equation}\label{n0:es2}
		\begin{aligned}
			\|(N_{\neq}\nabla_LC_{\neq})_0\|_{H^1}^2
			&\leq C\|e^{aA^{-\frac13}t}  \partial_X N_{\neq}\|_{H^1}^2
			(\|\partial_X\nabla_LC_{1,\neq}\|_{H^1}^2 
			+\|\partial_X\nabla_LC_{2,\neq}\|_{H^1}^2 )\\
			&\leq C\left(\mathcal{C}_{\partial_X N}(t)\| \partial_X\nabla_LC_{1,\neq}(t)\|_{H^1}^2
			+2AE_1 \mathcal{D}_{\partial_X C_2}(t)\right),\\
			\|(U_{\neq}N_{\neq})_0\|_{H^1}^2
			&\leq C\|\partial_X N_{\neq}\|^2_{H^1}\|e^{aA^{-\frac13}t}\partial_X U_{\neq}\|_{H^1}^2
			\leq CA^{\frac13}E_1
			\mathcal{D}_{\Omega}(t)
			\leq CA^{\frac13}E_1\mathcal{E}_2(t).
		\end{aligned}
	\end{equation}
	Combining \eqref{energy:zero-N}, \eqref{n0:es1}, and \eqref{n0:es2}, we obtain
	\begin{equation*}
		\begin{aligned}
			\frac{d}{dt}\mathcal E_0(t)
			\leq C\left(\frac{\|C_{1,0}\|_{H^2}^2+\|\partial_YC_{2,0}\|_{H^1}^2}{A}\mathcal E_0(t)
			+2E_1 \mathcal{D}_{\partial_X C_2}(t)
			+\frac{\| \partial_X\nabla_LC_{1,\neq}\|_{H^1}^2}{A}
			\mathcal E_1(t)
			+\frac{E_1}{A^\frac{2}{3}}\mathcal{E}_2(t).\right) 
		\end{aligned}
	\end{equation*}
	This completes the proof.
\end{proof}

We now combine the preceding estimates using the modified energy defined in \eqref{modified E1}.
	\begin{proof}[\normalfont\bfseries Proof of Proposition \ref{prop:short-time}]
		\begin{revision}
			If $E_1=0$, then $N_{\rm in}=0$ and $(\Omega_{\rm in})_{\neq}=(C_{\rm in})_{\neq}=0$. By the uniqueness of the solution, $N$, $\Omega_{\neq}$ and $C_{\neq}$ remain zero on $[0,T]$, so the asserted estimates hold. We therefore assume $E_1>0$.
			
			
		\end{revision}
		By \eqref{modified E1}, we have
		\begin{align*}
			\widetilde{\mathcal E}_1(t)
			=\mathcal E_1(t)
			+(\kappa-1)\mathcal C_{\partial_XC_2}(t).
		\end{align*}
		Combining Lemma~\ref{lem:zero mode-kelxin}, Lemma \ref{energye1}, Lemma \ref{energye0}, and \eqref{C2_X}, we obtain
		\begin{align*}
			&\frac{d}{dt}\big(\mathcal E_0(t)+\widetilde{\mathcal E}_1(t)\big)+\mathcal E_2(t)
			=\frac{d}{dt}\big(\mathcal E_0(t)+\mathcal E_1(t)
			+(\kappa-1)\mathcal C_{\partial_XC_2}(t)\big)+\mathcal E_2(t)\\
			\leq&~ C\bigg[
			\big(\mathcal E_0(t)+\mathcal E_1(t)\big)
			\frac{\|e^{aA^{-\frac13}t}\partial_X\nabla_LC_{1,\neq}(t)\|_{H^1}^2+\|\partial_YC_{2,0}(t)\|_{H^1}^2+\|C_{1,0}(t)\|_{H^2}^2}{A}\\
			&\quad+\frac{1+\|C_{1,0}\|_{L^\infty H^2}^2+\|C_{2,0}\|_{L^\infty H^1}^2+\|N_0\|_{L^\infty H^1}^2+\|U_0\|_{L^\infty H^1}^2+\|\Omega_0\|_{L^\infty H^1}^2+E_1}{A^{\frac23}}\mathcal E_2(t)\bigg]\\
			&\quad+(\kappa-1)\bigg(\frac{d}{dt}\mathcal C_{\partial_XC_2}(t)+\mathcal D_{\partial_XC_2}(t)\bigg)\\
			\leq&~ C\bigg[\big(\mathcal E_0(t)+\widetilde{\mathcal E}_1(t)\big)
			\frac{\|e^{aA^{-\frac13}t}\partial_X\nabla_LC_{1,\neq}(t)\|_{H^1}^2+\|\partial_YC_{2,0}(t)\|_{H^1}^2+\|C_{1,0}(t)\|_{H^2}^2}{A}\\
			&\quad+\frac{(1+E_1)\big(1+\|C_{1,0}\|_{L^\infty H^2}^2+\|C_{2,0}\|_{L^\infty H^1}^2+\|N_0\|_{L^\infty H^1}^2+\|U_0\|_{L^\infty H^1}^2+\|\Omega_0\|_{L^\infty H^1}^2+2E_1\big)}{A^{\frac23}}\mathcal E_2(t)\bigg]\\
			\leq&~ C\bigg[\big(\mathcal E_0(t)+\widetilde{\mathcal E}_1(t)\big)\frac{\|e^{aA^{-\frac13}t}\partial_X\nabla_LC_{1,\neq}(t)\|_{H^1}^2+\|\partial_YC_{2,0}(t)\|_{H^1}^2+\|C_{1,0}(t)\|_{H^2}^2}{A}\\
			&\quad+\frac{1+M^4+\|U_{1,\rm in}\|_{H^2}^4+\|C_{\rm in}\|_{H^1}^4+\|\Omega_{\rm in}\|_{H^1}^4+E_1^2}{A^{\frac23-\delta}}\mathcal E_2(t)\bigg].
			\alabel{energy:result1}
		\end{align*}
		Thus, for $0<\delta<\frac23$, if
		\[
		A\geq C\left(1+M^4+\|U_{1,\rm in}\|_{H^2}^4+\|C_{\rm in}\|_{H^1}^4+\|\Omega_{\rm in}\|_{H^1}^4+E_1^2\right)^{\frac{3}{2-3\delta}}
		=:A_2,
		\]
		with $C$ sufficiently large, the last term in \eqref{energy:result1} can be absorbed into the left-hand side, and we obtain
		\begin{equation}\label{energy:result2}
			\begin{aligned}
				\frac{d}{dt}\big(
				\mathcal E_0(t)+\widetilde{\mathcal E}_1(t)
				\big)+\frac12\mathcal E_2(t)
				\leq C\big(
				\mathcal E_0(t)+\widetilde{\mathcal E}_1(t)
				\big)G(t),
			\end{aligned}
		\end{equation}
		where
		\[
		G(t):=
		\frac{
			\|e^{aA^{-\frac13}t}
			\partial_X\nabla_LC_{1,\neq}(t)\|_{H^1}^2
			+\|\partial_YC_{2,0}(t)\|_{H^1}^2
			+\|C_{1,0}(t)\|_{H^2}^2}{A}.
		\]
		Multiplying \eqref{energy:result2} by the integrating factor 
		$e^{-C\int_0^tG(s)\,ds}$, we obtain
		\begin{equation*}
			\begin{aligned}
				\frac{d}{dt}\Big(
				e^{-C\int_0^tG(s)\,ds}
				\big(\mathcal E_0(t)+\widetilde{\mathcal E}_1(t)\big)
				\Big)+\frac12e^{-C\int_0^tG(s)\,ds}\mathcal E_2(t)
				\leq0.
			\end{aligned}
		\end{equation*}
		Integrating over $[0,t]$ and using $\kappa\geq1$, we find
		\begin{equation}\label{result1}
			\begin{aligned}
				\mathcal E_0(t)+\mathcal E_1(t)
				\leq\mathcal E_0(t)+\widetilde{\mathcal E}_1(t)
				\leq e^{C\int_0^tG(s)\,ds}
				\big(\mathcal E_0(0)+\widetilde{\mathcal E}_1(0)\big).
			\end{aligned}
		\end{equation}
		\begin{revision}
			For $A\geq A^*\geq\max\{A_0,A_1,A_2\}$, $\eqref{zero-mode eatimate-small}_5$, \eqref{C20-L2H2} and \eqref{c1r1} yield, 
			\begin{align*}
				\int_0^tG(s)\,ds
				&\leq \frac{\|e^{aA^{-\frac13}t}\partial_X\nabla_LC_{1,\neq}(t)\|_{L^2H^1}^2+\|C_{2,0}(t)\|_{L^2H^2}^2+\|C_{1,0}(t)\|_{L^2H^2}^2}{A}\\
				&\leq\|\partial_XC_{1,\neq}\|_{Z_a^1}^2+C\left(\|(C_{\rm in})_0\|_{H^1}^2+\frac{E_1^2+1}{A^{\frac23-\delta}}\right)
				\leq C\left(\|C_{\rm in}\|_{H^2}^2+1\right).
			\end{align*}
		\end{revision}
		
		Since $C_{2,\neq}(0)=0$, \eqref{mul1} also gives
		\begin{equation*}
			\begin{aligned}
				\mathcal E_0(0)+\widetilde{\mathcal E}_1(0)
				\leq C\Big(\|(N_{\rm in})_0\|_{H^1}^2
				+\|(N_{\rm in})_{\neq}\|_{H^2}^2
				+\|(\Omega_{\rm in})_{\neq}\|_{H^1}^2
				+\|(C_{\rm in})_{\neq}\|_{H^2}^2
				\Big).
			\end{aligned}
		\end{equation*}
		Combining these bounds with \eqref{result1}, and choosing
		the constant in the definition of $E_1$ sufficiently large,
		we conclude that
		\begin{align*}
			\mathcal{E}_{0}(t)+\mathcal{E}_{1}(t)
			\leq C\left(\|N_{\rm in}\|_{H^2}^2 +\|(\Omega_{\rm in})_{\neq}\|_{H^2}^2
			+\|(C_{\rm in})_{\neq}\|_{H^2}^2\right) e^{C\left(\|C_{\rm in}\|_{H^2}^2+1\right)}
			=:E_1, 
		\end{align*}
		This improves the bootstrap assumption \eqref{assume}.
		Since the assumption holds initially, a continuity argument
		yields the energy estimate on $[0,T]$.
		
		At $t=0$, the Kelvin and physical coordinates coincide. Hence,
		\[
		(N_{\rm in},C_{\rm in},U_{\rm in},\Omega_{\rm in})
		=(n_{\rm in},c_{\rm in},u_{\rm in},\omega_{\rm in}),
		\]
		and the definition of the vorticity gives
		\[
		\|(\Omega_{\rm in})_{\neq}\|_{H^2}
		=\|(\omega_{\rm in})_{\neq}\|_{H^2}
		\leq C\|(u_{\rm in})_{\neq}\|_{H^3}.
		\]
		By the definitions of $A_1$ and $A_2$, and enlarging the constant $C$
		if necessary, the constant $A^*$ in the statement satisfies
		\begin{align*}
			A^*:=C(1+M^4+\|u_{\rm in}\|_{H^3}^4+\|c_{\rm in}\|_{H^2}^4+\|n_{\rm in}\|_{H^2}^4)^{\frac{3}{2-3\delta}}e^{C\frac{\|c_{\rm in}\|_{H^2}^2+1}{2-3\delta}}\geq\max\{A_0,A_1,A_2\}.
		\end{align*}
		For $A\geq A^*$, \eqref{mul1} and
		$C_{\neq}=C_{1,\neq}+C_{2,\neq}$ imply
		\begin{equation}\label{re2}
			\begin{aligned}
				&\quad\|N_0(t)\|_{H^1}^2
				+e^{2aA^{-\frac13}t}\Big(
				\|\partial_XN_{\neq}(t)\|_{H^1}^2
				+\|\partial_XC_{\neq}(t)\|_{H^1}^2
				+\|\Omega_{\neq}(t)\|_{H^1}^2
				\Big)\\
				&\leq C\big(\mathcal E_0(t)+\mathcal E_1(t)\big)
				\leq C\left(
				\|n_{\rm in}\|_{H^2}^2
				+\|(u_{\rm in})_{\neq}\|_{H^3}^2
				+\|(c_{\rm in})_{\neq}\|_{H^2}^2
				\right)e^{C(\|c_{\rm in}\|_{H^2}^2+1)}.
			\end{aligned}
		\end{equation}
		We now transfer the estimates from Kelvin coordinates $(X,Y)$ back to the physical coordinates $(x,y)$. By using \eqref{relation}, for $j=0,1$, we have 
		\begin{align*}
			&\|n_0(t)\|_{L^2}
			=\|N_0(t)\|_{L^2},\quad
			\|\partial_x^j f_{\neq}(t)\|_{L^2}
			=\|\partial_X^j F_{\neq}(t)\|_{L^2},\\
			&\|\partial_x\nabla f_{\neq}(t)\|_{L^2}^2
			=\|\partial_X^2F_{\neq}(t)\|_{L^2}^2
			+\|(\partial_Y-t\partial_X)\partial_XF_{\neq}(t)\|_{L^2}^2\\
			&\qquad\qquad\qquad\leq (1+2t^2)\|\partial_X^2F_{\neq}(t)\|_{L^2}^2
			+2\|\partial_Y\partial_XF_{\neq}(t)\|_{L^2}^2
			\leq C(1+t^2)\|\partial_XF_{\neq}(t)\|_{H^1}^2.
		\end{align*}
		Combining these relations with \eqref{re2}, we obtain the desired estimate stated in Proposition \ref{prop:small time-main}. This completes the proof.
	\end{proof}

	\section{Proof of Proposition \ref{prop:long-time}}\label{sec:long-time}
	\begin{revision}
		Throughout this section, let $[0,T]$ be an interval of existence with $T>T^*$. Following~\cite{ZengZhangZi2021}, we first estimate the zero modes on $[T^*,T]$ and then use the data at $T^*$ to close the long-time bootstrap.
	\end{revision}
	
	\subsection{Zero-mode estimates}
	\label{subsec:long-zero-modes}
	

	\begin{lemma}\label{lem:long-time-zero-modes}
		Under the assumption \eqref{assume1}, there exists a positive constant $A_3$ depending on $E_2$, such that if $A\geq A_3:=\max\{1,E_2^{12}\}$, there holds 
		\begin{equation}\label{zero-mode eatimate-long}
			\begin{gathered}
				\|\omega_0\|_{L^\infty L^2}+\frac{1}{A^{\frac12}}\|\partial_y\omega_0\|_{L^2L^2}
				\leq C\left(\|(\omega_{\rm in})_0\|_{L^2}+1\right)=:B_1, \\
				\|u_{1,0}\|_{L^\infty H^1}
				\leq C\left(\|(u_{\rm in})_0\|_{L^2}+\|(\omega_{\rm in})_0\|_{L^2}+1\right)=:B_2, \\
				\|\partial_yc_0\|_{L^\infty L^2}
				\leq C\left(\|(\partial_yc_{\rm in})_0\|_{L^2}+M+1\right)=:B_3,\\
				\|n_0\|_{L^\infty L^2}
				\leq C\left(\|(n_{\rm in})_0\|_{L^2}+\|(\partial_yc_{\rm in})_0\|_{L^2}^2+M^2+1\right)=:B_4.
			\end{gathered}
		\end{equation}
	\end{lemma}
	
	\begin{proof}
		\underline{\textbf{I. Estimate of $\omega_0$.}}
		Taking the $L^2$ inner product of the equation $\eqref{eq:zero-mode-system}_3$ with $\omega_0$ and applying Young's inequality, we obtain
		\begin{align}
			\frac{d}{dt}\|\omega_0\|_{L^2}^2+\frac1A\|\partial_y\omega_0\|_{L^2}^2
			\leq\frac CA\|(u_{2,\neq}\omega_{\neq})_0\|_{L^2}^2. \label{omega0}
		\end{align}
		Since $\nabla\cdot u_{\neq}=0$, integration by parts gives
		\[
		\int_{\mathbb T\times\mathbb R}
		\partial_yu_{1,\neq}\partial_xu_{2,\neq}dxdy
		=\int_{\mathbb T\times\mathbb R}\partial_xu_{1,\neq}\partial_yu_{2,\neq}dxdy
		=-\|\partial_xu_{1,\neq}\|_{L^2}^2
		=-\|\partial_yu_{2,\neq}\|_{L^2}^2.
		\]
		Hence, it is obvious that
		\begin{align*}
			\|\omega_{\neq}\|_{L^2}^2
			=\|\nabla\times u_{\neq}\|_{L^2}^2
			=\|\partial_yu_{1,\neq}-\partial_xu_{2,\neq}\|_{L^2}^2
			=\|\nabla u_{\neq}\|_{L^2}^2. \alabel{u=w-xy}
		\end{align*}
		Integrating \eqref{omega0} over $[T^*,t]$ and using Lemma \ref{Poincare2} and $\eqref{fg-L2}_3$, we obtain
		\begin{align*}
			\|\omega_0\|_{L^{\infty}L^2}^2+\frac{1}{A}\|\partial_y\omega_0\|_{L^2L^2}^2
			\leq&~C\|\omega_0(T^*)\|_{L^2}^2+\frac{C}{A}\|\partial_x\nabla u_{\neq}\|_{L^2L^2}^2\|\omega_{\neq}\|_{L^{\infty}L^2}^2\\
			\leq&~C\|(\omega_{\rm in})_0\|_{L^2}^2+\frac{C}{A}\|\partial_x\omega_{\neq}\|_{L^2L^2}^2\|\partial_x\omega_{\neq}\|_{L^{\infty}L^2}^2\\
			\leq&~C\|(\omega_{\rm in})_0\|_{L^2}^2+\frac{C}{A^{\frac23}}\|\partial_x\omega_{\neq}\|_{X_a}^4
			\leq C\left(\|(\omega_{\rm in})_0\|_{L^2}^2+1\right),
		\end{align*}
		provided that
		\begin{align}
			A\geq E_2^6. \label{A01}
		\end{align}
		\underline{\textbf{II. Estimate of $u_{1,0}$.}}
		The equation for $u_{1,0}$ can be written as
		\begin{align*}
			\partial_tu_{1,0}-\frac{1}{A}\partial_y^2u_{1,0}
			=-\frac{1}{A}\partial_y(u_{2,\neq}u_{1,\neq})_0.
		\end{align*}
		Testing this equation with $u_{1,0}$ and applying Young's inequality, we obtain
		\begin{align*}
			\frac{d}{dt}\|u_{1,0}\|_{L^2}^2+\frac1A\|\partial_yu_{1,0}\|_{L^2}^2
			\leq\frac CA\|(u_{2,\neq}u_{1,\neq})_0\|_{L^2}^2.
		\end{align*}
		Integrating over $[T^*,t]$ and applying Lemma \ref{Poincare2} and $\eqref{fg-L2}_3$, we deduce that
		\begin{align*}
			\|u_{1,0}\|_{L^{\infty}L^2}^2+\frac{1}{A}\|\partial_yu_{1,0}\|_{L^2L^2}^2
			\leq&~C\|u_{1,0}(T^*)\|_{L^2}^2+\frac{C}{A}\|\partial_x\nabla u_{\neq}\|_{L^2L^2}^2\|u_{\neq}\|_{L^{\infty}L^2}^2\\
			\leq&~C\|(u_{\rm in})_0\|_{L^2}^2+\frac{C}{A}\|\partial_x\omega_{\neq}\|_{L^2L^2}^2\|\partial_x\nabla u_{\neq}\|_{L^{\infty}L^2}^2\\
			\leq&~C\|(u_{\rm in})_0\|_{L^2}^2+\frac{C}{A^{\frac23}}\|\partial_x\omega_{\neq}\|_{X_a}^4
			\leq C\left(\|(u_{\rm in})_0\|_{L^2}^2+1\right),
		\end{align*}
		provided that \eqref{A01} holds. Since $\partial_yu_{1,0}=\omega_0$, it follows that
		\begin{align*}
			\|u_{1,0}\|_{L^{\infty}H^1}^2
			\leq \|u_{1,0}\|_{L^{\infty}L^2}^2+\|\omega_0\|_{L^{\infty}L^2}^2
			\leq&~ C\left(\|(u_{\rm in})_0\|_{L^2}^2+\|(\omega_{\rm in})_0\|_{L^2}^2+1\right).
		\end{align*}
		\underline{\textbf{III. Estimate of $\partial_yc_0$.}}
		\begin{revision}
			We decompose $c_0$ as $c_0=c_{1,0}+c_{2,0}$, where $c_{1,0}$ and $c_{2,0}$ solve the following equations, respectively:
			\begin{equation}\label{c10}
				\begin{cases}
					\partial_tc_{1,0}-\frac1A\partial_y^2c_{1,0}+\frac1A c_{1,0}=\frac1A n_0,\\
					c_{1,0}(T^*)=0,
				\end{cases}
			\end{equation}
			and
			\begin{equation}\label{c20}
				\begin{cases}
					\partial_tc_{2,0}-\frac1A\partial_y^2c_{2,0}+\frac1A c_{2,0}
					=-\frac1A\partial_y(u_{2,\neq}c_{\neq})_0,\\
					c_{2,0}(T^*)=(c_{\rm in})_0.
				\end{cases}
			\end{equation}
		\end{revision}
		As in the estimate of \eqref{C1_YL2}, we obtain
		\begin{align*}
			\|\partial_yc_{1,0}\|_{L^{\infty}L^2}^2\leq CM^2. \alabel{c10_y}
		\end{align*}
		Differentiating \eqref{c20} with respect to $y$ and testing the resulting equation with $\partial_yc_{2,0}$, we obtain
		\begin{align*}
			\frac{d}{dt}\|\partial_yc_{2,0}\|_{L^2}^2+\frac{1}{A}\|\partial_yc_{2,0}\|_{H^1}^2\leq\frac{C}{A}\|\partial_y(u_{2,\neq}c_{\neq})_0\|_{L^2}^2.
		\end{align*}
		\begin{revision}
			The anisotropic Sobolev inequality used in the proof of
			$\eqref{fg-L2}_3$ and Lemma~\ref{Poincare2} yield
			\begin{align*}
				\|\partial_y(u_{2,\neq}c_{\neq})_0\|_{L^2}^2
				\leq&~C\left(\|c_{\neq}\partial_yu_{2,\neq}\|_{L^2}^2+\|u_{2,\neq}\partial_yc_{\neq}\|_{L^2}^2\right)\\
				\leq&~C\left(\|\partial_x\nabla c_{\neq}\|_{L^2}^2\|\partial_yu_{2,\neq}\|_{L^2}^2+\|\partial_x\nabla u_{\neq}\|_{L^2}^2\|\partial_yc_{\neq}\|_{L^2}^2\right)\\
				\leq&~C\|\partial_x\nabla c_{\neq}\|_{L^2}^2\|\partial_x\omega_{\neq}\|_{L^2}^2.
			\end{align*}
		\end{revision}
		Integrating in time and applying H\"older's inequality, we deduce that
		\begin{align*}
			\|\partial_yc_{2,0}\|_{L^{\infty}L^2}^2+\frac{1}{A}\|\partial_yc_{2,0}\|_{L^2H^1}^2
			\leq&~\|\partial_yc_{2,0}(T^*)\|_{L^2}^2+\frac{C}{A}\|\partial_x\nabla c_{\neq}\|_{L^2L^2}^2\|\partial_x\omega_{\neq}\|_{L^{\infty}L^2}^2\\
			\leq&~\|(\partial_yc_{\rm in})_0\|_{L^2}^2+\frac{C}{A^{\frac23}}\|\partial_x\nabla c_{\neq}\|_{X_a}^2\|\partial_x\omega_{\neq}\|_{X_a}^2\\
			\leq&~C\left(\|(\partial_yc_{\rm in})_0\|_{L^2}^2+1\right),\alabel{c20_y}
		\end{align*}
		where we used \eqref{A01} in the last inequality. Combining \eqref{c10_y} and \eqref{c20_y}, we obtain $\eqref{zero-mode eatimate-long}_3$.\\
		\underline{\textbf{IV. Estimate of $n_0$.}}
		Since the case of $M=0$ is trivial, we consider the case of $M>0$.
		Taking the $L^2$ inner product of $\eqref{eq:zero-mode-system}_1$ with $n_0$ and applying Young's inequality, we obtain
		\begin{align*}
			\frac12\frac{d}{dt}\|n_0\|_{L^2}^2+\frac{1}{A}\|\partial_yn_0\|_{L^2}^2
			\leq\frac{3}{8A}\|\partial_yn_0\|_{L^2}^2+\frac{C}{A}\left(\|n_0\partial_yc_0\|_{L^2}^2+\|(n_{\neq}\partial_yc_{\neq})_0\|_{L^2}^2+\|(u_{2,\neq}n_{\neq})_0\|_{L^2}^2\right). \alabel{ineq:n0}
		\end{align*}
		By H\"older's inequality and the one-dimensional Gagliardo-Nirenberg inequality, we have
		\begin{align*}
			\|n_0\partial_yc_0\|_{L^2}^2
			=&~|T|\|n_0\partial_yc_0\|_{L_y^2}^2
			\leq|T|\|n_0\|_{L_y^{\infty}}^2\|\partial_yc_0\|_{L_y^2}^2
			\leq C\|n_0\|_{L_y^2}\|\partial_yn_0\|_{L_y^2}\|\partial_yc_0\|_{L_y^2}^2\\
			\leq&~C\|\partial_yn_0\|_{L^2}\|n_0\|_{L^2}\|\partial_yc_0\|_{L^2}^2.
		\end{align*}
		Substituting this bound into \eqref{ineq:n0} and using Young's inequality, we deduce that
		\begin{align}\label{n-0 L22}
			\frac{d}{dt}\|n_0\|_{L^2}^2+\frac{1}{A}\|\partial_yn_0\|_{L^2}^2
			\leq&~\frac{C}{A}\big(\|n_0\|_{L^2}^2\|\partial_yc_0\|_{L^{\infty}L^2}^4+\|(n_{\neq}\partial_yc_{\neq})_0\|_{L^2}^2+\|(u_{2,\neq}n_{\neq})_0\|_{L^2}^2\big).
		\end{align}
		Due to the Nash inequality, we find
		\begin{align}\label{partial_yn_0 L22}
			-\|\partial_yn_0\|_{L^2}^2
			\leq-\frac{\|n_0\|_{L^2}^6}{C\|n_0\|_{L^1}^4}
			\leq-\frac{\|n_0\|_{L^2}^6}{CM^4}.
		\end{align}
		Combining \eqref{n-0 L22} and \eqref{partial_yn_0 L22}, we have
		\begin{align*}
			\frac{d}{dt}\|n_0\|_{L^2}^2
			\leq&-\frac{\|n_0\|_{L^2}^2}{CAM^4}\left(\|n_0\|_{L^2}^4-C^2M^4\|\partial_yc_0\|_{L^{\infty}L^2}^4\right)+\frac{C}{A}\big(\|(n_{\neq}\partial_yc_{\neq})_0\|_{L^2}^2+\|(u_{2,\neq}n_{\neq})_0\|_{L^2}^2\big).\alabel{ineq:n0finnal}
		\end{align*}
		For $t\geq T^*$, set
		\begin{align*}
			H(t):=\frac{C}{A}\int_{T^*}^t\big(\|(n_{\neq}\partial_yc_{\neq})_0\|_{L^2}^2+\|(u_{2,\neq}n_{\neq})_0\|_{L^2}^2\big)ds.
		\end{align*}
		Using $\eqref{fg-L2}_3$, Lemma~\ref{Poincare2}, and \eqref{u=w-xy}, we obtain
		\begin{align}\label{n,u_2, partial_yc}
			&\|(n_{\neq}\partial_yc_{\neq})_0\|_{L^2}^2
			\leq C\|\partial_xn_{\neq}\|_{L^2}\|\partial_x\nabla n_{\neq}\|_{L^2}\|\partial_yc_{\neq}\|_{L^2}^2
			\leq C\|\partial_xn_{\neq}\|_{L^2}\|\partial_x\nabla n_{\neq}\|_{L^2}\|\partial_x\nabla c_{\neq}\|_{L^2}^2,\\\nonumber
			&\|(u_{2,\neq}n_{\neq})_0\|_{L^2}^2
			\leq C\|\partial_xu_{\neq}\|_{L^2}\|\partial_x\nabla u_{\neq}\|_{L^2}\|n_{\neq}\|_{L^2}^2
			\leq C\|\partial_x\nabla u_{\neq}\|_{L^2}^2\|\partial_xn_{\neq}\|_{L^2}^2
			\leq C\|\partial_x\omega_{\neq}\|_{L^2}^2\|\partial_xn_{\neq}\|_{L^2}^2.
		\end{align}
		By H\"older's inequality and \eqref{n,u_2, partial_yc}, we have
		\begin{align*}
			H(t)
			\leq&~\frac{C}{A}\left(\|\partial_x\nabla c_{\neq}\|_{L^{\infty}L^2}^2\|\partial_xn_{\neq}\|_{L^2L^2}\|\partial_x\nabla n_{\neq}\|_{L^2L^2}+\|\partial_x\omega_{\neq}\|_{L^{\infty}L^2}^2\|\partial_xn_{\neq}\|_{L^2L^2}^2\right)\\
			\leq&~\frac{C}{A^{\frac13}}\|\partial_x\nabla c_{\neq}\|_{X_a}^2\|\partial_xn_{\neq}\|_{X_a}^2+\frac{C}{A^{\frac23}}\|\partial_x\omega_{\neq}\|_{X_a}^2\|\partial_xn_{\neq}\|_{X_a}^2
			\leq \frac{C}{A^{\frac13}}E_2^4
			\leq C, \alabel{H(t)}
		\end{align*}
		provided that
		$A\geq E_2^{12}.$	
		Thus, \eqref{ineq:n0finnal} can be rewritten as follows
		\begin{align*}
			\frac{d}{dt}\left(\|n_0\|_{L^2}^2-H(t)\right)
			\leq-\frac{\|n_0\|_{L^2}^2}{CAM^4}\left(\|n_0\|_{L^2}^4-C^2M^4\|\partial_yc_0\|_{L^{\infty}L^2}^4\right).
		\end{align*}
		\begin{revision}
			Applying a standard comparison argument and using \eqref{H(t)}
			and $\eqref{zero-mode eatimate-long}_3$, we obtain
		\end{revision}
		\begin{align*}
			\|n_0\|_{L^{\infty}L^2}^2
			\leq&~C\left(\|n_0(T^*)\|_{L^2}^2+M^2\|\partial_yc_0\|_{L^{\infty}L^2}^2+1\right)\\
			\leq&~C\left(\|(n_{\rm in})_0\|_{L^2}^2+\|(\partial_yc_{\rm in})_0\|_{L^2}^4+M^4+1\right).
		\end{align*}
		This completes the proof.
	\end{proof}

	\subsection{Nonzero-mode estimates and bootstrap closure}
	\label{subsec:long-closure}
	
	Using the zero-mode estimates, we now close the energy estimate for $\mathcal{E}_3(t)$.

	\begin{proof}[\normalfont\bfseries Proof of Proposition \ref{prop:long-time}]
		\underline{\textbf{I. Estimate of $\|\partial_xn_{\neq}\|_{X_a}$.}}
		Applying Corollary \ref{cor:physical-timespace} to $\eqref{eq:non1}_1$, 
		we obtain
		\begin{align*}
			\|\partial_x n_{\neq}\|_{X_a}^2
			\leq&~ C\bigg( \|\partial_xn_{\neq}(T^*)\|_{L^2}^2
			+\frac{\|e^{aA^{-\frac13}(t-T^*)} n_0 \partial_x \nabla c_{\neq}\|^2_{L^2L^2}
				+\|e^{aA^{-\frac13}(t-T^*)} \partial_x n_{\neq} \partial_y c_{0}\|^2_{L^2L^2}}{A}\\
			&+\frac{\|e^{aA^{-\frac13}(t-T^*)} \partial_x(n_{\neq}\nabla c_{\neq})\|^2_{L^2L^2}+\|e^{aA^{-\frac13}(t-T^*)} n_0 \partial_x u_{\neq}\|^2_{L^2L^2}}{A}\\
			&+\frac{\|e^{aA^{-\frac13}(t-T^*)} u_{1,0} \partial_x n_{\neq} \|^2_{L^2L^2}
				+\|e^{aA^{-\frac13}(t-T^*)} \partial_x(u_{\neq}n_{\neq})_{\neq}\|^2_{L^2L^2}}{A}\bigg).	
		\end{align*}
		By $\eqref{fg-L2}_1$, we have
		\begin{equation*}
			\|n_0 \partial_x \nabla c_{\neq}\|^2_{L^2}
			\leq C\|n_0 \|^2_{L^2}
			\|\partial_x \nabla c_{\neq}\|_{L^2}
			\|\partial_x \Delta c_{\neq}\|_{L^2}.
		\end{equation*} 
		Using the above estimate and H\"older's inequality, we obtain
		\begin{align*}
			\|e^{aA^{-\frac13}(t-T^*)} n_0 \partial_x \nabla c_{\neq}\|^2_{L^2L^2}
			\leq&~ C\| n_0 \|^2_{L^{\infty}L^2}
			\|e^{aA^{-\frac13}(t-T^*)}  \partial_x \nabla c_{\neq}\|_{L^2L^2}
			\|e^{aA^{-\frac13}(t-T^*)} \partial_x \Delta c_{\neq}\|_{L^2L^2}\\
			\leq&~ CA^{\frac23}\| n_0 \|^2_{L^{\infty}L^2}\|\partial_x\nabla c_{\neq}\|_{X_a}^2. \alabel{n_x-1}
		\end{align*}
		Similarly, by \eqref{u=w-xy}, we have 
		\begin{align*}
			\|e^{aA^{-\frac13}(t-T^*)} n_0 \partial_x u_{\neq}\|^2_{L^2L^2}
			\leq&~ C\|n_0\|^2_{L^{\infty}L^2}
			\|e^{aA^{-\frac13}(t-T^*)}  \partial_x u_{\neq}\|_{L^2L^2}
			\|e^{aA^{-\frac13}(t-T^*)} \partial_x \nabla u_{\neq}\|_{L^2L^2}\\
			\leq&~CA^{\frac13}\| n_0 \|^2_{L^{\infty}L^2}\|\partial_x\omega_{\neq}\|_{X_a}^2
			\alabel{n_x-2}
		\end{align*}
		and
		\begin{align*}
			\|e^{aA^{-\frac13}(t-T^*)} \partial_x n_{\neq} \partial_yc_{0}\|^2_{L^2L^2}
			\leq&~ C\|\partial_y c_{0}\|_{L^{\infty}L^2}^2
			\|e^{aA^{-\frac13}(t-T^*)} \partial_x n_{\neq}  \|_{L^2L^2} 
			\|e^{aA^{-\frac13}(t-T^*)} \partial_x\nabla n_{\neq}  \|_{L^2L^2} \\
			\leq&~ CA^{\frac23}\|\partial_y c_{0}\|_{L^{\infty}L^2}^2\|\partial_x n_{\neq}\|_{X_a}^2.\alabel{n_x-3}
		\end{align*}
		By $\eqref{fg-L2}_2$, we obtain
		\begin{align*}
			\|e^{aA^{-\frac13}(t-T^*)} u_{1,0} \partial_x n_{\neq}\|^2_{L^2L^2}
			\leq C\| u_{1,0} \|^2_{L^{\infty}H^1}\|e^{aA^{-\frac13}(t-T^*)}  \partial_x n_{\neq}\|_{L^2L^2}^2
			\leq CA^{\frac13}\| u_{1,0} \|^2_{L^{\infty}H^1}\|\partial_xn_{\neq}\|_{X_a}^2.\alabel{n_x-4}
		\end{align*}
		Notice
		$$
		\partial_x(n_{\neq}\nabla c_{\neq})=
		\partial_xn_{\neq}\nabla c_{\neq}+n_{\neq}\partial_x \nabla c_{\neq}.
		$$
		By the one-dimensional Gagliardo--Nirenberg inequality and Lemma \ref{Poincare2}, we have
		\begin{align*}
			\|\nabla c_{\neq}\|_{L^{\infty}_xL^2_{y}}^2\|\partial_xn_{\neq}\|_{L^2_{x}L^{\infty}_y}^2
			\leq&~C\|\nabla c_{\neq}\|_{L^2}\|\nabla\partial_xc_{\neq}\|_{L^2}\|\partial_xn_{\neq}\|_{L^2}\|\partial_x\partial_yn_{\neq}\|_{L^2}\\
			\leq&~C\|\partial_x\nabla c_{\neq}\|_{L^2}^2
			\|\partial_x n_{\neq}\|_{L^2}
			\|\partial_x\nabla n_{\neq}\|_{L^2}, 
		\end{align*} 
		and
		\begin{align*}
			\|n_{\neq}\|_{L^{\infty}}^2\|\partial_x \nabla c_{\neq}\|_{L^2}^2
			\leq&~C\|n_{\neq}\|_{L_x^{\infty}L_y^2}\|\partial_yn_{\neq}\|_{L_x^{\infty}L_y^2}\|\partial_x \nabla c_{\neq}\|_{L^2}^2\\
			\leq&~C\|n_{\neq}\|_{L^2}^{\frac12}\|\partial_xn_{\neq}\|_{L^2}^{\frac12}\|\partial_yn_{\neq}\|_{L^2}^{\frac12}\|\partial_y\partial_xn_{\neq}\|_{L^2}^{\frac12}\|\partial_x\nabla c_{\neq}\|_{L^2}^2\\
			\leq&~C\|\partial_xn_{\neq}\|_{L^2}\|\partial_x\nabla n_{\neq}\|_{L^2}\|\partial_x\nabla c_{\neq}\|_{L^2}^2.
		\end{align*}
		Consequently,
		\begin{align*}
			\|\partial_x(n_{\neq}\nabla c_{\neq})\|_{L^2}^2
			&\leq 
			C\|\nabla c_{\neq}\|_{L^{\infty}_xL^2_{y}}^2
			\|\partial_xn_{\neq}\|_{L^2_{x}L^{\infty}_y}^2
			+\|n_{\neq}\|_{L^{\infty}}^2\|\partial_x \nabla c_{\neq}\|_{L^2}^2\\
			&\leq C\|\partial_x\nabla c_{\neq}\|_{L^2}^2
			\|\partial_x n_{\neq}\|_{L^2}
			\|\partial_x\nabla n_{\neq}\|_{L^2}.
		\end{align*}
		Applying H\"older's inequality in time, we obtain
		\begin{align*}
			&\|e^{aA^{-\frac13}(t-T^*)} \partial_x(n_{\neq}\nabla c_{\neq})\|^2_{L^2L^2}\\
			&\quad\leq C
			\|e^{aA^{-\frac13}(t-T^*)}  \partial_x \nabla c_{\neq}\|^2_{L^{\infty}L^2}
			\|e^{aA^{-\frac13}(t-T^*)}  \partial_x n_{\neq}\|_{L^2L^2}
			\|e^{aA^{-\frac13}(t-T^*)}  \partial_x \nabla n_{\neq}\|_{L^2L^2}\\
			&\quad\leq CA^{\frac23}\|\partial_x\nabla c_{\neq}\|_{X_a}^2\|\partial_xn_{\neq}\|_{X_a}^2
			\leq CA^{\frac23}E_2^2\|\partial_xn_{\neq}\|_{X_a}^2.\alabel{n_x-5}
		\end{align*} 
		The same argument as in \eqref{n_x-5} gives
		\begin{equation}\label{n_x-6}
			\begin{aligned}
				\|e^{aA^{-\frac13}(t-T^*)} \partial_x(n_{\neq}u_{\neq})\|^2_{L^2L^2}
				&\leq C
				\|e^{aA^{-\frac13}(t-T^*)}  \partial_x u_{\neq}\|^2_{L^{\infty}L^2}
				\|e^{aA^{-\frac13}(t-T^*)}  \partial_x n_{\neq}\|_{L^2L^2}
				\|e^{aA^{-\frac13}(t-T^*)}  \partial_x \nabla n_{\neq}\|_{L^2L^2}\\
				&\leq CA^{\frac23}\|\partial_x\omega_{\neq}\|_{X_a}^2
				\|\partial_xn_{\neq}\|_{X_a}^2
				\leq CA^{\frac23}E_2^2\|\partial_xn_{\neq}\|_{X_a}^2.
			\end{aligned}
		\end{equation}
		Combining \eqref{n_x-1}-\eqref{n_x-6}, we obtain
		\begin{equation}\label{rn1}
			\begin{aligned}
				\|\partial_x n_{\neq}\|_{X_a}
				\leq &~C\Big(\|\partial_xn_{\neq}(T^*)\|_{L^2}
				+\frac{\|n_0\|_{L^{\infty}L^2}+\|u_{1,0}\|_{L^{\infty}H^1}}{A^{\frac13}}\left(\|\partial_x\omega_{\neq}\|_{X_a}+\|\partial_xn_{\neq}\|_{X_a}\right)\\
				&+\frac{E_2+\|\partial_y c_{0}\|_{L^{\infty}L^2}
					+\|n_{0}\|_{L^{\infty}L^2}}{A^{\frac16}}
				\left(\|\partial_xn_{\neq}\|_{X_a}+\|\partial_x\nabla c_{\neq}\|_{X_a}\right)\Big).
			\end{aligned}
		\end{equation}
		\underline{\textbf{II. Estimate of $\|\partial_x\nabla c_{\neq}\|_{X_a}$.}}
		Differentiating $\eqref{eq:non1}_2$ twice with respect to $x$, we obtain 
		\begin{equation*}
			\begin{aligned}
				\mathcal L_A \partial_x^2 c_{\neq}
				+\frac{1}{A}\partial_x^2 c_{\neq}
				=
				\frac{\partial_x^2n_{\neq}}{A}
				-\frac{\partial_x^2\big[
					u_{1,0}\partial_xc_{\neq}
					+u_{2,\neq}\partial_yc_0
					+(u_{\neq}\cdot\nabla c_{\neq})_{\neq}
					\big]}{A}.
			\end{aligned}
		\end{equation*}	
		Applying Corollary \ref{cor:physical-timespace} to the above equation, we have 
		\begin{equation}\label{equc1}
			\begin{aligned}
				\|\partial_x^2 c_{\neq}\|_{X_a}^2
				\leq&~ C\bigg(\|\partial_x^2 c_{\neq}(T^*)\|_{L^2}^2
				+\frac{\|e^{aA^{-\frac13}(t-T^*)} \partial_x \nabla n_{\neq} \|^2_{L^2L^2}}{A^2}	+\frac{\|e^{aA^{-\frac13}(t-T^*)}\partial_x^2(u_{\neq}c_{\neq})_{\neq}\|^2_{L^2L^2}}{A}\\
				&+\frac{\|e^{aA^{-\frac13}(t-T^*)} u_{1,0} \partial_x^3 c_{\neq}\|^2_{L^2L^2}
					+\|e^{aA^{-\frac13}(t-T^*)} \partial_x^2 u_{2,\neq}\partial_yc_0\|^2_{L^2L^2}}
				{A^{\frac53}}\bigg).
			\end{aligned}
		\end{equation}
		\begin{revision}
			Since $u_{\neq}=\nabla^\perp\Delta^{-1}\omega_{\neq}$, the $L^2$ boundedness of the Riesz transforms implies
			\begin{align*}
				&\|\partial_x^2 u_{2,\neq}\|_{L^2L^2}
				\leq \|\partial_x^2 \nabla^{\perp}\Delta^{-1}\omega_{\neq}\|_{L^2L^2}
				\leq \|\partial_x\omega_{\neq}\|_{X_a},\\
				&\|\partial_x^2\nabla u_{2,\neq}\|_{L^2L^2}
				=\|-\partial_x^2\Delta^{-1}(\partial_x\nabla\omega_{\neq})\|_{L^2L^2}
				\leq\|\partial_x\nabla\omega_{\neq}\|_{L^2L^2}
				\leq A^{\frac12}\|\partial_x\omega_{\neq}\|_{X_a}.
			\end{align*}
		\end{revision}
		Using the above estimates, $\eqref{fg-L2}_1$ and H\"older's inequality, we obtain
		\begin{equation}\label{c_x-1}
			\begin{aligned}
				&\|e^{aA^{-\frac13}(t-T^*)} \partial_x^2 u_{2,\neq}\partial_yc_0\|^2_{L^2L^2} \\
				\leq&~C\|\partial_yc_0\|_{L^{\infty}L^2}^2
				\|e^{aA^{-\frac13}(t-T^*)} \partial_x^2 u_{2,\neq}\|_{L^2L^2}
				\|e^{aA^{-\frac13}(t-T^*)} \partial_x^2 \nabla u_{2,\neq}\|_{L^2L^2} 
				\leq CA^{\frac12}\|\partial_yc_0\|_{L^{\infty}L^2}^2\|\partial_x\omega_{\neq}\|_{X_a}^2,
			\end{aligned}
		\end{equation}
		Similar to the estimate of \eqref{n_x-5}, we have 
		\begin{equation*}
			\begin{aligned}
				\|\partial_x^2(u_{\neq}c_{\neq})_{\neq}\|^2_{L^2}
				\leq&~C\left(\|\partial_xc_{\neq}\partial_xu_{\neq}\|_{L^2}^2+\|c_{\neq}\partial_x^2u_{\neq}\|_{L^2}^2+\|u_{\neq}\partial_x^2c_{\neq}\|_{L^2}^2\right)\\
				\leq&~C\left(\|\partial_xc_{\neq}\|_{L_x^2L_y^{\infty}}^2\|\partial_xu_{\neq}\|_{L_x^{\infty}L_y^2}^2+\|c_{\neq}\|_{L^{\infty}}^2\|\partial_x^2u_{\neq}\|_{L^2}^2+\|u_{\neq}\|_{L_x^{\infty}L_y^2}^2\|\partial_x^2c_{\neq}\|_{L_x^2L_y^{\infty}}^2\right)\\
				\leq&~C\|\partial_x^2u_{\neq}\|_{L^2}^2\|\partial_x^2c_{\neq}\|_{L^2}\|\partial_x^2\nabla c_{\neq}\|_{L^2}.
			\end{aligned}
		\end{equation*}
		Applying H\"older's inequality and \eqref{u=w-xy}, we obtain
		\begin{align*}
			\|e^{aA^{-\frac13}(t-T^*)}\partial_x^2(u_{\neq}c_{\neq})_{\neq}\|^2_{L^2L^2}				
			\leq&~ C\|e^{aA^{-\frac13}(t-T^*)}\partial_x^2u_{\neq}\|^2_{L^{\infty}L^2}
			\|\partial_x^2c_{\neq}\|_{L^2L^2}
			\|\partial_x^2\nabla c_{\neq}\|_{L^2L^2}\\
			\leq&~ CA^{\frac23}\|e^{aA^{-\frac13}(t-T^*)}\partial_x\omega_{\neq}\|_{L^{\infty}L^2}^2\|\partial_x^2 c_{\neq}\|_{X_a}^2
			\leq CA^{\frac23}E_2^2\|\partial_x^2 c_{\neq}\|_{X_a}^2. \alabel{c_x-2}
		\end{align*}
		By $\eqref{fg-L2}_2$, we have 
		\begin{align*}
			\|e^{aA^{-\frac13}(t-T^*)} u_{1,0} \partial_x^3 c_{\neq}\|^2_{L^2L^2}
			\leq \|u_{1,0}\|_{L^{\infty}H^1}^2
			\|e^{aA^{-\frac13}(t-T^*)} \partial_x^3 c_{\neq}\|^2_{L^2L^2}
			\leq A\|u_{1,0} \|_{L^{\infty}H^1}^2\|\partial_x^2 c_{\neq}\|_{X_a}^2. \alabel{c_x-3}
		\end{align*}
		Substituting \eqref{c_x-1}--\eqref{c_x-3} into \eqref{equc1}, we obtain 
		\begin{equation}\label{cx2}
			\begin{aligned}
				\|\partial_x^2 c_{\neq}\|_{X_a} 
				&\leq C\Big(\|\partial_x^2 c_{\neq}(T^*)\|_{L^2}+\frac{\|\partial_x n_{\neq}\|_{X_a}}{A^{\frac12}}
				+\frac{E_2+\|u_{1,0} \|_{L^{\infty}H^1}}
				{A^{\frac16}}\|\partial_x^2 c_{\neq}\|_{X_a}
				+\frac{\|\partial_yc_0\|_{L^{\infty}L^2}}{A^{\frac{7}{12}}}\|\partial_x\omega_{\neq}\|_{X_a}\Big)\\
				&\leq C\Big(\|\partial_x^2 c_{\neq}(T^*)\|_{L^2}+\frac{\|\partial_x n_{\neq}\|_{X_a}}{A^{\frac12}}
				+\frac{\|\partial_yc_0\|_{L^{\infty}L^2}\|\partial_x\omega_{\neq}\|_{X_a}}{A^{\frac{7}{12}}}\Big)+\frac12\|\partial_x^2 c_{\neq}\|_{X_a},
			\end{aligned}
		\end{equation}
		provided that
		\begin{align*}
			A\geq\left[2C(E_2+B_2)\right]^6=:A_4.
		\end{align*}
		
		Differentiating $\eqref{eq:non1}_2$ with respect to $x$ and $y$, we obtain
		\begin{align*}
			\mathcal L_A \partial_x\partial_yc_{\neq}
			+\frac{1}{A}\partial_x\partial_yc_{\neq}+\partial_x^2c_{\neq}
			=\frac{\partial_x\partial_yn_{\neq}}{A}
			-\frac{\partial_x\partial_y\big[u_{1,0}\partial_xc_{\neq}+u_{2,\neq}\partial_yc_0+(u_{\neq}\cdot\nabla c_{\neq})_{\neq}\big]}{A}.
		\end{align*}
		Applying Corollary \ref{cor:physical-timespace} to it, we have
		\begin{equation}\label{equc2}
			\begin{aligned}
				\|\partial_x\partial_y c_{\neq}\|_{X_a}^2
				&\leq C\bigg(\|\partial_x\partial_y c_{\neq}(T^*)\|_{L^2}^2
				+\frac{\|e^{aA^{-\frac13}(t-T^*)} \partial_x n_{\neq} \|^2_{L^2L^2}}{A}		+\frac{\|e^{aA^{-\frac13}(t-T^*)}\partial_x(u_{\neq}\cdot\nabla c_{\neq})\|^2_{L^2L^2}}{A}\\ 
				&\quad+\frac{\|e^{aA^{-\frac13}(t-T^*)} u_{1,0} \partial_x^2 c_{\neq}\|^2_{L^2L^2}
					+\|e^{aA^{-\frac13}(t-T^*)} \partial_x u_{2,\neq}\partial_yc_0\|^2_{L^2L^2}}{A}\\
				&\quad+A^{\frac{1}{3}}\|e^{aA^{-\frac13}(t-T^*)} \partial_x^2 c_{\neq} \|^2_{L^2L^2}\bigg).
			\end{aligned}
		\end{equation}
		By the one-dimensional Gagliardo--Nirenberg inequality, Lemma \ref{Poincare2}, and \eqref{assume1}, we obtain	
		\begin{equation*}
			\begin{aligned}
				\|e^{aA^{-\frac13}(t-T^*)}\partial_x(u_{\neq}\cdot\nabla c_{\neq})\|^2_{L^2L^2}
				\leq&~ C\|e^{aA^{-\frac13}(t-T^*)}\partial_x\nabla u_{\neq}\|^2_{L^{\infty}L^2}
				\|e^{aA^{-\frac13}(t-T^*)}\partial_x \nabla c_{\neq}\|^2_{L^2L^2}\\
				\leq&~ CA^{\frac13}E_2^2\|\partial_x\nabla c_{\neq}\|_{X_a}^2,
			\end{aligned}
		\end{equation*}	
		and
		\begin{equation*}
			\begin{aligned}
				\|e^{aA^{-\frac13}(t-T^*)} u_{1,0} \partial_x^2 c_{\neq} 	\|^2_{L^2L^2}
				\leq C\|u_{1,0} \|_{L^{\infty}H^1}^2
				\|e^{aA^{-\frac13}(t-T^*)} \partial_x^2 c_{\neq} 	\|^2_{L^2L^2}
				\leq CA^\frac13\|u_{1,0} \|_{L^{\infty}H^1}^2\|\partial_x^2 c_{\neq}\|_{X_a}^2.
			\end{aligned}
		\end{equation*}	
		The same argument as in \eqref{c_x-1} gives
		\begin{align*}
			\|e^{aA^{-\frac13}(t-T^*)} \partial_x u_{2,\neq}\partial_yc_0\|^2_{L^2L^2}
			\leq&~ C\|\partial_yc_0\|_{L^{\infty}L^2}^2
			\|e^{aA^{-\frac13}(t-T^*)} \partial_x u_{2,\neq}\|_{L^2L^2}
			\|e^{aA^{-\frac13}(t-T^*)} \partial_x \nabla u_{2,\neq}\|_{L^2L^2}\\
			\leq&~ CA^{\frac12}\|\partial_yc_0\|_{L^{\infty}L^2}^2\|\partial_x\omega_{\neq}\|_{X_a}^2.
		\end{align*}
		Substituting these estimates into \eqref{equc2}, we find
		\begin{equation*}
			\begin{aligned}
				\|\partial_x\partial_y c_{\neq}\|_{X_a}
				&\leq C\bigg(\|\partial_x\partial_y c_{\neq}(T^*)\|_{L^2}
				+A^{\frac13}\|\partial_x^2 c_{\neq}\|_{X_a}
				+\frac{\|\partial_x n_{\neq}\|_{X_a}}{A^{\frac13}}\\
				&\quad+\frac{E_2\|\partial_x\nabla c_{\neq}\|_{X_a}+\|u_{1,0}\|_{L^{\infty}H^1}\|\partial_x^2c_{\neq}\|_{X_a}}{A^{\frac13}}+\frac{\|\partial_yc_0\|_{L^{\infty}L^2}\|\partial_x\omega_{\neq}\|_{X_a}}{A^{\frac14}}\bigg),
			\end{aligned}
		\end{equation*}
		which along with \eqref{cx2} shows that 
		\begin{equation}\label{equc3}
			\begin{aligned}
				\|\partial_x\nabla c_{\neq}\|_{X_a}
				\leq&~ C\bigg(\|\partial_x\partial_y c_{\neq}(T^*)\|_{L^2}
				+A^{\frac13}\|\partial_x^2 c_{\neq}(T^*)\|_{L^2}
				+\frac{\|\partial_x n_{\neq}\|_{X_a}}{A^{\frac16}}\\
				&+\frac{\|\partial_yc_0\|_{L^{\infty}L^2}+\|u_{1,0}\|_{L^{\infty}H^1}+E_2}{A^{\frac14}}\left(\|\partial_x^2 c_{\neq}\|_{X_a}+\|\partial_x\nabla c_{\neq}\|_{X_a}+\|\partial_x\omega_{\neq}\|_{X_a}\right)\bigg).
			\end{aligned}
		\end{equation}
		By \eqref{nwc1} and Proposition \ref{prop:short-time}, we have
		\begin{equation}\label{xyxc2}
			\begin{aligned}
				\|\partial_x\partial_y c_{\neq}(T^*)\|_{L^2}
				+A^{\frac13}\|\partial_x^2 c_{\neq}(T^*)\|_{L^2}
				\leq C\left(\|n_{\rm in}\|_{H^2} +\|(u_{\rm in})_{\neq}\|_{H^3}
				+\|(c_{\rm in})_{\neq}\|_{H^2}\right) e^{C\left(\|c_{\rm in}\|_{H^2}^2+1\right)}.
			\end{aligned}
		\end{equation}
		Therefore, combining \eqref{xyxc2} with \eqref{equc2} and \eqref{equc3}, we conclude that 
		\begin{equation}\label{rc1}
			\begin{aligned}
				\|\partial_x\nabla c_{\neq}\|_{X_a}
				\leq&~ C\bigg(\left(\|n_{\rm in}\|_{H^2} +\|(u_{\rm in})_{\neq}\|_{H^3}
				+\|(c_{\rm in})_{\neq}\|_{H^2}\right)e^{C\left(\|c_{\rm in}\|_{H^2}^2+1\right)}
				+\frac{\|\partial_x n_{\neq}\|_{X_a}}{A^{\frac16}}\\
				&+\frac{\|\partial_yc_0\|_{L^{\infty}L^2}+\|u_{1,0}\|_{L^{\infty}H^1}+E_2}{A^{\frac16}}\left(\|\partial_x\nabla c_{\neq}\|_{X_a}+\|\partial_x\omega_{\neq}\|_{X_a}\right)\bigg).
			\end{aligned}
		\end{equation}
		\underline{\textbf{III. Estimate of $\|\partial_x\omega_{\neq}\|_{X_a}$.}}
		Differentiating $\eqref{eq:non1}_3$ with respect to $x$, we obtain
		\begin{equation*}
			\mathcal L_A\partial_x\omega_{\neq}
			=-\frac{\partial_x}{A}\big[\partial_xn_{\neq}
			+u_{1,0}\partial_x\omega_{\neq}
			+u_{2,\neq}\partial_y\omega_0
			+(u_{\neq}\cdot\nabla\omega_{\neq})_{\neq}
			\big],
		\end{equation*}
		Applying Corollary \ref{cor:physical-timespace} to this equation, we have 
		\begin{equation}\label{xw1}
			\begin{aligned}
				\|\partial_x\omega_{\neq}\|_{X_a}^2
				\leq&~C\bigg(\|\partial_x \omega_{\neq}(T^*)\|_{L^2}^2
				+\frac{\|e^{aA^{-\frac13}(t-T^*)} \partial_x n_{\neq}\|^2_{L^2L^2}}{A}		
				+\frac{\|e^{aA^{-\frac13}(t-T^*)}\partial_x(u_{\neq} \omega_{\neq})\|^2_{L^2L^2}}{A}\\
				&+\frac{\|e^{aA^{-\frac13}(t-T^*)} u_{1,0} \partial_x\omega_{\neq} \|^2_{L^2L^2}
					+\|e^{aA^{-\frac13}(t-T^*)} \partial_x u_{\neq}\omega_0\|^2_{L^2L^2}}{A}\bigg).
			\end{aligned}
		\end{equation}
		It follows from $\eqref{fg-L2}_2$ that
		\begin{align*}
			\|e^{aA^{-\frac13}(t-T^*)} u_{1,0} \partial_x \omega_{\neq}\|^2_{L^2L^2}
			\leq C\|u_{1,0}\|_{L^{\infty}H^1}^2\|e^{aA^{-\frac13}(t-T^*)}\partial_x \omega_{\neq} 	\|^2_{L^2L^2}
			\leq CA^{\frac13}\|u_{1,0}\|_{L^{\infty}H^1}^2
			\|\partial_x\omega_{\neq}\|_{X_a}^2.
		\end{align*}
		Using $\eqref{fg-L2}_1$, H\"older's inequality, and Lemma \ref{Poincare2}, we obtain
		\begin{align*}
			\|e^{aA^{-\frac13}(t-T^*)}\partial_x u_{\neq}\omega_0\|^2_{L^2L^2}
			\leq&~C\|\omega_0\|_{L^{\infty}L^2}^2
			\|e^{aA^{-\frac13}(t-T^*)}\partial_x u_{\neq}\|_{L^2L^2}\|e^{aA^{-\frac13}(t-T^*)}\partial_x\nabla u_{\neq}\|_{L^2L^2}\\
			\leq&~C\|\omega_0\|_{L^{\infty}L^2}^2\|e^{aA^{-\frac13}(t-T^*)}\partial_x\omega_{\neq}\|_{L^2L^2}^2\\
			\leq&~CA^{\frac13}\|\omega_0\|_{L^{\infty}L^2}^2
			\|\partial_x\omega_{\neq}\|_{X_a}^2.
		\end{align*}
		Arguing as in the estimate of \eqref{n_x-5}, we find
		\begin{align*}
			\|e^{aA^{-\frac13}(t-T^*)}\partial_x(u_{\neq} \omega_{\neq})\|^2_{L^2L^2}
			\leq&~C\|e^{aA^{-\frac13}(t-T^*)}\partial_x\omega_{\neq}\|^2_{L^{\infty}L^2}
			\|e^{aA^{-\frac13}(t-T^*)}\partial_xu_{\neq}\|_{L^2L^2}\|e^{aA^{-\frac13}(t-T^*)}\partial_x\nabla u_{\neq}\|_{L^2L^2}\\
			\leq&~C\|\partial_x\omega_{\neq}\|_{X_a}^2\|e^{aA^{-\frac13}(t-T^*)}\partial_x\omega_{\neq}\|_{L^2L^2}^2
			\leq CE_2^2A^{\frac13}\|\partial_x\omega_{\neq}\|_{X_a}^2.
		\end{align*}
		Substituting these estimates into \eqref{xw1}, we obtain 
		\begin{equation}\label{rw1}
			\begin{aligned}
				\|\partial_x\omega_{\neq}\|_{X_a} 
				\leq C\Big(\|\partial_x \omega_{\neq}(T^*)\|_{L^2}
				+\frac{\|\partial_x n_{\neq} \|_{X_a}}{A^{\frac13}}		
				+\frac{\|u_{1,0}\|_{L^{\infty}H^1}
					+\|\omega_0\|_{L^{\infty}L^2}+E_2}{A^{\frac13}}\|\partial_x\omega_{\neq}\|_{X_a}\Big).
			\end{aligned}
		\end{equation}
		Combining \eqref{rn1}, \eqref{rc1}, and \eqref{rw1}, we deduce that
		\begin{align*}
			\mathcal{E}_3(t)
			\leq&~C\bigg(\mathcal{E}_{3,\rm in}
			+\frac{1+E_2+\|\partial_y c_0\|_{L^{\infty}L^2}+\|n_0\|_{L^{\infty}L^2}+\|u_{1,0}\|_{L^{\infty}H^1}}{A^{\frac16}}\|\partial_xn_{\neq}\|_{X_a}\\
			&+\frac{E_2+\|\partial_y c_0\|_{L^{\infty}L^2}+\|n_0\|_{L^{\infty}L^2}+\|u_{1,0}\|_{L^{\infty}H^1}}{A^{\frac16}}\|\partial_x\nabla c_{\neq}\|_{X_a}\\
			&+\frac{E_2+\|n_0\|_{L^{\infty}L^2}+\|\partial_y c_0\|_{L^{\infty}L^2}+\|u_{1,0}\|_{L^{\infty}H^1}+\|\omega_0\|_{L^{\infty}L^2}}{A^{\frac16}}\|\partial_x\omega_{\neq}\|_{X_a}\bigg)\\
			\leq&~C\left(\mathcal{E}_{3,\rm in}+\frac{1+E_2+B_1+B_2+B_3+B_4}{A^{\frac16}}\mathcal{E}_{3}(t)\right). \alabel{E3}
		\end{align*}
		
		\begin{revision}
			For
			\[
			A\geq [2C(1+E_2+B_1+B_2+B_3+B_4)]^6=:A_5,
			\]
			the last term in \eqref{E3} can be absorbed. Hence,
			\[
			\mathcal{E}_3(t)\leq 2C\mathcal{E}_{3,\rm in}=:E_2.
			\]
			Finally, taking
			\[
			\tilde{A}:=\max\{A_3,A_4,A_5\},
			\]
			and using Lemma \ref{lem:long-time-zero-modes}, we complete the proof
			of Proposition \ref{prop:long-time}.

		\end{revision}
	\end{proof}

	\section{Proof of Proposition \ref{prop:infty-estimate}}
	\label{sec:uniform-bounds}
	\begin{revision}

		Throughout this section, let $[0,T]$ be an interval of existence. We use the short-time estimates on $[0,\min\{T,T^*\}]$ and the long-time estimates on $[T^*,T]$ when $T>T^*$. For $T\leq T^*$, the terms involving long-time bounds are omitted. For brevity, set
	\end{revision}
	\[
	Q:=E_1^{\frac12}+E_2+1.
	\]
	By \eqref{relation}, the estimates in Lemma \ref{lem:zero mode-kelxin}
	also hold in physical coordinates. We first estimate $\|n\|_{L^{\infty}L^{\infty}}$ and $\|u\|_{L^{\infty}L^{\infty}}$. We then use these bounds and an estimate for $\|\omega\|_{L^\infty L^4}$ to control $\|c\|_{L^{\infty}W^{1,\infty}}$. 
	

	\subsection{\texorpdfstring{$L^\infty$ estimates for $n$}
		{L-infinity estimates for n}}
	\label{subsec:n-infinity}
	

	\begin{lemma}\label{lem:n-infty}
		Under the conditions of Proposition \ref{prop:infty-estimate},
		\rev{there exists a positive constant $D_1$, depending only on
			$A$, $M$, $\|n_{\rm in}\|_{H^2}$, $\|c_{\rm in}\|_{H^2}$ and
			$\|u_{\rm in}\|_{H^3}$ 
			such that}
		\begin{align*}
			\|n\|_{L^{\infty}L^{\infty}}
			\leq D_1.
		\end{align*}
	\end{lemma}
	\begin{proof}
		For $p=2^j, j\geq1$, testing $\eqref{eq:1.7}_1$ with $2pn^{2p-1}$ yields
		\begin{align*}
			\frac{d}{dt}\|n^p\|_{L^2}^2+\frac{2(2p-1)}{pA}\|\nabla n^p\|_{L^2}^2
			=&~\frac{2(2p-1)}{A}\int_{\mathbb{T}\times\mathbb{R}}n^p\nabla c\cdot \nabla n^pdxdy\\
			\leq&~\frac{2p-1}{2pA}\|\nabla n^p\|_{L^2}^2+\frac{2p(2p-1)}{A}\|n^p\nabla c\|_{L^2}^2.\alabel{ineq:n_1}
		\end{align*}
		By H\"older's inequality, the one-dimensional Gagliardo--Nirenberg inequality, and Lemma~\ref{Poincare2}, we have 
		\begin{align*}
			\|n^p\nabla c\|_{L^2}^2
			\leq&~\|n^p\|_{L_x^2L_y^{\infty}}^2\|\nabla c\|_{L_x^{\infty}L_y^2}^2
			\leq C\|n^p\|_{L^2}\|\partial_yn^p\|_{L^2}\left(\|\partial_yc_0\|_{L_x^{\infty}L_y^2}+\|\nabla c_{\neq}\|_{L_x^{\infty}L_y^2}\right)^2\\
			\leq&~C\|n^p\|_{L^2}\|\partial_yn^p\|_{L^2}\left(\|\partial_yc_0\|_{L^2}+\|\nabla c_{\neq}\|_{L_x^{\infty}L_y^2}\right)^2\\
			\leq&~C\|n^p\|_{L^2}\|\partial_yn^p\|_{L^2}\left(\|\partial_yc_0\|_{L^2}+\|\partial_x\nabla c_{\neq}\|_{L^2}\right)^2.
		\end{align*}
		The Nash inequality gives
		\begin{align*}
			-\|\nabla n^p\|_{L^2}^2
			\leq-\frac{\|n^p\|_{L^2}^4}{C\|n^p\|_{L^1}^2}+\|n^p\|_{L^1}^2.
		\end{align*}
		Using these estimates and Young's inequality in \eqref{ineq:n_1}, we obtain
		\begin{align*}
			\frac{d}{dt}\|n^p\|_{L^2}^2
			\leq&-\frac{2p-1}{pA}\|\nabla n^p\|_{L^2}^2+\frac{Cp^3(2p-1)}{A}\|n^p\|_{L^2}^2\left(\|\partial_yc_0\|_{L^{\infty}L^2}+\|\partial_x\nabla c_{\neq}\|_{L^2}\right)^4\\
			\leq&-\frac{\|n^p\|_{L^2}^4}{CA\|n^p\|_{L^1}^2}+\frac{\|n^p\|_{L^1}^2}{A}+\frac{Cp^4}{A}\|n^p\|_{L^2}^2\left(1+\|\partial_yc_0\|_{L^{\infty}L^2}+\|\partial_x\nabla c_{\neq}\|_{L^2}\right)^4.
		\end{align*}
		\rev{
			Applying Lemma \ref{lem:uniform-iteration} to the above inequality, we deduce that}
		\begin{align*}
			\|n\|_{L^\infty L^\infty}
			\leq C\big(1+\|\partial_y c_0\|_{L^\infty L^2}
			+\|\partial_x\nabla c_{\neq}\|_{L^\infty L^2}\big)^2\big(\|n\|_{L^\infty L^2}+\|n_{\rm in}\|_{L^\infty}+1\big). \alabel{n-infty}
		\end{align*}
		By \eqref{prop-smalltime}, Proposition \ref{prop:long-time}, and the definition of $X_a$, we have
		\begin{align*}
			\|\partial_x\nabla c_{\neq}\|_{L^\infty L^2}
			\leq CE_1^{\frac12}\sup_{0\leq t\leq T_-}
			(1+t^2)^{\frac12}e^{-aA^{-\frac13}t}
			+\|\partial_x\nabla c_{\neq}\|_{X_a}
			\leq CQ\left(A^{\frac13}+1\right). \alabel{n-infty1}
		\end{align*}
		It follows from Propositions \ref{prop:small time-main}
		and \ref{prop:long-time}, Lemma \ref{lem:zero mode-kelxin} and
		Lemma \ref{lem:long-time-zero-modes}, with the interval convention above, that
		\begin{align*}
			&\|\partial_y c_0\|_{L^\infty L^2}
			\leq C(J_3+J_6+B_3),\\
			&\|n\|_{L^\infty L^2}
			\leq|\mathbb{T}|^{\frac12}\|n_0\|_{L^\infty L_y^2}+\|n_{\neq}\|_{L^{\infty}L^2}
			\leq C\left(E_1^{\frac12}+B_4+E_2\right)
			\leq C\left(Q+B_4\right).\alabel{n-infty2}
		\end{align*}
		Substituting \eqref{n-infty1} and \eqref{n-infty2} into \eqref{n-infty}, we conclude that
		\begin{align*}
			\|n\|_{L^{\infty}L^{\infty}}
			\leq C\big(J_3+J_6+B_3+Q^2+A^{\frac23}\big)^2\left(Q+B_4+\|n_{\rm in}\|_{L^{\infty}}\right)=:D_1.\alabel{n-D1}
		\end{align*}
		This completes the proof.
	\end{proof}

	\subsection{\texorpdfstring{$L^{\infty}$ estimates for $u$}{L-infinity estimates for u}}
	\label{subsec:velocity-estimates}
	\begin{lemma}\label{lem:u-infty}
		Under the conditions of Proposition \ref{prop:infty-estimate},
		\rev{there exists a positive constant $D_2$, depending only on
			$A$, $M$, $\|n_{\rm in}\|_{H^2}$, $\|c_{\rm in}\|_{H^2}$ and
			$\|u_{\rm in}\|_{H^3}$
			such that}
		\begin{align*}
			\|u\|_{L^\infty L^\infty}
			\leq D_2.
		\end{align*}
	\end{lemma}
	
	\begin{proof}
		Using the one-dimensional Gagliardo--Nirenberg inequality, Lemma \ref{Poincare2} and \eqref{u=w-xy}, we obtain
		\begin{align*}
			\|u_{\neq}\|_{L^{\infty}}
			\leq&~ \|u_{\neq}\|_{L_x^2L_y^{\infty}}^{\frac12}\|\partial_xu_{\neq}\|_{L_x^2L_y^{\infty}}^{\frac12}
			\leq C\|u_{\neq}\|_{L^2}^{\frac14}\|\partial_yu_{\neq}\|_{L^2}^{\frac14}\|\partial_xu_{\neq}\|_{L^2}^{\frac14}\|\partial_y\partial_xu_{\neq}\|_{L^2}^{\frac14}\\
			\leq&~C\|\partial_x\nabla u_{\neq}\|_{L^2}
			=C\|\partial_x\omega_{\neq}\|_{L^2}.\alabel{u-neq-infty}
		\end{align*}
		Since $u_0=(u_{1,0},0)$, it follows from \eqref{u-neq-infty} that
		\begin{align*}
			\|u\|_{L^\infty L^\infty}
			&\leq C\left(\|u_{1,0}\|_{L^{\infty}L_y^{\infty}}+\|u_{\neq}\|_{L^{\infty}L^{\infty}}\right)\\
			&\leq C\Big(\|u_{1,0}\|_{L^\infty H_y^1}
			+\|\partial_x\omega_{\neq}\|_{L^\infty L^2}\Big).
		\end{align*}
		By $\eqref{zero-mode eatimate-small}_1$ and $\eqref{zero-mode eatimate-long}_2$, we have
		\[
		\|u_{1,0}\|_{L^\infty H_y^1}
		\leq C(J_1+B_2).
		\]
		Moreover, Propositions~\ref{prop:small time-main}
		and \ref{prop:long-time} imply that
		\begin{align*}
			\|\partial_x\omega_{\neq}\|_{L^\infty L^2}
			\leq CE_1^{\frac12}
			+\|\partial_x\omega_{\neq}\|_{X_a}
			\leq CQ.
		\end{align*}
		Substituting these bounds into the above estimate for $u$, we conclude that
		\begin{align*}
			\|u\|_{L^\infty L^\infty}
			\leq C(J_1+B_2+Q):=D_2,
		\end{align*}
		which completes the proof.
	\end{proof}


	
	\subsection{\texorpdfstring{$L^\infty$ estimates for $c$ and $\nabla c$}
		{L-infinity estimates for c}}
	\label{subsec:c-infinity}
	
	
	\begin{lemma}\label{lem:c-infty}
		Under the conditions of Proposition \ref{prop:infty-estimate},
		\rev{there exists a positive constant $D_3$, depending only on
			$A$, $M$, $\|n_{\rm in}\|_{H^2}$, $\|c_{\rm in}\|_{H^2}$ and
			$\|u_{\rm in}\|_{H^3}$
			such that}
		\begin{align*}
			\|c\|_{L^{\infty}L^{\infty}}
			\leq D_3.
		\end{align*}
	\end{lemma}
	\begin{proof}
		Let $p=2^j$ with $j\geq1$. Taking the $L^2$ inner product of $\eqref{eq:1.7}_2$ with $2pc^{2p-1}$ and applying Young's inequality, we find
		\begin{align*}
			\frac{d}{dt}\|c^p\|_{L^2}^2+\frac{2(2p-1)}{pA}\|\nabla c^p\|_{L^2}^2+\frac{2p}{A}\|c^p\|_{L^2}^2
			=&\frac{2p}{A}\int_{\mathbb{T}\times\mathbb{R}}nc^{2p-1}dxdy
			\leq \frac{1}{A}\|n^p\|_{L^2}^2+\frac{2p-1}{A}\|c^p\|_{L^2}^2.
		\end{align*}
		Discarding the nonnegative gradient term, we deduce that
		\begin{align*}
			\frac{d}{dt}\|c^p\|_{L^2}^2+\frac{1}{A}\|c^p\|_{L^2}^2
			\leq\frac{1}{A}\|n^p\|_{L^2}^2.
		\end{align*}
		Multiplying both sides of the above inequality by $e^{\frac{t}{A}}$ and integrating over $[0,t]$, we obtain
		\begin{align*}
			\|c^p\|_{L^2}^2
			\leq&~e^{-\frac{t}{A}}\|c^p_{\rm in}\|_{L^2}^2+\frac{1}{A}e^{-\frac{t}{A}}\int_0^te^{\frac{s}{A}}\|n^p(s)\|_{L^2}^2ds
			\leq e^{-\frac{t}{A}}\|c^p_{\rm in}\|_{L^2}^2+\left(1-e^{-\frac{t}{A}}\right)\mathop{\sup}\limits_{0\leq t\leq T}\|n^p\|_{L^2}^2.
		\end{align*}
		Consequently,
		\begin{align*}
			\mathop{\sup}\limits_{0\leq t\leq T}\|c^p\|_{L^2}^2
			\leq \max\left\{\|c_{\rm in}^p\|_{L^2}^2,\mathop{\sup}\limits_{0\leq t\leq T}\|n^p\|_{L^2}^2\right\}.
		\end{align*}
		Taking both sides of the above inequality to the power $\frac{1}{2p}$ and letting $p\to\infty$, we conclude that
		\begin{align*}
			\|c\|_{L^{\infty}L^{\infty}}
			\leq \max\left\{\|c_{\rm in}\|_{L^{\infty}},\|n\|_{L^{\infty}L^{\infty}}\right\}
			\leq\|c_{\rm in}\|_{L^{\infty}}+\|n\|_{L^{\infty}L^{\infty}}=:D_3.
		\end{align*}
	\end{proof}
	It remains to estimate $\|\nabla c\|_{L^\infty L^\infty}$. To this end, we first establish a bound for $\|\omega\|_{L^\infty L^4}$.

	\begin{lemma}
		Under the assumptions of Proposition \ref{prop:infty-estimate},
		\rev{there exists a positive constant $D_4$, depending only on
			$A$, $M$, $\|n_{\rm in}\|_{H^2}$, $\|c_{\rm in}\|_{H^2}$ and
			$\|u_{\rm in}\|_{H^3}$
			such that}
		\begin{align*}
			\|\omega\|_{L^{\infty}L^4}
			\leq D_4.
		\end{align*}
	\end{lemma}
	
	\begin{proof}
		Multiplying $\eqref{eq:1.7}_3$ by $4\omega^3$, using integration by parts and Young's inequality, we obtain
		\begin{align*}
			&\frac{d}{dt}\|\omega^2\|_{L^2}^2+\frac{3}{A}\|\nabla \omega^2\|_{L^2}^2\\
			=&-\frac{4}{A}\int_{\mathbb{T}\times\mathbb{R}}\omega^{3}\partial_xn dxdy
			\leq\frac{3}{2A}\|\nabla \omega^2\|_{L^2}^2+\frac{C}{A}\|\omega n\|_{L^2}^2
			\leq\frac{3}{2A}\|\nabla \omega^2\|_{L^2}^2
			+\frac{C}{A}\|\omega\|_{L^2}^{2}\|n\|_{L^{\infty}}^2.
			\alabel{inequ:omega}
		\end{align*}
		The Nash inequality implies that
		\begin{align*}
			-\|\nabla \omega^2\|_{L^2}^2
			\leq-\frac{\|\omega^2\|_{L^2}^4}{C\|\omega^2\|_{L^1}^2}+\|\omega^2\|_{L^1}^2.
			\alabel{Nash2}
		\end{align*}
		Combining \eqref{inequ:omega} with \eqref{Nash2}, we deduce that
		\begin{align*}
			\frac{d}{dt}\|\omega^2\|_{L^2}^2
			\leq&-\frac{3}{2A}\|\nabla \omega^2\|_{L^2}^2
			+\frac{C}{A}\|\omega\|_{L^{\infty}L^2}^{2}\|n\|_{L^{\infty}L^{\infty}}^2\\
			\leq&-\frac{\|\omega^2\|_{L^2}^4}{CA\|\omega^2\|_{L^1}^2}+\frac{\|\omega^2\|_{L^1}^2}{A}
			+\frac{C}{A}\|\omega\|_{L^{\infty}L^2}^{2}\|n\|_{L^{\infty}L^{\infty}}^2.
		\end{align*}
		Using the comparison argument in the proof of \eqref{eq:iteration-recursion}, we obtain
		\begin{align*}
			\|\omega^2\|_{L^{\infty}L^2}^2
			\leq\max\left\{\|\omega_{\rm in}\|_{L^4}^4,
			C\|\omega\|_{L^{\infty}L^2}^4
			\left(1+\|n\|_{L^{\infty}L^{\infty}}^2\right),1\right\}.
		\end{align*}
		Taking fourth roots and applying Young's inequality, we find
		\begin{align*}
			\|\omega\|_{L^\infty L^4}
			\leq C\left(1+\|\omega_{\rm in}\|_{L^4}
			+\|\omega\|_{L^\infty L^2}^2
			+\|n\|_{L^\infty L^\infty}\right).
			\alabel{omega-4}
		\end{align*}
		By Propositions~\ref{prop:small time-main} and \ref{prop:long-time},
		Lemma~\ref{Poincare2}, $\eqref{zero-mode eatimate-small}_2$,
		and $\eqref{zero-mode eatimate-long}_1$, we have
		\begin{align*}
			\|\omega_0\|_{L^\infty L_y^2}
			&\leq C(J_2+B_1),\\
			\|\omega_{\neq}\|_{L^\infty L^2}
			&\leq CE_1^{\frac12}
			+\|\partial_x\omega_{\neq}\|_{X_a}
			\leq CQ.
		\end{align*}
		Since $\omega=\omega_0+\omega_{\neq}$, it follows that
		\begin{align*}
			\|\omega\|_{L^\infty L^2}
			\leq |\mathbb T|^{\frac12}\|\omega_0\|_{L^\infty L_y^2}
			+\|\omega_{\neq}\|_{L^\infty L^2}
			\leq C(J_2+B_1+Q).
		\end{align*}
		Substituting this bound and \eqref{n-D1} into \eqref{omega-4}, we conclude that
		\begin{align*}
			\|\omega\|_{L^\infty L^4}
			\leq C\left[\|\omega_{\rm in}\|_{L^4}
			+\left(J_2+B_1+Q\right)^2+D_1\right]=:D_4.
		\end{align*}
	\end{proof}
	
	\begin{lemma}\label{lem:c-xy-infty}
		Under the assumptions of Proposition \ref{prop:infty-estimate},
		\rev{there exists a positive constant $D_5$, depending only on
			$A$, $M$, $\|n_{\rm in}\|_{H^2}$, $\|c_{\rm in}\|_{H^2}$,
			$\|u_{\rm in}\|_{H^3}$ and $\|\nabla c_{\rm in}\|_{L^\infty}$
			such that}
		\begin{align*}
			\|\nabla c\|_{L^{\infty}L^{\infty}}
			\leq D_5.
		\end{align*}
	\end{lemma}
	\begin{proof}
		\underline{\textbf{I. Estimate of $\partial_x c$.}}
		Differentiating $\eqref{eq:1.7}_2$ with respect to $x$, we obtain
		\begin{align*}
			\partial_t\partial_xc+y\partial_x^2c-\frac1A\Delta\partial_xc+\frac{1}{A}u\cdot\nabla\partial_xc+\frac{1}{A}\partial_xc
			=\frac{1}{A}\partial_xn-\frac{1}{A}\partial_xu\cdot\nabla c.
		\end{align*}
		For $p=2^j$ with $j\geq1$, testing this equation with $2p(\partial_xc)^{2p-1}$ and using integration by parts, we find
		\begin{align*}
			&\quad\frac{d}{dt}\|(\partial_xc)^p\|_{L^2}^2+\frac{2(2p-1)}{pA}\|\nabla(\partial_xc)^p\|_{L^2}^2+\frac{2p}{A}\|(\partial_xc)^p\|_{L^2}^2\\
			&=\frac{2p}{A}\int_{\mathbb{T}\times\mathbb{R}}\partial_xn(\partial_xc)^{2p-1}dxdy
			-\frac{2p}{A}\int_{\mathbb{T}\times\mathbb{R}}\nabla\cdot(c\partial_xu)(\partial_xc)^{2p-1}dxdy\\
			&=-\frac{2(2p-1)}{A}\int_{\mathbb{T}\times\mathbb{R}}\partial_x(\partial_xc)^p n(\partial_xc)^{p-1}dxdy
			+\frac{2(2p-1)}{A}\int_{\mathbb{T}\times\mathbb{R}}\nabla(\partial_xc)^p c\partial_xu(\partial_xc)^{p-1}dxdy\\
			&\leq\frac{2p-1}{2pA}\|\nabla(\partial_xc)^p\|_{L^2}^2+\frac{Cp^2}{A}\left(\|n\|_{L^4}^2+\|c\|_{L^{\infty}}^2\|\partial_xu\|_{L^4}^2\right)\|\partial_xc\|_{L^{4p-4}}^{2p-2}.\alabel{c_x}
		\end{align*}
		By the Gagliardo--Nirenberg inequality, we have
		\begin{align*}
			\|\partial_xc\|_{L^{4p-4}}^{2p-2}\leq C\|\nabla(\partial_xc)^p\|_{L^2}^{\frac{p-2}{p}}\|(\partial_xc)^p\|_{L^2}+C\|(\partial_xc )^p\|_{L^2}^{2-\frac{2}{p}}.\alabel{c_x-GN}
		\end{align*}
		The Nash inequality implies that
		\begin{align*}
			-\|\nabla (\partial_xc)^p\|_{L^2}^2
			\leq-\frac{\|(\partial_xc)^p\|_{L^2}^4}{C\|(\partial_xc)^p\|_{L^1}^2}+\|(\partial_xc)^p\|_{L^1}^2.
			\alabel{Nash3}
		\end{align*}
		\rev{By the threshold formulation of Lemma
			\ref{lem:uniform-iteration}, it suffices to establish the
			following differential inequality at times when
			$\|(\partial_xc)^p\|_{L^2}^2>1$.}
		Using Young's inequality, \eqref{c_x-GN}, and \eqref{Nash3}, we deduce from \eqref{c_x} that
		\begin{align*}
			&\frac{d}{dt}\|(\partial_xc)^p\|_{L^2}^2+\frac{2p}{A}\|(\partial_xc)^p\|_{L^2}^2\\
			\leq&-\frac{1}{A}\|\nabla(\partial_xc)^p\|_{L^2}^2+\frac{Cp^4}{A}\left(\|n\|_{L^{\infty}L^4}^2+\|c\|_{L^{\infty}L^{\infty}}^2\|\partial_xu\|_{L^{\infty}L^4}^2\right)^{\frac{2p}{p+2}}\|(\partial_xc)^p\|_{L^2}^{\frac{2p}{p+2}}\\
			&+\frac{Cp^2}{A}\left(\|n\|_{L^{\infty}L^4}^2+\|c\|_{L^{\infty}L^{\infty}}^2\|\partial_xu\|_{L^{\infty}L^4}^2\right)\|(\partial_xc)^p\|_{L^2}^{2-\frac{2}{p}}\\
			\leq&-\frac{\|(\partial_xc)^p\|_{L^2}^4}{CA\|(\partial_xc)^p\|_{L^1}^2}+\frac{\|(\partial_xc)^p\|_{L^1}^2}{A}\\
			&+\frac{Cp^4}{A}\left(\|n\|_{L^{\infty}L^4}^2+\|c\|_{L^{\infty}L^{\infty}}^2\|\partial_xu\|_{L^{\infty}L^4}^2+1\right)^2\|(\partial_xc)^p\|_{L^2}^{2}.
		\end{align*}
		By Lemma \ref{lem:uniform-iteration}, we have
		\begin{align*}
			\|\partial_x c\|_{L^{\infty}L^{\infty}}
			\leq C\left(\|n\|_{L^{\infty}L^4}^2+\|c\|_{L^{\infty}L^{\infty}}^2\|\partial_xu\|_{L^{\infty}L^4}^2+1\right)\left(\|\partial_xc\|_{L^{\infty}L^2}+\|(\partial_xc)_{\rm in}\|_{L^{\infty}}+1\right). \alabel{c_x-infty1}
		\end{align*}
		It follows from Proposition \ref{prop:small time-main}, Proposition \ref{prop:long-time} and Lemma \ref{Poincare2} that
		\begin{align*}
			\|\partial_xc\|_{L^{\infty}L^2}
			\leq&CE_1^{\frac12}+\|\partial_x\nabla c_{\neq}\|_{X_a}
			\leq CQ. \alabel{c_x-L2}
		\end{align*}
		Since $u=\nabla^\perp\Delta^{-1}\omega$, the Calder\'on--Zygmund
		estimate on $\mathbb T\times\mathbb R$ gives
		\begin{align*}
			\|\partial_xu\|_{L^{\infty}L^4}
			\leq C\|\omega\|_{L^{\infty}L^4}
			\leq C D_4.
		\end{align*}
		By the Gagliardo--Nirenberg inequality and \eqref{n-infty2}, we have
		\begin{align*}
			\|n\|_{L^{\infty}L^4}
			\leq \|n\|_{L^{\infty}L^2}^{\frac12}\|n\|_{L^{\infty}L^{\infty}}^{\frac12}
			\leq C\left(Q+B_4\right)^{\frac12}D_1^{\frac{1}{2}},\alabel{n-infty4}
		\end{align*}
		Substituting the above estimates into \eqref{c_x-infty1}, we conclude that
		\begin{align*}
			\|\partial_x c\|_{L^{\infty}L^{\infty}}
			\leq&~ C\left((Q+B_4)D_1+D_3^2 D_4^2+1\right)\left(Q+\|(\partial_xc)_{\rm in}\|_{L^{\infty}}+1\right)\\
			\leq&~ C\left(Q^2+B_4^2+D_1^2+D_3^2 D_4^2\right)\left(Q+\|(\partial_xc)_{\rm in}\|_{L^{\infty}}\right). \alabel{c_x-infty}
		\end{align*}
		\underline{\textbf{II. Estimate of $\partial_y c$.}}	
		Taking the derivative of $\eqref{eq:1.7}_2$ with respect to $y$, we find
		\begin{align*}
			\partial_t\partial_yc
			+y\partial_x\partial_yc
			+\frac{1}{A}u\cdot\nabla\partial_yc
			-\frac{1}{A}\Delta\partial_yc
			+\frac{1}{A}\partial_yc
			=
			\frac{1}{A}\partial_yn
			-\partial_xc
			-\frac{1}{A}\nabla\cdot(c\partial_yu).
		\end{align*}
		Multiplying both sides of the above equation by the test function
		$2p(\partial_yc)^{2p-1}$, integrating over
		$\mathbb{T}\times\mathbb{R}$, and using integration by parts, we obtain
		\begin{align*}
			&\quad\frac{d}{dt}\|(\partial_yc)^p\|_{L^2}^2
			+\frac{2(2p-1)}{pA}\|\nabla(\partial_yc)^p\|_{L^2}^2
			+\frac{2p}{A}\|(\partial_yc)^p\|_{L^2}^2\\
			&=
			\frac{2p}{A}\int_{\mathbb{T}\times\mathbb{R}}\partial_yn(\partial_yc)^{2p-1}dxdy
			-2p\int_{\mathbb{T}\times\mathbb{R}}\partial_xc(\partial_yc)^{2p-1}dxdy
			-\frac{2p}{A}\int_{\mathbb{T}\times\mathbb{R}}\nabla\cdot(c\partial_yu)(\partial_yc)^{2p-1}dxdy\\
			&=-\frac{2(2p-1)}{A}\int_{\mathbb{T}\times\mathbb{R}}\partial_y(\partial_yc)^p n(\partial_yc)^{p-1}dxdy
			-2p\int_{\mathbb{T}\times\mathbb{R}}\partial_xc(\partial_yc)^{2p-1}dxdy\\
			&\quad+\frac{2(2p-1)}{A}\int_{\mathbb{T}\times\mathbb{R}}\nabla(\partial_yc)^p\cdot c\partial_yu(\partial_yc)^{p-1}dxdy
			=:I_1+I_2+I_3.\alabel{I1+I2+I3}
		\end{align*}
		By H\"older's and Young's inequalities, we have
		\begin{align*}
			I_1
			\leq&~\frac{2p-1}{4pA}\|\partial_y(\partial_yc)^p\|_{L^2}^2+\frac{Cp(2p-1)}{A}\|n(\partial_yc)^{p-1}\|_{L^2}^2\\
			\leq&~\frac{2p-1}{4pA}\|\nabla(\partial_yc)^p\|_{L^2}^2+\frac{Cp^2}{A}\|n\|_{L^{\infty}L^4}^2\|\partial_yc\|_{L^{4p-4}}^{2p-2},
		\end{align*}
		and
		\begin{align*}
			I_3
			\leq&~\frac{2p-1}{4pA}\|\nabla(\partial_yc)^p\|_{L^2}^2+\frac{Cp(2p-1)}{A}\|c\partial_yu(\partial_yc)^{p-1}\|_{L^2}^2\\
			\leq&~\frac{2p-1}{4pA}\|\nabla(\partial_yc)^p\|_{L^2}^2+\frac{Cp^2}{A}\|c\|_{L^\infty L^\infty}^2\|\partial_yu\|_{L^\infty L^4}^2\|\partial_yc\|_{L^{4p-4}}^{2p-2}.
		\end{align*}
		The Gagliardo--Nirenberg inequality implies that
		\begin{align*}
			\|\partial_yc\|_{L^{4p-4}}^{2p-2}
			=\|(\partial_yc)^p\|_{L^{4-\frac4p}}^{2-\frac2p}
			\leq C\Big(\|\nabla(\partial_yc)^p\|_{L^2}^{1-\frac2p}\|(\partial_yc)^p\|_{L^2}+\|(\partial_yc)^p\|_{L^2}^{2-\frac2p}\Big).
		\end{align*}
		Adding the estimates for $I_1$ and $I_3$ and applying Young's inequality,
		we deduce that
		\begin{align*}
			I_1+I_3
			\leq&~\frac{2p-1}{2pA}\|\nabla(\partial_yc)^p\|_{L^2}^2+\frac{Cp^2}{A}\left(\|n\|_{L^{\infty}L^4}^2+\|c\|_{L^\infty L^\infty}^2\|\partial_yu\|_{L^\infty L^4}^2\right)\|\partial_yc\|_{L^{4p-4}}^{2p-2}\\
			\leq&~\frac{3(2p-1)}{4pA}\|\nabla(\partial_yc)^p\|_{L^2}^2+\frac{Cp^4}{A}\left(\|n\|_{L^{\infty}L^4}^2+\|c\|_{L^\infty L^\infty}^2\|\partial_yu\|_{L^\infty L^4}^2\right)^{\frac{2p}{p+2}}\|(\partial_yc)^p\|_{L^2}^{\frac{2p}{p+2}}\\
			&+\frac{Cp^2}{A}\left(\|n\|_{L^\infty L^4}^2
			+\|c\|_{L^\infty L^\infty}^2\|\partial_yu\|_{L^\infty L^4}^2\right)\|(\partial_yc)^p\|_{L^2}^{2-\frac2p}.\alabel{I1+I3}
		\end{align*}
		For $I_2$, applying the H\"older's inequality, the Gagliardo--Nirenberg inequality, the Young's inequality and \eqref{c_x-L2}, we obtain
		\begin{align*}
			I_2
			\leq&~ 2p\|\partial_xc\|_{L^2}\|\partial_yc\|_{L^{4p-2}}^{2p-1}
			\leq Cp\|\partial_xc\|_{L^{\infty}L^2}\Big(\|\nabla(\partial_yc)^p\|_{L^2}^{1-\frac1p}\|(\partial_yc)^p\|_{L^2}+\|(\partial_yc)^p\|_{L^2}^{2-\frac1p}\Big)\\
			\leq&~\frac{2p-1}{4pA}\|\nabla(\partial_yc)^p\|_{L^2}^2+CAp^{\frac{2p}{p+1}}Q^{\frac{2p}{p+1}}\|(\partial_yc)^p\|_{L^2}^{\frac{2p}{p+1}}+CpQ\|(\partial_yc)^p\|_{L^2}^{2-\frac1p}.\alabel{I2}
		\end{align*}
		We use the Nash inequality
		\begin{align*}
			-\|\nabla (\partial_yc)^p\|_{L^2}^2
			\leq-\frac{\|(\partial_yc)^p\|_{L^2}^4}{C\|(\partial_yc)^p\|_{L^1}^2}+\|(\partial_yc)^p\|_{L^1}^2.
			\alabel{Nash4}
		\end{align*}
		\rev{Again, Lemma \ref{lem:uniform-iteration} requires the
			differential inequality only at times when
			$\|(\partial_yc)^p\|_{L^2}^2>1$.}
		Substituting \eqref{I1+I3} and \eqref{I2} into \eqref{I1+I2+I3} and using \eqref{Nash4}, we deduce that
		\begin{align*}
			&\frac{d}{dt}\|(\partial_yc)^p\|_{L^2}^2+\frac{2p}{A}\|(\partial_yc)^p\|_{L^2}^2\\
			\leq&-\frac{2p-1}{pA}\|\nabla(\partial_yc)^p\|_{L^2}^2+CAp^2 Q^{\frac{2p}{p+1}}\|(\partial_yc)^p\|_{L^2}^{\frac{2p}{p+1}}+CpQ\|(\partial_yc)^p\|_{L^2}^{2-\frac1p}\\
			&+\frac{Cp^4}{A}\left(\|n\|_{L^\infty L^4}^2
			+\|c\|_{L^\infty L^\infty}^2\|\partial_yu\|_{L^\infty L^4}^2+1\right)^{2}\Big(\|(\partial_yc)^p\|_{L^2}^{\frac{2p}{p+2}}+\|(\partial_yc)^p\|_{L^2}^{2-\frac2p}\Big)\\
			\leq&-\frac{\|(\partial_yc)^p\|_{L^2}^4}{CA\|(\partial_yc)^p\|_{L^1}^2}+\frac{\|(\partial_yc)^p\|_{L^1}^2}{A}\\
			&+\frac{Cp^4}{A}\Big[\Big(\|n\|_{L^\infty L^4}^2
			+\|c\|_{L^\infty L^\infty}^2\|\partial_yu\|_{L^\infty L^4}^2+1\Big)^{2}+A^2Q^2\Big]\|(\partial_yc)^p\|_{L^2}^2.		\end{align*}
		By Lemma \ref{lem:uniform-iteration}, we have
		\begin{align*}
			\|\partial_yc\|_{L^{\infty}L^{\infty}}
			\leq C\Big[\left(\|n\|_{L^\infty L^4}^2
			+\|c\|_{L^\infty L^\infty}^2\|\partial_yu\|_{L^\infty L^4}^2+1\right)^{2}+A^2Q^2\Big]^{\frac12}\left(\|\partial_yc\|_{L^{\infty}L^2}+\|(\partial_yc)_{\rm in}\|_{L^{\infty}}+1\right).
		\end{align*}
		It follows from Lemma \ref{Poincare2} and \eqref{n-infty1} that
		\begin{align*}
			\|\partial_yc_{\neq}\|_{L^{\infty}L^2}
			\leq\|\partial_x\nabla c_{\neq}\|_{L^{\infty}L^2}
			\leq CQ\left(A^{\frac13}+1\right).
		\end{align*}
		By Lemma \ref{lem:zero mode-kelxin} and Lemma \ref{lem:long-time-zero-modes}, we have
		\begin{align*}
			\|\partial_yc_0\|_{L^{\infty}L^2}
			\leq C\left(J_3+J_6+B_3\right).
		\end{align*}
		Combining these estimates, we obtain
		\begin{align*}
			\|\partial_yc\|_{L^{\infty}L^2}
			\leq C\left(J_3+J_6+B_3+A^{\frac13}Q\right).
		\end{align*}
		\rev{The same Calder\'on--Zygmund estimate gives
			$\|\partial_yu\|_{L^\infty L^4}\leq CD_4$.
			Using \eqref{n-infty4},
			we conclude that}
		\begin{align*}
			\|\partial_yc\|_{L^{\infty}L^{\infty}}
			\leq&~C\big(Q^2+B_4^2+D_1^2+D_3^2 D_4^2+AQ\big)\big(J_3+J_6+B_3+A^{\frac23}+Q^2+\|(\partial_yc)_{\rm in}\|_{L^{\infty}}\big).\alabel{c_y-infty}
		\end{align*}
		Finally, \eqref{c_x-infty} and \eqref{c_y-infty} imply that
		\begin{align*}
			\|\nabla c\|_{L^{\infty}L^{\infty}}
			\leq&~ C\left(Q^2+B_4^2+D_1^2+D_3^2 D_4^2+AQ\right)\left(J_3+J_6+B_3+A^{\frac23}+Q^2+\|\nabla c_{\rm in}\|_{L^{\infty}}\right)=:D_5.
		\end{align*}
	\end{proof}
	
	\begin{proof}[\normalfont\bfseries Proof of Proposition \ref{prop:infty-estimate}]
		\begin{revision}
			The result follows from Lemma~\ref{lem:n-infty}, Lemma~\ref{lem:u-infty}, Lemma~\ref{lem:c-infty} and Lemma~\ref{lem:c-xy-infty}.
		\end{revision}
	\end{proof}
		
		\appendix
		\section{Moser--Alikakos iteration}\label{sec_A}
		\begin{lemma}\label{lem:uniform-iteration}
			Let $T>0$ and $A,K\geq1$. Suppose that
			$f\in C([0,T];L^2)$ with $f(0)=f_{\rm in}\in L^\infty$,
			where all spatial norms are taken over $\mathbb T\times\mathbb R$
			and all time norms over $[0,T]$.
			Assume that, for every $p=2^j$, $j\geq1$,
			the map $t\mapsto\|f^p(t)\|_{L^2}^2$ is continuous on $[0,T]$,
			differentiable on $(0,T)$, and satisfies
			\begin{equation}\label{eq:iteration-hypothesis}
				\frac{d}{dt}\|f^p(t)\|_{L^2}^2
				\leq \frac1A\left(
				-\frac{\|f^p(t)\|_{L^2}^4}
				{C\|f^p(t)\|_{L^1}^2}
				+\|f^p(t)\|_{L^1}^2
				+Cp^4K\|f^p(t)\|_{L^2}^2
				\right)
			\end{equation}
			\rev{for every $t\in(0,T)$ with $\|f^p(t)\|_{L^2}^2>1$
			}
			where $C\geq1$ is independent of $p$, $A$, $K$, and $T$.
			\rev{No differential inequality is required when
				$\|f^p(t)\|_{L^2}^2\leq1$.}
			Then
			\begin{equation}\label{eq:iteration-conclusion}
				\|f\|_{L^\infty L^\infty}
				\leq CK^{\frac12}\left(
				\|f\|_{L^\infty L^2}
				+\|f_{\rm in}\|_{L^\infty}+1
				\right),
			\end{equation}
			with a possibly larger constant $C$, independent of $A$, $K$, and $T$.
		\end{lemma}
		
		\begin{proof}
			Fix $p=2^j$, $j\geq1$, and suppose that
			$\|f\|_{L^\infty L^p}<\infty$.
			We claim that
			\begin{equation}\label{eq:iteration-recursion}
				\|f^p\|_{L^\infty L^2}^2
				\leq
				\max\left\{
				\|f_{\rm in}^p\|_{L^2}^2,\,
				2C^2p^4K\|f\|_{L^\infty L^p}^{2p},\,
				1
				\right\}.
			\end{equation}
			If \eqref{eq:iteration-recursion} did not hold, there would exist
			$\varepsilon>0$ and $t_1\in(0,T]$ such that
			\[
			\|f^p(t_1)\|_{L^2}^2
			>
			\max\left\{
			\|f_{\rm in}^p\|_{L^2}^2,\,
			2C^2p^4K\|f\|_{L^\infty L^p}^{2p},\,
			1
			\right\}+\varepsilon.
			\]
			Let $t_*\in(0,t_1)$ be the first time such that
			\[
			\|f^p(t_*)\|_{L^2}^2
			=
			\max\left\{
			\|f_{\rm in}^p\|_{L^2}^2,\,
			2C^2p^4K\|f\|_{L^\infty L^p}^{2p},\,
			1
			\right\}+\varepsilon.
			\]
			By continuity and the definition of $t_*$, we have
			\[
			\left.\frac{d}{dt}\|f^p(t)\|_{L^2}^2\right|_{t=t_*}
			\geq0.
			\]
			On the other hand,
			\begin{align*}
				\|f^p(t_*)\|_{L^2}^2
				>2C^2p^4K\|f\|_{L^\infty L^p}^{2p}
				\geq 2C^2p^4K\|f^p(t_*)\|_{L^1}^2.
			\end{align*}
			\rev{Since $\|f^p(t_*)\|_{L^2}^2>1$, the differential
				inequality is available at this first crossing, and
				$\|f^p(t_*)\|_{L^1}>0$.}
			Thus, \eqref{eq:iteration-hypothesis} gives
			\begin{align*}
				\left.\frac{d}{dt}\|f^p(t)\|_{L^2}^2\right|_{t=t_*}
				\leq&~\frac1A\Big(-\frac{\|f^p(t_*)\|_{L^2}^4}{C\|f^p(t_*)\|_{L^1}^2}
				+\|f^p(t_*)\|_{L^1}^2
				+Cp^4K\|f^p(t_*)\|_{L^2}^2\Big)\\
				\leq&~\frac{\|f^p(t_*)\|_{L^2}^2}{A}\Big(-\frac{\|f^p(t_*)\|_{L^2}^2}{C\|f^p(t_*)\|_{L^1}^2}
				+2Cp^4K\Big)
				<0,
			\end{align*}
			a contradiction. Hence, \eqref{eq:iteration-recursion} holds.
			
			Set
			$B:=\|f\|_{L^\infty L^2}+\|f_{\rm in}\|_{L^\infty}+1.$
			By interpolation,
			\[
			\|f_{\rm in}^p\|_{L^2}^2
			\leq\|f_{\rm in}\|_{L^\infty}^{2p-2}\|f_{\rm in}\|_{L^2}^2
			\leq B^{2p}.
			\]
			Raising both sides of \eqref{eq:iteration-recursion}
			to the power $\frac1{2p}$, with $p=2^j$, we obtain
			\begin{align*}
				\|f\|_{L^\infty L^{2^{j+1}}}
				\leq\max\{B,\left(2C^2p^4K\right)^{\frac1{2^{j+1}}}\|f\|_{L^\infty L^{2^j}},1\}
				\leq\left(2C^2 16^jK\right)^{\frac1{2^{j+1}}}\max\left\{B,\,\|f\|_{L^\infty L^{2^j}},\,1\right\}.
			\end{align*}
			Since $B\geq1$ and $\|f\|_{L^\infty L^2}\leq B$,
			iterating this inequality for $j=1,\ldots,m$ yields
			\begin{align*}
				\|f\|_{L^\infty L^{2^{m+1}}}
				\leq&~\left(2C^2 16^m K\right)^{\frac{1}{2^{m+1}}}\max\{B,\|f\|_{L^{\infty}L^{2^m}},1\}\\
				\leq&~\max\{B,\|f\|_{L^{\infty}L^2},1\}\prod_{j=1}^{m}\left(2C^2K\,16^j\right)^{\frac1{2^{j+1}}}
				\leq 16\sqrt{2}\,CK^{\frac12}B,
			\end{align*}
			where we have used
			\[
			\sum_{j=1}^{\infty}\frac1{2^{j+1}}=\frac12,\qquad
			\sum_{j=1}^{\infty}\frac{j}{2^{j+1}}=1.
			\]
			Letting $m\to\infty$, we conclude that
			\[
			\|f\|_{L^\infty L^\infty}
			\leq CK^{\frac12}\left(\|f\|_{L^\infty L^2}+\|f_{\rm in}\|_{L^\infty}+1\right).
			\]
			This completes the proof.
		\end{proof}

		\begin{revision}
			\section{Local well-posedness and continuation}
			\label{sec:local-continuation}
			
			In this appendix, we establish local well-posedness and the blow-up criterion used in the proof of Theorem~\ref{thm:main}. Throughout this appendix, $A\geq1$ is fixed, and all spatial norms are taken over $\mathbb T\times\mathbb R$. We first recall the following product estimates.
			
			\begin{lemma}[Proposition A.29 in \cite{BV2022}]\label{Hs estimate}
				Let $f$, $g\in H^s(\mathbb{T}\times\mathbb{R})$ for $s>0$. Then
				\begin{align*}
					\|fg\|_{H^s}
					\leq C\left(\|f\|_{H^s}\|g\|_{L^{\infty}}+\|g\|_{H^s}\|f\|_{L^{\infty}}\right).
				\end{align*}
				For $s>1$, there holds
				\begin{align*}
					\|fg\|_{H^s}\leq C\|f\|_{H^s}\|g\|_{H^s}.
				\end{align*}
			\end{lemma}
			\begin{lemma}\label{lem:local-continuation}
				Assume that
				\[
				0\leq n_{\mathrm{in}}\in L^1\cap H^2,\qquad
				0\leq c_{\mathrm{in}}\in H^2\cap W^{1,\infty},\qquad
				u_{\mathrm{in}}\in H^3,\qquad \nabla\cdot u_{\mathrm{in}}=0.
				\]
				For each fixed $A\geq1$, there exist a maximal existence time $T_{\max}\in(0,\infty]$ and a unique strong solution $(n,c,u)$ of \eqref{eq:1.7} on $[0,T_{\max})$. For every $0<T<T_{\max}$, we obtain
				\begin{align*}
					n&\in C([0,T];L^1\cap H^2)\cap L^2(0,T;H^3),\\
					c&\in C([0,T];H^2)\cap L^2(0,T;H^3)\cap L^\infty(0,T;W^{1,\infty}),\\
					u&\in C([0,T];H^3)
					\qquad 
					\nabla\cdot u=0.
					\alabel{app:solution-class}
				\end{align*}
				Moreover, we have
				\[
				n(t),c(t)\geq0,\qquad
				\|n(t)\|_{L^1}=\|n_{\mathrm{in}}\|_{L^1}=:M.
				\]
				If $T_{\max}<\infty$, then
				\begin{align*}
					\limsup_{t\to T_{\max}}
					\left(\|n(t)\|_{L^\infty}+\|u(t)\|_{L^\infty}+\|c(t)\|_{W^{1,\infty}}
					\right)
					=\infty.
					\alabel{app:blowup}
				\end{align*}
			\end{lemma}
			\begin{proof}
				Let $0<T\leq1$, $A\geq1$ and choose the Banach space
				\[
				{\mathcal Y}=L^{\infty}([0,T];H^{\frac32})\times L^{\infty}([0,T];H^2\cap W^{1,\infty})\times L^{\infty}([0,T];H^3_\sigma),\qquad H^3_\sigma=\{v\in H^3:\nabla\cdot v=0\}.
				\]
				We equip $\mathcal{Y}$ with the norm
				\begin{align*}
					\|(n,c,u)\|_{\mathcal{Y}}
					=&\|n\|_{L^\infty(0,T;H^{\frac32})}
					+\|c\|_{L^\infty(0,T;H^2)}
					+\|\nabla c\|_{L^\infty(0,T;L^\infty)}
					+\|u\|_{L^\infty(0,T;H^3)}.
				\end{align*}
				Set
				\[
				R:=2C\left(
				\|n_{\rm in}\|_{H^{\frac{3}{2}}}
				+\|c_{\rm in}\|_{H^2}
				+\|\nabla c_{\rm in}\|_{L^\infty}
				+\|u_{\rm in}\|_{H^3}
				\right),
				\]
				and define the closed ball
				\[
				B_R:=\left\{
				(n,c,u)\in\mathcal Y:
				\|(n,c,u)\|_{\mathcal Y}\leq R
				\right\}.
				\]

				For scalar initial data $f_{\rm in}$, define $S_A(t)$ by $f(t)=S_A(t)f_{\rm in}$, where $f$ solves
				\begin{align*}
					\partial_t f+y\partial_x f-\frac{1}{A}\Delta f=0,\qquad f(0)=f_{\rm in}.
				\end{align*}
				The full velocity equation is
				\begin{align*}
					\partial_tu+y\partial_xu+\begin{pmatrix}u_2\\0\end{pmatrix}-\frac1A\Delta u+\frac1A\nabla p
					=\frac1A\left(\begin{pmatrix}0\\n\end{pmatrix}-u\cdot\nabla u\right),
					\qquad \nabla\cdot u=0.
					\alabel{app:velocity}
				\end{align*}
				Let $\mathbb{S}_A(t)$ be the solution operator for the linear part of \eqref{app:velocity} on divergence-free vector fields, and let $\mathbb P$ denote the Leray projection.
				Following the argument in \cite{BedrossianVicolWang2018}, we obtain the following estimates for the scalar and velocity semigroups, where $r,\sigma\geq0$ and $0<t\leq1$:
				\begin{equation}\label{Property of operator}
					\begin{aligned}
						\|S_A(t)f\|_{H^{r+\sigma}}
						&\leq C A^{\frac{\sigma}{2}} t^{-\frac{\sigma}{2}}\|f\|_{H^r},\\
						\|\nabla S_A(t)f\|_{L^\infty}
						&\leq C A^{\frac12} t^{-\frac12}\|f\|_{L^\infty},\\
						\|\mathbb S_A(t)\mathbb P f\|_{H^{r+\sigma}}
						&\leq C A^{\frac{\sigma}{2}}t^{-\frac{\sigma}{2}}\|f\|_{H^r}.
					\end{aligned}
				\end{equation}
				Using Duhamel's formula, we define $\Phi=(\Phi_n,\Phi_c,\Phi_u)$ on $B_R$ by
				\begin{equation}\label{Phi}
					\begin{aligned}
						\Phi_n(n,c,u)(t)
						={}&S_A(t)n_{\mathrm{in}}
						-\frac1A\int_0^tS_A(t-s)\nabla\cdot(n\nabla c+un)(s)\,ds,\\
						\Phi_c(n,c,u)(t)
						={}&e^{-\frac tA}S_A(t)c_{\mathrm{in}}+\frac1A\int_0^te^{-\frac{t-s}{A}}S_A(t-s)
						\bigl(n-u\cdot\nabla c\bigr)(s)\,ds,\\
						\Phi_u (n,c,u)(t)
						={}&\mathbb S_A(t)u_{\mathrm{in}}+\frac1A\int_0^t\mathbb S_A(t-s)\mathbb P
						\left[\begin{pmatrix}0\\n\end{pmatrix}-u\cdot\nabla u\right](s)\,ds.
					\end{aligned}
				\end{equation}
				\noindent\textbf{Step 1. $\Phi$ maps $B_R$ into itself.}
				
				\noindent\underline{\textbf{I. Estimate of $\Phi_n$.}}
				By Lemma \ref{Hs estimate}, we have
				\begin{align*}
					\|n\nabla c\|_{H^1}
					&\leq C\left(
					\|n\|_{H^1}\|\nabla c\|_{L^{\infty}}+\|\nabla c\|_{H^1}\|n\|_{L^{\infty}}
					\right)\\
					&\leq C\|n\|_{H^{\frac32}}\left(\|\nabla c\|_{L^{\infty}}+\|c\|_{H^2}\right)
					\leq CR^2,\alabel{nc-H1}
				\end{align*}
				and
				\begin{align*}
					\|un\|_{H^1}
					\leq C\left(\|u\|_{H^1}\|n\|_{L^{\infty}}+\|n\|_{H^1}\|u\|_{L^{\infty}}\right)
					\leq C\|u\|_{H^2}\|n\|_{H^{\frac32}}
					\leq CR^2.\alabel{un-H1}
				\end{align*}
				Using $\eqref{Property of operator}_1$, \eqref{nc-H1} and \eqref{un-H1}, we obtain
				\begin{align*}
					\|\Phi_n(n,c,u)\|_{L^\infty(0,T;H^{\frac{3}{2}})}
					&\leq C\|n_{\rm in}\|_{H^{\frac{3}{2}}}
					+CA^{-\frac{1}{4}}
					\sup_{0\leq t\leq T}
					\int_0^t(t-s)^{-\frac{3}{4}}
					\left(
					\|n\nabla c\|_{H^1}+\|un\|_{H^1}
					\right)(s)\,ds\\
					&\leq C\|n_{\rm in}\|_{H^{\frac{3}{2}}}
					+CA^{-\frac{1}{4}}R^2
					\sup_{0\leq t\leq T}
					\int_0^t(t-s)^{-\frac{3}{4}}\,ds\\
					&\leq C\|n_{\rm in}\|_{H^{\frac{3}{2}}}
					+CA^{-\frac{1}{4}}T^{\frac{1}{4}}R^2.\alabel{eq:n-output}
				\end{align*}
				\noindent\underline{\textbf{II. Estimate of $\Phi_c$.}}
				By Lemma \ref{Hs estimate}, we have
				\begin{align*}
					\|u\cdot\nabla c\|_{H^1}
					\leq& C\left(\|u\|_{H^1}\|\nabla c\|_{L^{\infty}}+\|\nabla c\|_{H^1}\|u\|_{L^{\infty}}\right)\\
					\leq& C\|u\|_{H^2}\left(\|\nabla c\|_{L^{\infty}}+\|c\|_{H^2}\right)
					\leq CR^2.\alabel{uc-H1}
				\end{align*}
				By $\eqref{Property of operator}_1$ and \eqref{uc-H1}, we find
				\begin{equation*}
					\begin{aligned}
						\|\Phi_c(n,c,u)(\cdot,t)\|_{H^2}
						&\leq C\|c_{\mathrm{in}}\|_{H^2}
						+CA^{-\frac12}\int_0^t(t-s)^{-\frac12}
						\|n-u\cdot\nabla c\|_{H^1}(s)\,ds\\
						&\leq C\|c_{\mathrm{in}}\|_{H^2}
						+CA^{-\frac12} t^{\frac12}(R+R^2).
					\end{aligned}
				\end{equation*}
				Therefore,
				\begin{equation}\label{eq:c-output}
					\|\Phi_c(n,c,u)\|_{L^\infty H^2}
					\leq C\|c_{\mathrm{in}}\|_{H^2}+CA^{-\frac12}T^{\frac12}(R+R^2).
				\end{equation}
				
				To estimate $\nabla\Phi_c$, we use the commutation identities
				\[
				\partial_xS_A(t)=S_A(t)\partial_x,
				\qquad
				\partial_yS_A(t)=S_A(t)(\partial_y-t\partial_x).
				\]
				By these identities and the $L^\infty$ contraction property of $S_A(t)$, we have
				\[
				\|\nabla S_A(t)c_{\rm in}\|_{L^\infty}
				\leq C\|\nabla c_{\rm in}\|_{L^\infty},
				\qquad 0\leq t\leq1.
				\]
				Using $\eqref{Property of operator}_2$ and the Sobolev embedding, we obtain
				\begin{align*}
					\|\nabla\Phi_c(n,c,u)\|_{L^\infty(0,T;L^\infty)}
					&\leq C\|\nabla c_{\rm in}\|_{L^\infty}
					+CA^{-\frac{1}{2}}
					\sup_{0\leq t\leq T}
					\int_0^t(t-s)^{-\frac{1}{2}}
					\|n-u\cdot\nabla c\|_{L^\infty}(s)\,ds\\
					&\leq C\|\nabla c_{\rm in}\|_{L^\infty}
					+CA^{-\frac{1}{2}}
					\sup_{0\leq t\leq T}
					\int_0^t(t-s)^{-\frac{1}{2}}
					\left(
					\|n(s)\|_{L^\infty}
					+\|u(s)\|_{L^\infty}\|\nabla c(s)\|_{L^\infty}
					\right)\,ds\\
					&\leq C\|\nabla c_{\rm in}\|_{L^\infty}
					+CA^{-\frac{1}{2}}(R+R^2)
					\sup_{0\leq t\leq T}
					\int_0^t(t-s)^{-\frac{1}{2}}\,ds\\
					&\leq C\|\nabla c_{\rm in}\|_{L^\infty}
					+CA^{-\frac{1}{2}}T^{\frac{1}{2}}(R+R^2).\alabel{Phi-c-infty}
				\end{align*}

				\noindent\underline{\textbf{III. Estimate of $\Phi_u$.}}
				The Leray projection satisfies
				\[
				\|\mathbb P f\|_{H^r}\leq C\|f\|_{H^r},
				\qquad r\geq0.
				\]
				By Lemma \ref{Hs estimate}, we have
				\begin{align*}
					\|u\cdot\nabla u\|_{H^2}
					\leq C\|u\|_{H^2}\|\nabla u\|_{H^2}
					\leq C\|u\|_{H^3}^2
					\leq CR^2.\alabel{uu-H2}
				\end{align*}
				Applying the semigroup estimates to $\eqref{Phi}_3$, we obtain
				\begin{align*}
					\|\Phi_u(n,c,u)(\cdot,t)\|_{H^3}
					\leq& C\|u_{\rm in}\|_{H^3}
					+CA^{-\frac14}\int_0^t(t-s)^{-\frac34}
					\|n(s)\|_{H^{\frac32}}\,ds\\
					&+CA^{-\frac12}\int_0^t(t-s)^{-\frac12}
					\|u(s)\cdot\nabla u(s)\|_{H^2}\,ds\\
					\leq& C\|u_{\rm in}\|_{H^3}
					+CA^{-\frac14}R\int_0^t(t-s)^{-\frac34}\,ds
					+CA^{-\frac12}R^2\int_0^t(t-s)^{-\frac12}\,ds\\
					\leq&C\|u_{\rm in}\|_{H^3}
					+CA^{-\frac14}t^{\frac14}(R+R^2).
				\end{align*}
				Taking the supremum over $0\leq t\leq T$, we obtain
				\begin{align*}
					\|\Phi_u(n,c,u)\|_{L^\infty(0,T;H^3)}
					\leq C\|u_{\rm in}\|_{H^3}
					+CA^{-\frac14}T^{\frac14}(R+R^2).\alabel{Phi_u}
				\end{align*}
				Combining \eqref{eq:n-output}, \eqref{eq:c-output} and \eqref{Phi-c-infty} with \eqref{Phi_u}, we conclude
				\begin{align*}
					\|\Phi(n,c,u)\|_{\mathcal{Y}}
					\leq& C\bigg(\|n_{\mathrm{in}}\|_{H^{\frac32}}+\|c_{\mathrm{in}}\|_{H^2}+\|\nabla c_{\mathrm{in}}\|_{L^\infty}+\|u_{\mathrm{in}}\|_{H^3}+A^{-\frac14} T^{\frac14}R^2+A^{-\frac12}T^{\frac12}(R+R^2)\\
					&+A^{-\frac12}T^{\frac12}(R+R^2)+A^{-\frac14}T^{\frac14}(R+R^2)\bigg)\\
					\leq&\frac{1}{2}R+CA^{-\frac14}T^{\frac14}(R^2+R).\alabel{combine}
				\end{align*}
				Then, we take $T>0$ sufficiently small such that
				\begin{equation}\label{tau choose}
					0<T\leq
					\min\left\{
					1,\frac{A}{16C^4(1+R)^4}
					\right\}.
				\end{equation}
				This choice ensures that
				$
				CA^{-\frac14}T^{\frac14}(R^2+R)\leq\frac{R}{2}.
				$
				By \eqref{combine} and \eqref{tau choose}, we have $\|\Phi(n,c,u)\|_{\mathcal{Y}}\leq R$, so $\Phi$ maps $B_R$ into itself.
				
				\noindent\textbf{Step 2. $\Phi$ is a contraction on $B_R$.}
				Let $(n_1,c_1,u_1),(n_2,c_2,u_2)\in B_R$.
				
				\noindent\underline{\textbf{I. Estimate of $\Phi_n$.}}
				The definition of $\Phi_n$ gives
				\begin{align*}
					\Phi_n(n_1,c_1,u_1)-\Phi_n(n_2,c_2,u_2)
					=&-\frac1A\int_0^t S_A(t-s)\nabla\cdot
					\Big[(n_1-n_2)\nabla c_1+n_2\nabla(c_1-c_2)\\
					&+(u_1-u_2)n_1+u_2(n_1-n_2)\Big](s)\,ds.
				\end{align*}
				Using the same estimates in \eqref{nc-H1} and \eqref{un-H1}, we obtain
				\begin{align*}
					\|(n_1-n_2)\nabla c_1\|_{H^1}
					\leq& C\|n_1-n_2\|_{H^{\frac32}}\left(\|\nabla c_1\|_{L^{\infty}}+\|c_1\|_{H^2}\right)
					\leq CR\|n_1-n_2\|_{H^{\frac32}},\\
					\|n_2\nabla(c_1-c_2)\|_{H^1}
					\leq& C\|n_2\|_{H^{\frac32}}\left(\|\nabla (c_1-c_2)\|_{L^{\infty}}+\|c_1-c_2\|_{H^2}\right)\\
					\leq& CR\left(\|\nabla (c_1-c_2)\|_{L^{\infty}}+\|c_1-c_2\|_{H^2}\right),
				\end{align*}
				and
				\begin{align*}
					\|(u_1-u_2)n_1+u_2(n_1-n_2)\|_{H^1}
					\leq& C\big(
					\|u_1-u_2\|_{H^2}\|n_1\|_{H^{\frac32}}
					+\|u_2\|_{H^2}\|n_1-n_2\|_{H^{\frac32}}
					\big)\\
					\leq& CR\big(
					\|u_1-u_2\|_{H^2}+\|n_1-n_2\|_{H^{\frac32}}
					\big).
				\end{align*}
				It follows from $\eqref{Property of operator}_1$ that
				\begin{align*}
					\|\Phi_n(n_1,c_1,u_1)(t)-\Phi_n(n_2,c_2,u_2)(t)\|_{H^{\frac32}}
					\leq& CA^{-\frac14}R
					\|(n_1-n_2,c_1-c_2,u_1-u_2)\|_{\mathcal Y}
					\int_0^t(t-s)^{-\frac34}\,ds\\
					\leq& CA^{-\frac14}t^{\frac14}R
					\|(n_1-n_2,c_1-c_2,u_1-u_2)\|_{\mathcal Y}.
				\end{align*}
				Taking the supremum over $0\leq t\leq T$, we obtain
				\begin{align*}
					\|\Phi_n(n_1,c_1,u_1)-\Phi_n(n_2,c_2,u_2)\|
					_{L^\infty(0,T;H^{\frac32})}
					\leq CA^{-\frac14}T^{\frac14}R
					\|(n_1-n_2,c_1-c_2,u_1-u_2)\|_{\mathcal Y}.\alabel{Phi-n1n2}
				\end{align*}

				\noindent\underline{\textbf{II. Estimate of $\Phi_c$.}}
				Using the estimates in \eqref{uc-H1}, we obtain
				\begin{align*}
					&\|(u_1-u_2)\cdot\nabla c_1
					+u_2\cdot\nabla(c_1-c_2)\|_{H^1}\\
					\leq&C\|u_1-u_2\|_{H^2}\left(\|\nabla c_1\|_{L^{\infty}}+\|c_1\|_{H^2}\right)
					+C\|u_2\|_{H^2}\left(\|\nabla (c_1-c_2)\|_{L^{\infty}}+\|c_1-c_2\|_{H^2}\right)\\
					\leq& CR\|(n_1-n_2,c_1-c_2,u_1-u_2)\|_{\mathcal Y}.
				\end{align*}
				By the Sobolev embedding, we have
				\begin{align*}
					&\|(u_1-u_2)\cdot\nabla c_1
					+u_2\cdot\nabla(c_1-c_2)\|_{L^\infty}\\
					\leq&\|u_1-u_2\|_{L^\infty}\|\nabla c_1\|_{L^\infty}
					+\|u_2\|_{L^\infty}\|\nabla(c_1-c_2)\|_{L^\infty}\\
					\leq& CR\|(n_1-n_2,c_1-c_2,u_1-u_2)\|_{\mathcal Y}.
				\end{align*}
				Applying the semigroup estimates and taking the supremum over $0\leq t\leq T$,
				we obtain
				\begin{align*}
					&\|\Phi_c(n_1,c_1,u_1)-\Phi_c(n_2,c_2,u_2)\|
					_{L^\infty(0,T;H^2)}+\|\nabla[\Phi_c(n_1,c_1,u_1)-\Phi_c(n_2,c_2,u_2)]\|
					_{L^\infty(0,T;L^\infty)}\\
					\leq& CA^{-\frac{1}{2}}(1+R)
					\|(n_1-n_2,c_1-c_2,u_1-u_2)\|_{\mathcal Y}
					\sup_{0\leq t\leq T}\int_0^t(t-s)^{-\frac{1}{2}}\,ds\\
					\leq& CA^{-\frac{1}{2}}T^{\frac{1}{2}}(1+R)
					\|(n_1-n_2,c_1-c_2,u_1-u_2)\|_{\mathcal Y}.\alabel{Phi-c1c2}
				\end{align*}

				\noindent\underline{\textbf{III. Estimate of $\Phi_u$.}}
				Using the same estimates in \eqref{uu-H2}, we obtain
				\begin{align*}
					\|(u_1-u_2)\cdot\nabla u_1
					+u_2\cdot\nabla(u_1-u_2)\|_{H^2}
					\leq& C\big(\|u_1\|_{H^3}+\|u_2\|_{H^3}\big)
					\|u_1-u_2\|_{H^3}\\
					\leq& CR\|u_1-u_2\|_{H^3}.
				\end{align*}
				Applying the semigroup estimates 
				and taking the supremum over $0\leq t\leq T$, we obtain
				\begin{align*}
					&\|\Phi_u(n_1,c_1,u_1)-\Phi_u(n_2,c_2,u_2)\|
					_{L^\infty(0,T;H^3)}\\
					\leq& CA^{-\frac{1}{4}}
					\|n_1-n_2\|_{L^\infty(0,T;H^{\frac{3}{2}})}
					\sup_{0\leq t\leq T}\int_0^t(t-s)^{-\frac{3}{4}}\,ds\\
					&+CA^{-\frac{1}{2}}R
					\|u_1-u_2\|_{L^\infty(0,T;H^3)}
					\sup_{0\leq t\leq T}\int_0^t(t-s)^{-\frac{1}{2}}\,ds\\
					\leq& C\big(
					A^{-\frac{1}{4}}T^{\frac{1}{4}}
					+A^{-\frac{1}{2}}T^{\frac{1}{2}}R
					\big)
					\|(n_1-n_2,c_1-c_2,u_1-u_2)\|_{\mathcal Y}\\
					\leq& CA^{-\frac{1}{4}}T^{\frac{1}{4}}(1+R)
					\|(n_1-n_2,c_1-c_2,u_1-u_2)\|_{\mathcal Y}.\alabel{Phi_u1u2}
				\end{align*}
				By \eqref{Phi-n1n2}--\eqref{Phi_u1u2} and the choice of \(T\) in \eqref{tau choose}, we have
				\begin{align*}
					&\|\Phi(n_1,c_1,u_1)-\Phi(n_2,c_2,u_2)\|_{\mathcal Y}\\
					\leq& CA^{-\frac{1}{4}}T^{\frac{1}{4}}(1+R)
					\|(n_1-n_2,c_1-c_2,u_1-u_2)\|_{\mathcal Y}\\
					\leq& \frac{1}{2}
					\|(n_1-n_2,c_1-c_2,u_1-u_2)\|_{\mathcal Y}.
				\end{align*}
				
				By the Banach fixed point theorem, $\Phi$ has a unique
				fixed point $(n,c,u)\in B_R$.
				Since $n_{\rm in},c_{\rm in}\geq0$, negative-part estimates
				for the density and chemical equations give $n,c\geq0$.
				Using periodicity in $x$ and a cutoff argument in $y$,
				we obtain
				\[
				\|n(t)\|_{L^1}
				=\int_{\mathbb T\times\mathbb R}n(t,x,y)\,dx\,dy
				=\int_{\mathbb T\times\mathbb R}n_{\rm in}(x,y)\,dx\,dy
				=\|n_{\rm in}\|_{L^1},
				\qquad 0\leq t\leq T.
				\]

				\noindent\textbf{Step 3. Higher regularity of $c$.}
				The following energy estimates are justified by regularization. Taking the $L^2$ inner product of $\eqref{eq:1.7}_2$ with $c$ over $\mathbb T\times\mathbb R$, by $\nabla\cdot u=0$ and Young's inequality, we obtain
				\begin{equation}
					\frac{d}{dt}\|c\|_{L^2}^2
					+\frac{2}{A}\|\nabla c\|_{L^2}^2
					+\frac{1}{A}\|c\|_{L^2}^2
					\leq\frac{1}{A}\|n\|_{L^2}^2.
					\label{app:c-l2}
				\end{equation}
				
				Applying $\nabla$ to $\eqref{eq:1.7}_2$ and taking the $L^2$ inner product with $\nabla c$ over $\mathbb T\times\mathbb R$, by integration by parts and Young's inequality, we obtain,
				\begin{align*}
					\frac{1}{2}\frac{d}{dt}\|\nabla c\|_{L^2}^2
					+\frac{1}{A}\|\nabla^2c\|_{L^2}^2
					+\frac{1}{A}\|\nabla c\|_{L^2}^2
					=&-\int_{\mathbb{T}\times\mathbb{R}}\partial_xc\,\partial_ycdxdy
					+\frac{1}{A}\int_{\mathbb{T}\times\mathbb{R}}\left[\nabla n\cdot\nabla c-\nabla(u\cdot\nabla c)\cdot\nabla c\right]dxdy\\
					\leq&\|\nabla c\|_{L^2}^2+\frac{1}{A}\|\nabla n\|_{L^2}\|\nabla c\|_{L^2}+\frac{1}{A}\int_{\mathbb{T}\times\mathbb{R}}(u\cdot\nabla c)\Delta cdxdy\\
					\leq&\frac{1}{2A}\|\nabla n\|_{L^2}^2+(\frac{1}{2A}+1)\|\nabla c\|_{L^2}^2+\frac{1}{2A}\|\Delta c\|_{L^2}^2+\frac{C}{A}\|u\|_{H^2}^2\|\nabla c\|_{L^2}^2.
				\end{align*}
				Using $\|\nabla^2c\|_{L^2}=\|\Delta c\|_{L^2}$, we have
				\begin{equation}
					\begin{aligned}
						\frac{d}{dt}\|\nabla c\|_{L^2}^2
						+\frac{1}{A}\|\Delta c\|_{L^2}^2
						+\frac{1}{A}\|\nabla c\|_{L^2}^2
						\leq& C\left(1+\frac{1}{A}\|u\|_{H^2}^2\right)
						\|\nabla c\|_{L^2}^2
						+\frac{C}{A}\|\nabla n\|_{L^2}^2.
					\end{aligned}
					\label{app:c-h1}
				\end{equation}
				
				Using
				\[
				\Delta(y\partial_xc)
				=y\partial_x\Delta c+2\partial_x\partial_yc,
				\]
				we apply $\Delta$ to $\eqref{eq:1.7}_2$ and take the $L^2$ inner product with $\Delta c$ over $\mathbb T\times\mathbb R$. By integration by parts and Young's inequality, we obtain
				
				\begin{equation}
					\begin{aligned}
						&\frac{1}{2}\frac{d}{dt}\|\Delta c\|_{L^2}^2
						+\frac{1}{A}\|\nabla\Delta c\|_{L^2}^2
						+\frac{1}{A}\|\Delta c\|_{L^2}^2\\
						=&-2\int_{\mathbb{T}\times\mathbb{R}}\partial_x\partial_yc\,\Delta cdxdy
						-\frac{1}{A}\int_{\mathbb{T}\times\mathbb{R}}\nabla n\cdot\nabla\Delta cdxdy
						+\frac{1}{A}\int_{\mathbb{T}\times\mathbb{R}}\nabla(u\cdot\nabla c)\cdot\nabla\Delta cdxdy\\
						\leq&C\|\Delta c\|_{L^2}^2+\frac{1}{2A}\|\nabla\Delta c\|_{L^2}^2+\frac{C}{A}\|\nabla n\|_{L^2}^2+\frac{C}{A}\|u\|_{H^2}^2\|c\|_{H^2}^2,
					\end{aligned}
					\label{the estimate about delta c}
				\end{equation}
				where we have used 
				\begin{align*}
					\|\nabla(u\cdot\nabla c)\|_{L^2}
					&\leq C\big(
					\|\nabla u\|_{L^4}\|\nabla c\|_{L^4}
					+\|u\|_{L^\infty}\|\nabla^2c\|_{L^2}
					\big)\\
					&\leq C\|u\|_{H^2}\|c\|_{H^2}.
				\end{align*}
				Rearranging \eqref{the estimate about delta c}, we obtain
				\begin{align*}
					\frac{d}{dt}\|\Delta c\|_{L^2}^2+\frac{1}{A}\|\nabla\Delta c\|_{L^2}^2+\frac{2}{A}\|\Delta c\|_{L^2}^2
					\leq C\|\Delta c\|_{L^2}^2+\frac{C}{A}\big(\|\nabla n\|_{L^2}^2+\|u\|_{H^2}^2\|c\|_{H^2}^2\big).\alabel{app:c-second}
				\end{align*}
				Combining \eqref{app:c-l2}, \eqref{app:c-h1}, and \eqref{app:c-second}, we obtain
				\begin{equation*}
					\frac{d}{dt}\|c\|_{H^2}^2
					+\frac{1}{A}\|c\|_{H^3}^2
					\leq C\|c\|_{H^2}^2
					+\frac{C}{A}\big(
					\|n\|_{H^1}^2+\|u\|_{H^2}^2\|c\|_{H^2}^2
					\big).\alabel{c-H2}
				\end{equation*}
				Integrating over $[0,T]$ and using $(n,c,u)\in B_R$, we obtain
				\begin{equation}
					\frac{1}{A}\int_0^T\|c(t)\|_{H^3}^2\,dt
					\leq C\|c_{\rm in}\|_{H^2}^2
					+CTR^2+\frac{CT}{A}(R^2+R^4).
					\label{app:c-h3-integral}
				\end{equation}
				Therefore, $c\in L^2(0,T;H^3)$.

				\noindent\textbf{Step 4. Higher regularity of $n$.}
				By Lemma \ref{Hs estimate}, we have
				\begin{align*}
					\|n\nabla c\|_{H^2}
					\leq& C\|n\|_{H^2}\|\nabla c\|_{H^2}
					\leq C\|n\|_{H^2}\|c\|_{H^3},\\
					\|un\|_{H^2}
					\leq& C\|u\|_{H^2}\|n\|_{H^2}.
				\end{align*}
				Using the same argument as for \eqref{c-H2}, we obtain
				\begin{align*}
					\frac{1}{2}\frac{d}{dt}\|n\|_{H^2}^2
					+\frac{1}{A}\|\nabla n\|_{H^2}^2
					\leq& C\|n\|_{H^2}^2
					+\frac{C}{A}\big(
					\|n\nabla c\|_{H^2}+\|un\|_{H^2}
					\big)\|\nabla n\|_{H^2}\\
					\leq& C\|n\|_{H^2}^2
					+\frac{C}{A}\big(
					\|c\|_{H^3}+\|u\|_{H^2}
					\big)\|n\|_{H^2}\|\nabla n\|_{H^2}.
				\end{align*}
				By Young's inequality, we have
				\begin{equation}\label{n-H2 estimate}
					\frac{d}{dt}\|n\|_{H^2}^2
					+\frac{1}{A}\|\nabla n\|_{H^2}^2
					\leq C\left(
					1+\frac{1}{A}\|u\|_{H^2}^2
					+\frac{1}{A}\|c\|_{H^3}^2
					\right)\|n\|_{H^2}^2.
				\end{equation}
				Applying \eqref{app:c-h3-integral} and Gronwall's inequality, we obtain
				\begin{align*}
					\|n\|_{L^\infty(0,T;H^2)}^2
					&\leq \|n_{\rm in}\|_{H^2}^2
					\exp\left\{
					CT+\frac{C}{A}\int_0^T
					\big(\|u(s)\|_{H^2}^2+\|c(s)\|_{H^3}^2\big)\,ds
					\right\}\\
					&\leq \|n_{\rm in}\|_{H^2}^2
					\exp\left\{
					C\|c_{\rm in}\|_{H^2}^2
					+CT(1+R^2)+\frac{CT}{A}(R^2+R^4)
					\right\}.
				\end{align*}
				Integrating \eqref{n-H2 estimate} over $[0,T]$, we also obtain
				$n\in L^2(0,T;H^3)$.

				\noindent\textbf{Step 5. Continuity in time.}
				The preceding product estimates imply
				\begin{align*}
					\|n\nabla c+un\|_{L^2(0,T;H^2)}
					&\leq C\|n\|_{L^\infty(0,T;H^2)}
					\left(
					\|c\|_{L^2(0,T;H^3)}
					+T^{\frac{1}{2}}\|u\|_{L^\infty(0,T;H^2)}
					\right)<\infty.
				\end{align*}
				
				Fix $t_0\in(0,T]$ and $\varepsilon>0$.
				For any $a\in[0,t_0)$ with $n(a)\in H^2$
				and $t\in[a,T]$, the Duhamel formula gives
				\[
				n(t)-S_A(t-a)n(a)
				=-\frac{1}{A}\int_a^t
				S_A(t-s)\nabla\cdot(n\nabla c+un)(s)\,ds
				=:W_a(t).
				\]
				The $H^2$ energy estimate for $W_a$ gives
				\[
				\frac{d}{dt}\|W_a\|_{H^2}^2
				+\frac{1}{A}\|\nabla W_a\|_{H^2}^2
				\leq C\|W_a\|_{H^2}^2
				+\frac{C}{A}\|n\nabla c+un\|_{H^2}^2.
				\]
				Since $W_a(a)=0$ and $T\leq1$, Gronwall's inequality implies
				\[
				\|W_a(t)\|_{H^2}
				\leq CA^{-\frac{1}{2}}
				\left(
				\int_a^t\|n\nabla c+un\|_{H^2}^2(s)\,ds
				\right)^{\frac{1}{2}}.
				\]
				For $t\in[a,T]$, we obtain
				\begin{align*}
					\|n(t)-n(t_0)\|_{H^2}
					&\leq
					\|\big(S_A(t-a)-S_A(t_0-a)\big)n(a)\|_{H^2}
					+\|W_a(t)\|_{H^2}+\|W_a(t_0)\|_{H^2}\\
					&\leq
					\|\big(S_A(t-a)-S_A(t_0-a)\big)n(a)\|_{H^2}\\
					&\quad+CA^{-\frac{1}{2}}
					\left(
					\int_a^{\max\{t,t_0\}}
					\|n\nabla c+un\|_{H^2}^2(s)\,ds
					\right)^{\frac{1}{2}}.
				\end{align*}
				By the integrability of $\|n\nabla c+un\|_{H^2}^2$,
				we choose $a\in(0,t_0)$ with $n(a)\in H^2$
				sufficiently close to $t_0$ such that
				\[
				CA^{-\frac{1}{2}}
				\left(
				\int_a^{t_0}\|n\nabla c+un\|_{H^2}^2(s)\,ds
				\right)^{\frac{1}{2}}
				<\frac{\varepsilon}{4}.
				\]
				
				With this $a$ fixed, by the strong continuity of $S_A(t)$
				and the absolute continuity of the integral, we can choose
				$\delta\in(0,t_0-a)$ such that
				\[
				\begin{aligned}
					\|\big(S_A(t-a)-S_A(t_0-a)\big)n(a)\|_{H^2}
					&<\frac{\varepsilon}{2},\\
					CA^{-\frac{1}{2}}
					\left(
					\int_a^{\max\{t,t_0\}}
					\|n\nabla c+un\|_{H^2}^2(s)\,ds
					\right)^{\frac{1}{2}}
					&<\frac{\varepsilon}{2}
				\end{aligned}
				\]
				whenever $t\in[0,T]$ and $|t-t_0|<\delta$.
				Using the preceding difference estimate, we obtain
				\[
				\|n(t)-n(t_0)\|_{H^2}<\varepsilon,
				\]
				which proves continuity at every $t_0\in(0,T]$.
				
				At $t_0=0$, let $\varepsilon>0$.
				The strong continuity of $S_A(t)$ and the integrability of
				$\|n\nabla c+un\|_{H^2}^2$ allow us to choose $\delta>0$
				such that, for $t\in[0,T]$ with $t<\delta$,
				\begin{align*}
					\|n(t)-n_{\rm in}\|_{H^2}
					\leq& \|S_A(t)n_{\rm in}-n_{\rm in}\|_{H^2}
					+CA^{-\frac{1}{2}}
					\left(
					\int_0^t\|n\nabla c+un\|_{H^2}^2(s)\,ds
					\right)^{\frac{1}{2}}\\
					<&\frac{\varepsilon}{2}+\frac{\varepsilon}{2}
					=\varepsilon.
				\end{align*}
				Thus, $n\in C([0,T];H^2)$. By nonnegativity, mass conservation and dominated convergence, we also have
				\[
				\begin{aligned}
					\|n(t)-n(t_0)\|_{L^1}
					&=2\|n_{\rm in}\|_{L^1}
					-2\int_{\mathbb T\times\mathbb R}
					\min\{n(t,x,y),n(t_0,x,y)\}\,dx\,dy
					\longrightarrow0
				\end{aligned}
				\]
				as $t\to t_0$ within $[0,T]$.
				Hence, $n\in C([0,T];L^1\cap H^2)$.
				
				Moreover, we know
				\[
				n-u\cdot\nabla c\in L^\infty(0,T;H^1),
				\qquad
				\mathbb P\left(
				\begin{pmatrix}0\\n\end{pmatrix}-u\cdot\nabla u
				\right)\in L^\infty(0,T;H^2).
				\]
				Using the corresponding Duhamel formulas, the strong continuity of the semigroups and the integrable kernel $A^{-\frac{1}{2}}(t-s)^{-\frac{1}{2}}$, we obtain by the same argument
				\[
				c\in C([0,T];H^2),
				\qquad
				u\in C([0,T];H^3).
				\]
				
				\noindent\textbf{Step 6. Blow-up criterion.}
				Suppose that $T_{\max}<\infty$ and
				\[
				\sup_{0\leq t<T_{\max}}
				\left(
				\|n(t)\|_{L^\infty}
				+\|u(t)\|_{L^\infty}
				+\|c(t)\|_{W^{1,\infty}}
				\right)\leq K
				\]
				for some $K\geq1$.
				Using $u\cdot\nabla u=\nabla\cdot(u\otimes u)$
				and the product estimates, we obtain
				\begin{align*}
					\frac{d}{dt}\|n\|_{H^2}^2
					+\frac{1}{A}\|\nabla n\|_{H^2}^2
					&\leq C\left(1+\frac{K^2}{A}\right)\|n\|_{H^2}^2
					+\frac{CK^2}{A}
					\left(\|c\|_{H^3}^2+\|u\|_{H^3}^2\right),\\
					\frac{d}{dt}\|c\|_{H^2}^2
					+\frac{1}{A}\|c\|_{H^3}^2
					&\leq C\left(1+\frac{K^2}{A}\right)\|c\|_{H^2}^2
					+\frac{C}{A}\|n\|_{H^2}^2
					+\frac{CK^2}{A}\|u\|_{H^3}^2,\\
					\frac{d}{dt}\|u\|_{H^3}^2
					+\frac{1}{A}\|\nabla u\|_{H^3}^2
					&\leq C\left(1+\frac{K^2}{A}\right)\|u\|_{H^3}^2
					+\frac{C}{A}\|n\|_{H^2}^2.
				\end{align*}
				Here $C$ is independent of $A$ and $T$. Multiplying the last two inequalities by $C_0K^2$
				for a sufficiently large constant $C_0$, and adding them to the first, we absorb $\frac{CK^2}{A}\|c\|_{H^3}^2$ and obtain
				\begin{align*}
					\frac{d}{dt}
					\left[
					\|n\|_{H^2}^2
					+C_0K^2\left(\|c\|_{H^2}^2+\|u\|_{H^3}^2\right)
					\right]
					\leq C\left(1+\frac{K^2}{A}\right)
					\left[
					\|n\|_{H^2}^2+C_0K^2\left(\|c\|_{H^2}^2+\|u\|_{H^3}^2\right)
					\right].
				\end{align*}
				By Gronwall's inequality and $T_{\max}<\infty$, we have
				\[
				\sup_{0\leq t<T_{\max}}
				\left(
				\|n(t)\|_{H^2}^2
				+\|c(t)\|_{H^2}^2
				+\|u(t)\|_{H^3}^2
				\right)<\infty.
				\]
				By this bound, $\|n(t)\|_{L^1}=\|n_{\rm in}\|_{L^1}$ and
				$\|c(t)\|_{W^{1,\infty}}\leq K$, the local existence
				argument gives a positive lifespan independent of
				the restarting time.
				Restarting sufficiently close to $T_{\max}$ and using
				uniqueness, we extend the solution beyond $T_{\max}$,
				contradicting its maximality.
				This proves \eqref{app:blowup}.

			\end{proof}

			\end{revision}

			\section*{Acknowledgement}
			S. Cui is partially supported by the National Postdoctoral Researcher Support Program   under Grant  GZB20260732 and the China Postdoctoral Science Foundation  under Grant  2026M793415. W. Wang is partially supported by the National Key R\&D Program of China under Grant 2023YFA1009200 and the National Natural Science Foundation of China under Grant 12471219.
			
			\section*{Declaration of generative AI use}
			
			During the preparation of this manuscript, the authors used ChatGPT (OpenAI) for language polishing and occasional assistance in discussing and checking intermediate calculations. All research ideas, mathematical arguments, proofs, and conclusions were developed and independently verified by the authors. The authors take full responsibility for the content of the manuscript.

			\section*{Declaration of competing interest}
			The authors declare that they have no known competing financial interests
			or personal relationships that could have appeared to influence the work reported in this paper.


\begin{thebibliography}{99}	
				\bibitem{Alikakos1979}
				N.~D. Alikakos, $L^p$ bounds of solutions of reaction-diffusion equations, Comm. Partial Differential Equations {4} (1979), 827--868.
				
				\bibitem{BGM2017}
				Bedrossian J., Germain P. and Masmoudi N. (2017). On the stability threshold for the 3D Couette flow in Sobolev regularity. Annals of Mathematics, 185, 541--608.
				
				\bibitem{BedrossianHe2017}
				J.~Bedrossian and S.~He, Suppression of blow-up in Patlak--Keller--Segel via shear flows, SIAM J. Math. Anal. {49} (2017), 4722--4766.
				
				\bibitem{BedrossianMasmoudiVicol2016}
				J.~Bedrossian, N.~Masmoudi and V.~Vicol, Enhanced dissipation and inviscid damping in the inviscid limit of the Navier--Stokes equations near the two dimensional Couette flow, Arch. Ration. Mech. Anal. {219} (2016), 1087--1159.
				
				\bibitem{BedrossianVicolWang2018}
				J.~Bedrossian, V.~Vicol and F.~Wang, The Sobolev stability threshold for 2D shear flows near Couette, J. Nonlinear Sci. {28} (2018), 2051--2075.
				
				\bibitem{BV2022}
				J. Bedrossian and V. Vicol, The mathematical analysis of the incompressible Euler and Navier--Stokes equations: an introduction, Graduate Studies in Mathematics, vol. 225, American Mathematical Society, Providence, RI, 2022.
				
				\bibitem{CalvezCorrias2008}
				V.~Calvez and L.~Corrias, The parabolic--parabolic Keller--Segel model in $\mathbb R^2$, Commun. Math. Sci. {6} (2008), 417--447.
				
				\bibitem{CuiWang2024}
				S.~Cui and W.~Wang, Suppression of blow-up in multispecies Patlak--Keller--Segel--Navier--Stokes system via the Poiseuille flow in a finite channel, SIAM J. Math. Anal. {56} (2024), 7683--7712.
				
				\bibitem{CuiWangWang2026}
				S.~Cui, L.~Wang and W.~Wang, Suppression of blow-up for the 3D Patlak--Keller--Segel--Navier--Stokes system via the Couette flow, J. Math. Pures Appl. {210} (2026), 103874.
				
				\bibitem{DengShiWang2025}
				S. Deng, B. Shi and W. Wang, Suppression of blow-up in 3-D Keller--Segel model via the Couette flow in whole space, J. Differential Equations {432} (2025), 113265.
				
				
				\bibitem{FengShiWang2022}
				Y. Feng, B. Shi and W. Wang, Dissipation enhancement of planar helical flows and applications to three-dimensional Kuramoto--Sivashinsky and Keller--Segel equations, J. Differential Equations {313} (2022), 420--449.
				
				\bibitem{He2018}
				S.~He, Suppression of blow-up in parabolic--parabolic Patlak--Keller--Segel via strictly monotone shear flows, Nonlinearity {31} (2018), 3651--3688.
				
				\bibitem{He2023}
				S.~He, Enhanced dissipation and blow-up suppression in a chemotaxis-fluid system, SIAM J. Math. Anal. {55} (2023), 2615--2643.
				
				\bibitem{He2025Alternating}
				S.~He, Time-dependent flows and their applications in parabolic-parabolic Patlak-Keller-Segel systems part I: alternating flows, J. Funct. Anal. {288} (2025), 110786.
				
				\bibitem{He2025Shear}
				S.~He, Time-dependent shear flows and their applications in parabolic--parabolic Patlak--Keller--Segel systems, Nonlinearity {38} (2025), 035029.
				
				\bibitem{Hu2024}
				Z.~Hu, Suppression of chemotactic singularity via viscous flow with large buoyancy, SIAM J. Math. Anal. {56} (2024), 7866--7902.
				
				\bibitem{HuKiselevYao2025}
				Z.~Hu, A.~Kiselev and Y.~Yao, Suppression of chemotactic singularity by buoyancy, Geom. Funct. Anal. {35} (2025), 812--841.
				
				\bibitem{KellerSegel1970}
				E.~F. Keller and L.~A. Segel, Initiation of slime mold aggregation viewed as an instability, J. Theoret. Biol. {26} (1970), 399--415.
				
				\bibitem{KiselevXu2016}
				A.~Kiselev and X.~Xu, Suppression of chemotactic explosion by mixing, Arch. Ration. Mech. Anal. {222} (2016), 1077--1112.
				
				\bibitem{LiXiangXu2025}
				H.~Li, Z.~Xiang and X.~Xu, Suppression of blow-up in Patlak--Keller--Segel--Navier--Stokes system via the Poiseuille flow, J. Differential Equations {434} (2025), 113301.
				
				\bibitem{LZ2024}
				Li H. and Zhao W. (2024). Asymptotic stability in the critical space of 2D monotone shear flow in the viscous fluid. Communications in Mathematical Physics, 405, 267.
				
				\bibitem{Patlak1953}
				C.~S. Patlak, Random walk with persistence and external bias, Bull. Math. Biophys. {15} (1953), 311--338.
				
				\bibitem{ShiWang2024}
				B. Shi and W. Wang, Enhanced dissipation and blow-up suppression for the three
				dimensional Keller--Segel equation with the plane Couette--Poiseuille flow, J. Differential Equations 403 (2024), 368--405.
				
				\bibitem{Schweyer2014}
				R.~Schweyer, Stable blow-up dynamic for the parabolic-parabolic Patlak-Keller-Segel model, arXiv:1403.4975, 2014.
				
				\bibitem{WangWangWu2026}
				G. Wang, W. Wang and T. Wu, Suppression of blow-up in Patlak--Keller--Segel--Navier--Stokes system via the Couette flow in whole space,
				Adv. Nonlinear Anal. 15 (2026), 20250148.
				
				\bibitem{WeiZhang2021}
				D.~Wei and Z.~Zhang, Transition threshold for the 3D Couette flow in Sobolev space, Comm. Pure Appl. Math. {74} (2021), 2398--2479.
				
				\bibitem{WeiZhang2023}
				D.~Wei and Z.~Zhang, Nonlinear enhanced dissipation and inviscid damping for the 2D Couette flow, Tunis. J. Math. {5} (2023), 573--592.
				
				\bibitem{ZengZhangZi2021}
				L.~Zeng, Z.~Zhang and R.~Zi, Suppression of blow-up in Patlak--Keller--Segel--Navier--Stokes system via the Couette flow, J. Funct. Anal. {280} (2021), 108967.
				
				
				
				
				
				
				
				
			\end{thebibliography}
		\end{document}